\documentclass[]{amsart} 
\usepackage[final]{showkeys}       
\usepackage[breaklinks]{hyperref}
\usepackage{amssymb, tabularx}
\usepackage{tikz, float}
\usetikzlibrary{calc,decorations.pathreplacing,matrix,arrows,positioning}

\newcommand{\NN}{\mathbb{N}}
\newcommand{\ZZ}{\mathbb{Z}}

\newcommand{\V}{\mathcal{V}}
\newcommand{\A}{\mathbb{A}}

\renewcommand{\phi}{\varphi}
\renewcommand{\epsilon}{\varepsilon}

\theoremstyle{plain} \newtheorem{thm}{Theorem}[section]

\newtheorem{lem}[thm]{Lemma}
\newtheorem{cor}[thm]{Corollary}
\theoremstyle{definition} \newtheorem{defn}[thm]{Definition}
\theoremstyle{definition} \newtheorem*{defn*}{Definition}
\theoremstyle{remark} 
\theoremstyle{plain} \newtheorem*{claim}{Claim}
\newenvironment{claimproof}
  {\begin{proof}[Proof of claim]  }
  {\end{proof}}

\newcommand{\mat}[1]{\left( \begin{matrix} #1 \end{matrix} \right)}
\newcommand{\Case}[1]{\smallskip \noindent \textbf{Case #1:}}
\newcommand{\m}[1]{\mathbb{#1}}   

\numberwithin{equation}{section}  

\newcommand{\B}{\mathbb{B}}
\newcommand{\C}{\mathbb{C}}
\newcommand{\D}{\mathbb{D}}
\renewcommand{\S}{\mathbb{S}}
\newcommand{\W}{\mathbb{W}}

\newcommand{\M}{\mathcal{M}}
\newcommand{\T}{\mathcal{T}}
\newcommand{\0}{\mathbf{0}}
\newcommand{\id}{\text{id}}

\newcommand{\HS}{\textbf{HS}}

\newcommand{\Cg}{\text{Cg}}
\newcommand{\Mod}{\text{Mod}}
\newcommand{\Con}{\text{Con}}
\newcommand{\Pol}{\text{Pol}}
\newcommand{\Range}{\text{Range}}
\newcommand{\Supp}{\text{Supp}}

\newcommand{\Chain}[1]{\stackrel{\underline{#1}}{\ }}

\title{The Undecidability of the Definability of Principal Subcongruences}
\author{Matthew Moore}
\date{\today}
\address{
  Vanderbilt University;
  Nashville, TN 37240;
  U.S.A.}
\email{matthew.moore@vanderbilt.edu}

\begin{document} \maketitle
\begin{abstract}
For each Turing machine $\T$, we construct an algebra $\A'(\T)$ such that
the variety generated by $\A'(\T)$ has definable principal subcongruences if
and only if $\T$ halts, thus proving that the property of having definable
principal subcongruences is undecidable for a finite algebra. A consequence
of this is that there is no algorithm that takes as input a finite algebra a
decides whether that algebra is finitely based.
\end{abstract} 

\section{Introduction}
Given a variety $\V$, the \emph{residual bound} of $\V$ is the least
cardinal $\lambda$ that is strictly larger than the cardinality of every
subdirectly irreducible (=SI) member of $\V$. If such a $\lambda$ exists we
write $\kappa(\V) = \lambda$, and if no such $\lambda$ exists, then we write
$\kappa(\V) = \infty$. If $\V = \V(\A)$ is the variety generated by the
algebra $\A$, we define $\kappa(\A) = \kappa(\V(\A))$. The
\emph{RS-conjecture} (see \cite{HobbyMcKenzieStructureFiniteAlgebras}) is
the conjecture that $\kappa(\A)\geq \omega$ implies $\kappa(\A) = \infty$
for finite $\A$. McKenzie \cite{McKenzieResidualBounds} disproves this
conjecture by exhibiting an algebra with residual bound precisely $\omega$.
This algebra is used by McKenzie as a basis for his groundbreaking paper
\cite{McKenzieResidualBoundNotComp}, in which to each Turing machine $\T$ an
algebra $\A(\T)$ is associated such that $\kappa(\A(\T)) < \omega$ if and
only if $\T$ halts, thus proving that the property of having a finite
residual bound is an undecidable property of a finite algebra.

The problem of algorithmically determining whether $\kappa(\A)<\omega$ is
closely related to a problem due to Alfred Tarski, called \emph{Tarski's
finite basis problem}. An algebra $\A$ is said to be \emph{finitely based}
if the infinite set of identities which are true in $\A$ can be derived from
a finite subset of them. Tarski's problem is the question: is there an
algorithm that takes as input a finite algebra and determines whether it is
finitely based? McKenzie \cite{McKenzieTarskisFiniteBasis} uses a
construction similar to his $\A(\T)$ construction to provide a negative
answer to this question, and Willard \cite{WillardTarskisFiniteBasis} shows
that, in fact, the original $\A(\T)$ is finitely based if and only if $\T$
halts. Thus, there is no algorithm that decides whether an algebra is
finitely based for every finite algebra.

An algebra $\A$ is finitely based if and only if $\V(\A)$ (the variety
generated by $\A$) is finitely axiomatizable. One approach to proving that
$\V(\A)$ is finitely axiomatizable is to first show that $\kappa(\A) <
\omega$, and then to show that $\V(\A)$ has definable principal congruences.
These two features are sufficient to imply that $\V(\A)$ is finitely
axiomatizable. A formula $\psi(w,x,y,z)$ defined in an algebraic first-order
language $L$ is said to be a \emph{congruence formula} if $\psi$ is
existential positive and for every model $\A$ of $L$ and for all $a,b,c,d\in
A$, $\A\models \psi(a,b,c,d)$ implies that $(a,b)$ belongs to the congruence
generated by the pair $(c,d)$. A class $\mathcal{C}$ of algebras of the same
type is said to have \emph{definable principal congruences (DPC)} if there
is a congruence formula $\psi$ such that for every $\A\in \mathcal{C}$ and
every $a,b\in A$ the formula $\psi(x,y,a,b)$ defines the relation
``$(x,y)\in \Con^{\B}(a,b)$''. McKenzie \cite{McKenzieParaprimalVarieties}
shows that if a variety $\V$ has definable principal congruences and
$\kappa(\V) < \omega$, then $\V$ is finitely based.

Baker and Wang \cite{BakerWangDPSC} generalize DPC by saying that a class of
algebras $\mathcal{C}$ (all of the same type) has \emph{definable principal
subcongruences (DPSC)} if there are congruence formulas $\Gamma$ and $\psi$
such that for all $\A\in \mathcal{C}$ and all $a,b\in A$, if $a\neq b$ then
there exist $c,d\in A$ such that $c\neq d$, $\A\models \Gamma(c,d,a,b)$, and
$\psi(-,-,c,d)$ defines $\Cg^\A(c,d)$. It is convenient to observe that if
the type of $\mathcal{C}$ is finite, then there is a first-order formula,
$\Pi_\psi(y,z)$ so that $\A\models \Pi_{\psi}(c,d)$ (where $\A$ is any
algebra of this type) if and only if $\{(a,b) \mid \A\models \psi(a,b,c,d)
\}$ is the congruence generated by $(c,d)$. In symbols, a class
$\mathcal{C}$ of algebras of the same finite type has DPSC if there are
congruence formulas $\Gamma(w,x,y,z)$ and $\psi(w,x,y,z)$ such that for all
$\A\in \mathcal{C}$, 
\[
  \A\models \forall a,b \left[ a\neq b \rightarrow \exists c,d \left[ c\neq
    d \wedge \Gamma(c,d,a,b) \wedge \Pi_\psi(c,d) \right] \right].
\]

\newcolumntype{A}{>{\centering\arraybackslash} m{.45\linewidth}}
\newcolumntype{B}{>{\centering\arraybackslash} m{.45\linewidth}}
\begin{figure}[h] \begin{tabular}{AB}
  \begin{tikzpicture}[font=\scriptsize,smooth]
    \node (ab) {$\forall \; \Cg(a,b)$};
    \node (cd) [below=of ab,xshift=-2em] {$\forall \; \Cg(c,d)$};
    \coordinate (t1) at (ab|-cd);
    \coordinate (t1) at ($(t1) !0.5! (ab)$);
    \draw (t1) ellipse (6em and 5em)
      (ab) to[out=225,in=90] node[right]{$\A\models\psi(c,d,a,b)$} (cd);
    \node at (current bounding box.north west) {$\Con(\A)$};
  \end{tikzpicture}
  &
  \begin{tikzpicture}[font=\scriptsize,smooth]
    \node (ab) {$\forall \; \Cg(a,b)$};
    \node (cd) [below=of ab,xshift=-2em] {$\exists \; \Cg(c,d)$};
    \node (rs) [below=of cd,xshift=2em] {$\forall \; \Cg(r,s)$};
    \coordinate (t1) at ($(ab) !0.5! (rs)$);
    \draw (t1) ellipse (7em and 7em)
      (ab) to[out=225,in=90] node[right]{$\B\models\Gamma(c,d,a,b)$} (cd)
      (cd) to[out=270,in=135] node[right]{$\B\models\psi(r,s,c,d)$} (rs);
    \node at (current bounding box.north west) {$\Con(\B)$};
  \end{tikzpicture}
\end{tabular}
\caption{$\A$ has DPC via $\psi$, and $\B$ has DPSC via $\Gamma$ and
  $\psi$.}
\end{figure}
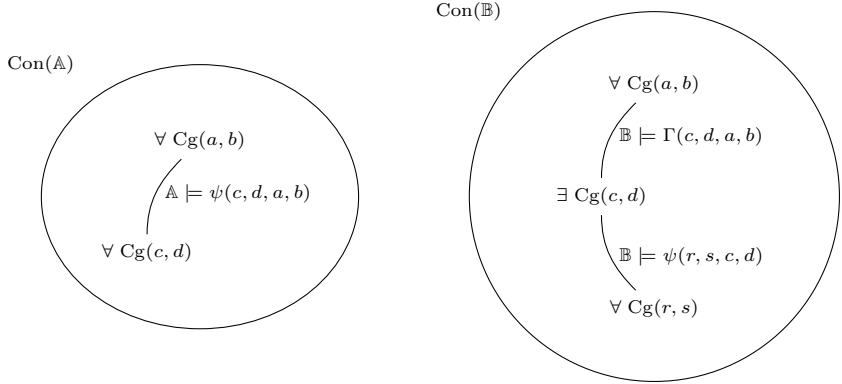

Baker and Wang \cite{BakerWangDPSC} use the fact that congruence
distributive varieties have definable principal subcongruences to give a new
proof of K. Baker's Finite Basis Theorem \cite{BakerFiniteBasis}: if $\A$ is
a finite algebra of finite type and $\V(\A)$ is congruence distributive,
then $\A$ is finitely based. Willard \cite{WillardFiniteBasisTheorem}
extends Baker's theorem by showing that if the variety has finite type, is
residually finite, and is congruence $\wedge$-semidistributive ($\V(\A(\T))$
has these features if $\T$ halts), then the variety is finitely based.
Since $\V(\A(\T))$ is finitely axiomatizable if and only if $\T$ halts, and
finitely axiomatizability is so closely related to DPC and DPSC, it is
natural to consider whether the failure of finite axiomatizability when $\T$
does not halt is related to a failure of DPC or DPSC in $\V(\A(\T))$.

The main result of this paper is to construct an algebra $\A'(\T)$ based on
McKenzie's $\A(\T)$ and to show that $\A'(\T)$ has definable principal
subcongruences if and only if $\T$ halts. Since the halting problem is
undecidable, this proves that the property of having DPSC is undecidable
(i.e. there is no algorithm that takes as input a finite algebra and giving
as output the correct answer to the question: ``does the variety generated
by this algebra have DPSC?''). The proof of this involves many cases, an
exploration of ``$\A'(\T)$-arithmetic'', and a fine analysis of the
polynomials of $\A'(\T)$. We begin in Section \ref{sec:defining A(T)} with a
description of the algebra $\A'(\T)$. Section \ref{sec:modifying mckenzies
argument} details the modifications to McKenzie's original argument that are
necessary to show that $\kappa(\A'(\T))<\omega$ if and only if $\T$ halts.
Section \ref{sec:SIs} then gives a detailed description of the subdirectly
irreducible algebras of $\V(\A'(\T))$ that will be needed throughout. The
proof that DPSC is undecidable is broken into two cases, depending on
whether $\T$ does or does not halt. The case where $\T$ halts is addressed
in Section \ref{sec:T halts}, and is quite complicated, involving many
subcases. The case where $\T$ does not halt is addressed in Section
\ref{sec:T does not halt}, and a short negative answer to Tarski's problem
using the undecidability of definable principal subcongruences is given in
this section as well.

The results in this paper originated with the examination of properties of
$\A(\T)$. Finite axiomatizability is closely related to the properties of
definable principal congruences and definable principal subcongruences, and
there was a natural question of whether McKenzie's negative answer to
Tarski's finite basis problem was the consequence of a more primitive result
concerning either DPC or DPSC. Although it is true that this is the
situation for $\A'(\T)$, it was recently shown by the author in
\cite{MooreAlgNoDPSC} that the original $\A(\T)$ does \emph{not} have DPSC.
This is the first known example of a congruence $\wedge$-semidistributive
variety with finite residual bound that does not have DPSC. The methods used
to prove the undecidability of definable principal subcongruences do not
appear to be amenable to proving the undecidability of definable principal
congruences, but the overall structure of the argument and the fine analysis
of polynomials in $\V(\A'(\T))$ may provide a foundation for proving the
undecidability of DPC as well.

\section{Defining $\A'(\T)$} \label{sec:defining A(T)} We define a
\emph{Turing machine} $\T$ to be a finite list of $5$-tuples $(s,r,w,d,t)$,
called the \emph{instructions} of $\T$, and interpreted as ``if in state $s$
and reading $r$, then write $w$, move $d$, and enter state $t$.'' The set of
states is finite, $r,w\in \{0,1\}$, and $d\in \{\text{L},\text{R}\}$. A
Turing machine takes as input an infinite bidirectional tape
$\tau:\ZZ\to\{0,1\}$ that has finite support (i.e. $\tau^{-1}(\{1\})$ is
finite). If $\T$ stops computation on some input, then $\T$ is said to have
\emph{halted} on that input. For this reason, we say that the Turing machine
halts (without specifying the input) if it halts on the empty tape
$\tau(x)=0$. We enumerate the states of $\T$ as $\{\mu_0,\ldots,\mu_n\}$,
where $\mu_1$ is the initial (starting) state, and $\mu_0$ is the halting
state (which might not ever be reached). $\T$ has no instruction of the form
$(\mu_0,r,w,d,t)$ but for every pair $(i,r)$ with $1\leq i\leq n$ and $r\in
\{0,1\}$ does have precisely one instruction of the form $(\mu_i, r, w, d,
t)$.

Given a Turing machine $\T$ with states $\{\mu_0,\ldots,\mu_n\}$, we
associate to $\T$ an algebra $\A'(\T)$. We will now describe the algebra
$\A'(\T)$. Let
\begin{gather*}
  U = \{1,2,\text{H}\}, 
    \qquad \qquad 
    W = \{C,D,\partial C,\partial D\}, 
    \qquad \qquad 
    A = \{0\}\cup U\cup W,\\
  V_{ir}^s = \{C_{ir}^s, D_{ir}^s, M_i^r, \partial C_{ir}^s,
    \partial D_{ir}^s, \partial M_i^r\} 
    \qquad \text{for} \qquad 
    0\leq i\leq n \text{ and } \{r,s\}\subseteq \{0,1\}, \\
  V_{ir} = V_{ir}^0 \cup V_{ir}^1,
    \qquad \qquad 
    V_i = V_{i0} \cup V_{i1}, 
    \qquad \qquad 
    V = \bigcup \{ V_i \mid 0\leq i\leq n \}.
\end{gather*}
The underlying set of $\A'(\T)$ is $A'(\T) = A\cup V$. In the operations
defined below, the ``$\partial$'' is taken to be a permutation of order $2$
with domain $V\cup W$ (e.g.  $\partial\partial C = C$), and is referred to
as ``bar''. It should be mentioned that $\partial$ is \emph{not} an
operation of $\A'(\T)$.  We now describe the fundamental operations of
$\A'(\T)$. The algebra $\A'(\T)$ is a height $1$ $\wedge$-semilattice (i.e. it
is ``flat'') with bottom element $0$:
\newcolumntype{A}{>{\centering\arraybackslash} m{.45\linewidth}}
\newcolumntype{B}{>{\centering\arraybackslash} m{.45\linewidth}}
\begin{center} \begin{tabular}{AB}
  $\displaystyle{ x\wedge y = \begin{cases}
    x & \text{if } x = y, \\
    0 & \text{otherwise.}
  \end{cases} }$
  &
  \begin{tikzpicture}[font=\small,smooth,node distance=1em]
    \node (x1) {$x_1$};
    \node (x2) [right=of x1] {$x_2$};
    \node (dots) [right=of x2] {$\cdots$};
    \node (zero) [below=of x2] {$0$};
    \draw (x1) -- (zero) (x2)--(zero) (dots) -- (zero);
  \end{tikzpicture}
\end{tabular} \end{center}
This semilattice structure induces an order, $\leq$, on algebras in
$\V(\A'(\T))$. There is a binary nonassociative ``multiplication'', defined
by
\begin{align*}
  & 2\cdot D = H\cdot C = D, & & 1\cdot C = C, \\
  & 2\cdot \partial D = H\cdot \partial C = \partial D, 
  & & 1\cdot \partial C = \partial C,
\end{align*}
and $x\cdot y = 0$ otherwise. The next operations play the role of
controlling the production of large SI's (i.e. those SI's not contained in 
$\HS(\A'(\T))$) in McKenzie's original argument. Such SI's are fully
described in Section \ref{sec:SIs}. Define
\begin{align*}
  J(x,y,z) 
    = (x\wedge \partial y \wedge z) \vee (x\wedge y)
    = \begin{cases}
      x          & \text{if } x = y, \\
      x \wedge z & \text{if } x = \partial y\in V\cup W, \\
      0          & \text{otherwise},
    \end{cases} \\
  J'(x,y,z) 
    = (x\wedge y \wedge z) \vee (x\wedge \partial y)
    = \begin{cases}
      x \wedge z & \text{if } x = y, \\
      x          & \text{if } x = \partial y\in V\cup W, \\
      0          & \text{otherwise},
    \end{cases}
\end{align*} 
and
\[
  K(x,y,z) 
  = (\partial x \wedge y) \vee (\partial x \wedge \partial y \wedge
  z) \vee (x \wedge y \wedge z) 
  = \begin{cases}
    y & \text{if } x = \partial y\in V\cup W, \\
    z & \text{if } x = y = \partial z\in V\cup W, \\
    x\wedge y\wedge z & \text{otherwise}.
  \end{cases}
\]
(In expressions like $x\wedge \partial y \wedge z$, if $y$ does not lie in
the domain of $\partial$, then we take $\partial y$ to be $0$). As we will
see, the $J$ and $J'$ operations force a certain easy-to-analyze structure
on the SI's of the variety, and the $K$ operation allows us to simplify
certain kinds of polynomials in the SI's in $\mathbf{HS}(\A'(\T))$ (i.e. the
small SI's). This simplification of polynomials will be the key to showing
that $\V(\A'(\T))$ has DPSC when $\T$ halts. Define
\begin{align*}
  S_0(u,x,y,z) & = \begin{cases}
    (x\wedge y) \vee (x\wedge z) & \text{if } u\in V_0, \\
    0 & \text{otherwise},
  \end{cases} \\
  S_1(u,x,y,z) & = \begin{cases}
    (x\wedge y) \vee (x\wedge z) & \text{if } u\in \{1,2\}, \\
    0 & \text{otherwise},
  \end{cases} \\
  S_2(u,v,x,y,z) & = \begin{cases}
    (x\wedge y) \vee (x\wedge z) & \text{if } u = \partial v\in V\cup W, \\
    0 & \text{otherwise}.
  \end{cases}
\end{align*}

The operation of left multiplication by $y$, $\lambda_y(x) = y\cdot x$, is,
in general, not injective. The next operation will allow us to produce
``barred'' elements (i.e. produce $\partial x$ from $x$) in cases when
$\lambda_y$ is not injective. Let
\[
  T(w,x,y,z) = \begin{cases}
    w\cdot x & \text{if } w\cdot x = y \cdot z \text{ and } (w,x)=(y,z), \\
    \partial(w\cdot x) & \text{if } w\cdot x = y\cdot z\neq 0 \text{ and } 
      (w,x)\neq (y,z), \\
    0 & \text{otherwise}.
  \end{cases}
\]

Next, we define operations that emulate the computation of the Turing
machine on some tape. First, we define an operation that when applied to
certain elements of $A'(\T)^\ZZ$ will produce something that represents a
``blank tape'':
\[
  I(x) = \begin{cases}
    C_{10}^0 & \text{if } x = 1, \\
    M_1^0 & \text{if } x = \text{H}, \\
    D_{10}^0 & \text{if } x = 2, \\
    0 & \text{otherwise}.
  \end{cases}
\]
For each instruction of $\T$ of the form $(\mu_i,r,s,\text{L},\mu_j)$ and
each $t\in \{0,1\}$ define an operation
\[
  L_{irt}(x,y,u) = \begin{cases}
    C_{jt}^{s'} & \text{if } x = y = 1 \text{ and } u = C_{ir}^{s'} 
      \text{ for some } s', \\
    M_j^t & \text{if } x = \text{H}, y = 1, \text{ and } u = C_{ir}^t, \\
    D_{jt}^s & \text{if } x = 2, y = \text{H}, \text{ and } u = M_i^r, \\
    D_{jt}^{s'} & \text{if } x = y = 2 \text{ and } u = D_{ir}^{s'} 
      \text{ for some } s', \\
    \partial v & \text{if } u\in V \text{ and } 
      L_{irt}(x,y,\partial u) = v\in V \text{ by the above lines}, \\
    0 & \text{otherwise}.
  \end{cases}
\]
Let $\mathcal{L}$ be the set of all such operations. Similarly, for each
instruction of $\T$ of the form $(\mu_i,r,s,\text{R},\mu_j)$ and each $t\in
\{0,1\}$ define an operation
\[
  R_{irt}(x,y,u) = \begin{cases}
    C_{jt}^{s'} & \text{if } x = y = 1 \text{ and } u = C_{ir}^{s'} 
      \text{ for some } s', \\
    C_{jt}^s & \text{if } x = \text{H}, y = 1, \text{ and } u = M_i^r, \\
    M_j^t & \text{if } x = 2, y = \text{H}, \text{ and } u = D_{ir}^t, \\
    D_{jt}^{s'} & \text{if } x = y = 2 \text{ and } u = D_{ir}^{s'} 
      \text{ for some } s', \\
    \partial v & \text{if } u\in V \text{ and } 
      R_{irt}(x,y,\partial u) = v\in V \text{ by the above lines}, \\
    0 & \text{otherwise}.
  \end{cases}
\]
Let $\mathcal{R}$ be the set of all such operations.

When applied to certain elements from $A'(\T)^\ZZ$, these operations
simulate the computation of the Turing machine $\T$ on different input
tapes. Certain other elements of $\{1,2,H\}^{\ZZ}$ serve to track the
position of the Turing machine's head when operations from $\mathcal{L}\cup
\mathcal{R}$ are applied to elements of $\A'(\T)^{\ZZ}$ that encode the
contents of the tape. Define a binary relation $\prec$ on $\{1,2,\text{H}\}$
by $x\prec y$ if and only if $(x,y) = (2,2)$, $(x,y) = (2,\text{H})$, or
$(x,y) = (1,1)$. For $F\in \mathcal{L}\cup \mathcal{R}$ note that $F(x,y,z)
= 0$ except when $x\prec y$. As with the multiplication operation,
operations of the form $m(x) = F(u,v,x)$ with $F\in \mathcal{L}\cup
\mathcal{R}$ are not injective. The next two operations are very much like
the $T$ operation in that they allow us to produce barred elements in some
situations where $m(x)$ fails to be injective. For $F\in \mathcal{L}\cup
\mathcal{R}$ define operations
\begin{align*}
  U_F^0(x,y,z,u) & = \begin{cases}
    \partial F(y,z,u) & \text{if } x\prec z,\; x\neq y, \text{ and } F(y,z,u)\neq 0, \\
    F(y,z,u) & \text{if } x\prec z,\; x = y, \text{ and } F(y,z,u) \neq 0, \\
    0 & \text{otherwise}.
  \end{cases} \\
  U_F^1(x,y,z,u) & = \begin{cases}
    \partial F(x,y,u) & \text{if } x\prec z,\; y\neq z, \text{ and } F(x,y,u)\neq 0, \\
    F(x,y,u) & \text{if } x\prec z,\; y = z, \text{ and } F(x,y,u)\neq 0, \\
    0 & \text{otherwise}.
  \end{cases}
\end{align*}

The operations of $\A'(\T)$ are
\[
  \big\{0, \wedge, (\cdot), J, J', K, S_0, S_1, S_2, T, I \big\} \cup \mathcal{L} 
  \cup \mathcal{R} \cup \big\{U_F^0, U_F^1 \mid F\in \mathcal{L}
  \cup \mathcal{R} \big\}.
\]
Observe that all operations of $\A'(\T)$ are monotonic with respect to the
order induced by the semilattice structure. The $\A(\T)$ algebra from
\cite{McKenzieResidualBoundNotComp} has the same underlying set as $\A'(T)$,
and all of the same operations except for $K$. McKenzie proved the following
theorem.

\begin{thm}[McKenzie \cite{McKenzieResidualBoundNotComp}]
$\kappa(\A(\T)) < \omega$ if and only if $\T$ halts.
\end{thm}

\noindent The fact that $\V(\A'(\T))$ has only finitely many subdirectly
irreducible algebras, all finite, if $\T$ halts is needed to prove that this
variety has definable principal subcongruences. Since we have modified the
algebra that this theorem refers to, we must show that this theorem as well
as other important properties of $\A(\T)$ still hold.

\section{Modifying McKenzie's Argument} \label{sec:modifying mckenzies argument}
McKenzie's argument is quite detailed and long, and fortunately only needs
to be added to, not changed. In this section we will detail the specific
additions to the arguments in papers \cite{McKenzieResidualBounds} and
\cite{McKenzieResidualBoundNotComp} that are needed, in order to prove that
the large subdirectly irreducible algebras in $\V(\A'(\T))$ satisfy
$K(x,y,z) = x\wedge y\wedge z$ and are otherwise precisely the same as the
large subdirectly irreducible algebras in $\V(\A(\T))$ as described in
\cite{McKenzieResidualBoundNotComp}.

\begin{lem} \label{lem:flat and S2=0 then K=J'=xyz}
Suppose that $\B\in \V(\A'(\T))$ is flat and $\B\models S_2(u,v,x,y,z)
\approx 0$. Then 
\begin{enumerate}
  \item $\B\models K(x,y,z) \approx x\wedge y\wedge z$, and
  \item $\B\models J'(x,y,z) \approx x\wedge y\wedge z$.
\end{enumerate} 
\end{lem}
\begin{proof}
We begin with item (1). $\B\in \V(\A'(\T))$, so say $\B = \C/\theta$, where
$\C\leq \A'(\T)^L$ and $\theta\in \Con(\C)$. Suppose that $B\not\models
K(x,y,z)\approx x\wedge y\wedge z$. Then there are $a,b,c\in C$ with
$K(a,b,c) \not\theta\; (a\wedge b\wedge c)$. In particular, from the
definition of $K$ this means that at least $2$ of $a,b,c$ lie in distinct
$\theta$-classes, and by flatness, $(a\wedge b \wedge c)\;\theta\; 0$. We
therefore have that $K(a,b,c)\not\theta\; 0$.

Let $\alpha = K(a,b,c)$. From the definition of $K$ and since $K(a,b,c)\neq
(a\wedge b\wedge c)$, for each $l\in \Supp(\alpha)$ we have $a(l)\in \{
\alpha(l),\partial \alpha(l)\}$. If $\alpha \;\theta\; a$, then since
$\A'(\T)\models K(x,y,z)\wedge x \approx K(x,y,z)\wedge x\wedge y \wedge z$,
\[
  K(a,b,c) 
  = \alpha 
  \;\theta\; \alpha\wedge a
  = K(a,b,c)\wedge a
  = K(a,b,c) \wedge a\wedge b\wedge c.
\]
By the flatness of $\B$, this implies $K(a,b,c) \;\theta\; (a\wedge b\wedge
c)$, which contradicts observations in the first paragraph. It follows that
$\alpha\not\theta\; a$, so $(\alpha\wedge a) \;\theta\; 0$. From the
definition of $K$, the hypothesis that $\B\models S_2(u,v,x,y,z)\approx 0$,
and these observations,
\[
  0 
  \;\theta\; S_2(a,\alpha,a,a,a) 
  = K(\alpha,a,0)
  \;\theta\;
  K(\alpha,a,\alpha\wedge a)
  = \begin{cases}
    a(l) & \text{for } l\in \Supp(\alpha), \\
    0    & \text{for } l\not\in\Supp(\alpha).
  \end{cases}
\]
Let $a'= K(\alpha,a,\alpha\wedge a)$. Then $K(a,b,c) = K(a',b,c) \;\theta\;
K(0,b,c) = 0$, a contradiction.

We will now prove item (2) from item (1). Assume to the contrary that
$\B\not\models J'(x,y,z) \approx x\wedge y \wedge z$. Then there are
$d,e,f\in B$ such that $J'(d,e,f)\neq d\wedge e \wedge f$. If $d\wedge
e\wedge f\neq 0$, then by the flatness of $\B$ it follows that $d=e=f$. In
this case $J'(d,e,f) = d\wedge e\wedge f$ by definition, so it must be that
$d\wedge e\wedge f = 0$. Thus, we are assuming that $J'(d,e,f)\neq 0$.

Since $\A'(\T)\models K(x,y,x) \approx J'(x,y,x) \geq J'(x,y,z)$, for any
$d,e,f\in B$ by the proof of item (1) above we have
\[
  d\wedge e = K(d,e,d) = J'(d,e,d) \geq J'(d,e,f).
\]
$\B$ is flat, and by the previous paragraph $J'(d,e,f)\neq 0$. This means
that $J(d,e,f) = d\wedge e\neq 0$. Therefore $d=e$, and
\[
  J'(d,e,f)
  = J'(d,d,f)
  = d\wedge d\wedge f
  = d\wedge e\wedge f,
\]
a contradiction.
\end{proof}

Many additions to McKenzie's argument occur where induction on polynomial
complexity is used, and the following lemma is the crux of the additional
argumentation in most of these instances.

\begin{lem} \label{lem:A(T) res small}
Let $L$ be an index set and suppose that $\B\leq \A'(\T)^{L}$ and $C\subseteq
B$ are such that
\begin{enumerate}
  \item if $c\in C$ then $c(l)\neq 0$ for all $l\in L$ (we will say that $c$
    is nowhere $0$), and
  \item if $c\in C$ and $a\in B$ are such that $c(l) \in \{a(l),\partial
    a(l)\}$ for all $l\in L$, then $c = a$.
\end{enumerate}
If $f_1(x), f_2(x), f_3(x)$ are polynomials of $\B$ such that for all
$i$ either $f_i(x)$ is constant or $f_i^{-1}(C)\subseteq C$, then the
polynomial $f(x) = K(f_1(x),f_2(x),f_3(x))$ is also either constant in $\B$
or $f^{-1}(C)\subseteq C$.
\end{lem}
\begin{proof}
Let $f_1(x), f_2(x), f_3(x)$ be as in the statement of the lemma, and
let $f(x) = K(f_1(x),f_2(x),f_3(x))$. We will show that $f(x)$ is either
constant or $f^{-1}(C)\subseteq C$. Suppose that $a\in B$ and $f(a) \in C$.
Since
\begin{align*}
  f(a) 
  & = K(f_1(a),f_2(a),f_3(a)) \\
  & = (\partial f_1(a) \wedge f_2(a))
    \vee (\partial f_1(a) \wedge \partial f_2(a) \wedge f_3(a)) 
    \vee (f_1(a)\wedge f_2(a)\wedge f_3(a)),
\end{align*}
and $f(a)$ is nowhere $0$ and $\A'(\T)$ is flat, for each $l\in L$
\begin{itemize}
  \item $\partial f_1(a)(l) = f_2(a)(l) = f(a)(l)$, or
  \item $\partial f_1(a)(l) = \partial f_2(a)(l) = f_3(a)(l) = f(a)(l)$, or
  \item $f_1(a)(l) = f_2(a)(l) = f_3(a)(l) = f(a)(l)$.
\end{itemize}
For $b\in \{f_1(a),f_2(a)\}$ it follows from these formulas that $b(l)\in
\{f(a)(l),\partial f(a)(l)\}$ for all $l\in L$. Hypothesis $(2)$ implies
that $f(x) = f_1(a) = f_2(a)$, and then $f(a)\leq f_3(a)$. Since $f(a)$ is
nowhere $0$, we have $f(a) = f_3(a)$ as well. Thus $f_i(a)\in C$ for $i\in
\{0,1,2\}$. Thus, finally, if $a\not\in C$, then each $f_i$ is constant on
$B$, and then so is $f$.
\end{proof}

\begin{defn} \label{defn:0 absorbing}
Let $\mathcal{C}$ be a class of algebras of the same type whose reduct to
$\{0,\wedge\}$ is a meet semilattice. $\mathcal{C}$ is said to be
\emph{$0$-absorbing} if for every fundamental operation $F(x_1,\ldots,x_n)$,
every $\A\in \mathcal{C}$, and every $a_1,\ldots,a_n\in A$,
\[
  0\in \{a_1,\ldots,a_n\} 
  \qquad \text{implies} \qquad
  F(a_1,\ldots,a_n) = 0.
\]
$\mathcal{C}$ is said to \emph{commute with $\wedge$} if for every
fundamental operation $F$ of some arity $n$,
\[
  \mathcal{C} \models F(x_1,\ldots, x_n) \wedge F(y_1,\ldots,y_n) 
  \approx F(x_1\wedge y_1, \ldots, x_n\wedge y_n).
\]
\end{defn}

We now enumerate the needed additions to McKenzie's proofs in papers
\cite{McKenzieResidualBounds} and \cite{McKenzieResidualBoundNotComp}. To
avoid needlessly long definitions and discussions, the additions will be
presented assuming that the reader has the appropriate paper on hand to
reference. Overall, we will proceed through the main argument in
\cite{McKenzieResidualBoundNotComp} and divert to
\cite{McKenzieResidualBounds} when the main argument makes reference to it.
\begin{enumerate}
  \item In general, we note that $K$ is monotonic, and if $\A\in
  \V(\A'(\T))$ is flat and $\A\models S_2(u,v,x,y,z) \approx 0$, then
  $\A\models J'(x,y,z) \approx K(x,y,z) \approx x\wedge y\wedge z$. This is
  Lemma \ref{lem:flat and S2=0 then K=J'=xyz}.
  
  \item \label{enum: mckenzie add S_Z} In
  \cite{McKenzieResidualBoundNotComp} in the proof of Lemma $4.1$, elements
  $\alpha_n$ and $\beta_n$ of $\A'(\T)^{\ZZ}$ are defined as
  \[
    \alpha_n(k) = \begin{cases}
      1 & \text{if } k<n, \\
      H & \text{if } k=n, \\
      2 & \text{if } k>n,
    \end{cases}
    \qquad \text{and} \qquad
    \beta_n(k) = \begin{cases}
      C & \text{if } k<n, \\
      D & \text{if } k\geq n.
    \end{cases}
  \]
  Let $\Gamma$ be the subuniverse of the algebra generated by these
  elements, $\Sigma$ the set of all configuration elements generated by the
  $\alpha_n$ (that is, the set of all nowhere $0$ outputs of
  $\mathcal{L}\cup\mathcal{R}\cup \{I\}$), and $\Gamma_0$ the subset of
  $\Gamma$ consisting of elements that are $0$ at some coordinate.  It is
  necessary to prove that the set 
  \[
    \Gamma' = \Gamma_0 \cup \Sigma \cup \{\alpha_n, \beta_n \mid n\in \ZZ\}
  \]
  is closed under the operation $K$. By construction, if $u\in
  \Gamma'\setminus \Gamma_0$, then for each $l\in L$, $u(l)$ cannot be a
  barred element (e.g. $\partial C$, $\partial D$, $\partial C_{ir}^s$,
  etc.). From the definition of $K$, we have that if $a,b,c\in
  \Gamma'\setminus \Gamma_0$, then $K(a,b,c) = a\wedge b\wedge c\in \Gamma$.
  Thus
  $K(\Gamma'\setminus\Gamma_0,\Gamma'\setminus\Gamma_0,\Gamma'\setminus\Gamma_0)
  \subseteq \Gamma'$.  The set $\Gamma_0$ contains elements that have a
  value of $0$ at some coordinate. Since $K$ is $0$-absorbing in its first
  and second coordinates, $K(\Gamma_0,\Gamma',\Gamma') \cup
  K(\Gamma',\Gamma_0,\Gamma')\subseteq \Gamma_0$. Furthermore, if $a,b\in
  \Gamma'\setminus \Gamma_0$ and $c\in \Gamma_0$ then since $a(l)\neq
  \partial b(l)$ for any $l\in \ZZ$ ($\Gamma'\setminus \Gamma_0$ contains no
  barred elements), we have that $K(a,b,c) = a\wedge b\wedge c$, so in this
  case $K(a,b,c)\in \Gamma_0$ since $w\in \Gamma_0$. Therefore
  $K(\Gamma'\setminus\Gamma_0,\Gamma'\setminus\Gamma_0,\Gamma_0)\subseteq
  \Gamma_0$. We have now show that $K(\Gamma',\Gamma',\Gamma')\subseteq
  \Gamma'$, so we are done.

  \item The proof of Lemma $5.3$ in \cite{McKenzieResidualBoundNotComp}
  needs only a few added words at the end of the first full paragraph on
  page 41 to demonstrate that our new operation $K$ can be dealt with the
  same way as the operations $J$ and $J'$ are handled in this proof.

  \item Prior to the statement of Lemma $5.5$ in
  \cite{McKenzieResidualBoundNotComp}, it is written that the lemma is a
  restatement of Lemmas $6.7$-$6.9$ of \cite{McKenzieResidualBounds}. All of
  these lemmas go through without modification, except for Lemma $6.8$.
  Lemma $6.8$ concerns itself with a subalgebra $\m{B}$ of $\A'(\T)^L$ and a
  subset $B_1$ of $B$ defined by
  \[
    B_1 = \big\{ u\in B \mid u = p 
    \text{ or } x_0 x_1\cdots x_n = p 
    \text{ and } u\in \{x_0\,\ldots, x_n\} \big\}
  \]
  ($p$ is a fixed element of $\B$ that has the property, amongst many
  others, of being nowhere $0$). The product in $x_0 x_1\cdots x_n$ in the
  definition of $B_1$ associates to the right. At the very start of the proof
  of Lemma $6.8$, induction on the complexity of polynomials is used to
  prove that if $u\in B$ and $f(u)\in B_1$ then $f(x)$ is either constant or
  $u\in B_1$. Lemma $6.6$ in \cite{McKenzieResidualBounds} states that $B_1$
  consists of elements that are nowhere $0$ and having the property that if
  $u\in B_1$ and $v\in B$ are such that $u(l) \in\{v(l),\partial v(l)\}$ for
  all $l\in L$, then $u = v$. Taking $C=B_1$ in Lemma \ref{lem:A(T) res
  small} above, the inductive step of the proof for the $K$ operation
  follows.

  \item In \cite{McKenzieResidualBoundNotComp} in the proof of Lemma $5.7$
  part (iii), induction on the complexity of polynomials is used to prove
  that if $f(x)$ is a non-constant polynomial of $\B$ and $f(u)\in B_1$ for
  some $u\in B$, then $u\in B_1$. This is the same argument that appears in
  the previous item above.

  \item A consequence of these lemmas is that every large SI of
  $\V(\A'(\T))$ is flat and models $S_i(\overline{u},x,y,z)\approx 0$ for
  every $i\in \{0,1,2\}$. By Lemma \ref{lem:flat and S2=0 then K=J'=xyz}, we
  have that $K(x,y,z) \approx x\wedge y\wedge z$ in large SI's (in fact, in
  large SI's $K(x,y,z) \approx x\wedge y\wedge z \approx J(x,y,z)$).
  Therefore, the addition of the $K$ operation does not change the structure
  of the large SI's of $\V(\A'(\T))$.
\end{enumerate}

This completes the changes that are needed to adapt McKenzie's description
of large SI algebras in $\V(\A(\T))$ to $\V(\A'(\T))$. We will now give an
explicit description of exactly what these algebras look like.

\section{Subdirectly irreducible algebras in $\V(\A'(\T))$} \label{sec:SIs}
Define terms $e_0, e_1, e_2$ in $\V(\A'(\T))$ by
\begin{equation} \label{defn:e_i} \begin{aligned}
  e_0(m,x) & = S_0(m,x,x,x), & e_2(m,n,x) & = S_2(m,n,x,x,x), \\
  e_1(m,x) & = S_1(m,x,x,x).
\end{aligned} \end{equation}
The argument in the previous section shows that large SI's model
$e_i(\overline{y},x)\approx 0$ for all $i\in \{0,1,2\}$. The small
subdirectly irreducible algebras break into two categories: those that
satisfy $\exists \overline{n} [ e_i(\overline{n},x)\approx x ]$ for some
$i\in \{0,1,2\}$, and those that do not (in which case they satisfy
$e_i(\overline{y},x) \approx 0$ for all $i\in \{0,1,2\}$). As we will see in
Lemma \ref{lem:dpsc all factors}, all SI algebras that do model
$\exists\overline{n} [ e_i(\overline{n},x)\approx x ]$ for some $i\in
\{0,1,2\}$ have J\'{o}nsson polynomials and are thus congruence
distributive, and by an argument due to Baker and Wang \cite{BakerWangDPSC},
these algebras will also have DPSC.

There are only three different isomorphism types for small SI algebras
satisfying $e_i(\overline{y},x) \approx 0$ for all $i\in \{0,1,2\}$. Two of
these small SI's are subalgebras of $\A'(\T))$, and the remaining one is the
$4$-element quotient
\begin{equation} \label{defn:W}
  \W = \left<H,C\right>/\Cg(M_1^0,0) = \{0,H,C,D,M_1^0\}/\Cg(M_1^0,0)
\end{equation}
(this will be proved in Lemma \ref{lem:small SI's e_i=0}). The fundamental
operations of $\A'(\T)$ are all identically $0$ in $\W$ except for $\wedge$,
which makes $\left< W;\wedge\right>$ a flat semilattice, and the following
operations:
\begin{gather*} 
  \begin{aligned}
    & x\cdot y = 0         && \text{except for} && H\cdot C = D,   \\
    & T(w,x,y,z) = 0       && \text{except for} && T(H,C,H,C) = D,
  \end{aligned} \\
  J(x,y,z) = x\wedge y, 
    \quad \text{and} \quad J'(x,y,z) = K(x,y,z) = x\wedge y \wedge z.
\end{gather*}

\begin{lem} \label{lem:small SI's e_i=0}
Let $\B\in \HS(\A'(T))$ be nontrivial and subdirectly irreducible and such
that $\B\models e_i(\overline{y},x)\approx 0$ for all $i\in \{0,1,2\}$. Then
$\B$ is isomorphic to the two element subalgebra $\{0,C\} \leq \A'(\T)$, the
three element subalgebra $\{0,H,M_1^0\} \leq \A'(\T)$, or to the four
element quotient $\W$.
\end{lem}
\begin{proof}
We will first consider subalgebras. Suppose that $\B\leq \A'(\T)$ is SI.
Since $\B\models S_i(\overline{n},x,x,x)\approx 0$ for all $i$, we have
$(\{1,2\}\cup V_0)\cap B = \emptyset$ and $\{x,\partial x\}\not\subseteq B$
for $x\in W\cup V$ (i.e. the ``bar-able" elements of $A'(\T)$). It follows
that all
fundamental operations are identically $0$ except for $\wedge$ and
\begin{equation} \label{eqn:small SI's e_i=0 polys} \begin{gathered} 
  \begin{aligned}
    & I(x) = 0 && \text{except for} && I(H) = M_1^0, \\
    & x\cdot y = 0 && \text{except for} && H\cdot C = D 
      \text{ and } H\cdot \partial C = \partial D,
  \end{aligned} \\
  J(x,y,z) = x\wedge y, \\
  K(x,y,z)=J'(x,y,z) = x\wedge y\wedge z, \text{ and} \\
  T(w,x,y,z) = (w\wedge y)\cdot (x\wedge z).
\end{gathered} \end{equation}

There are two cases depending upon whether or not $H$ is an element of $B$.
For the first case, suppose that $H\not\in B$. Then $x\cdot y = T(w,x,y,z) =
I(x) = 0$, so $\B$ is a flat semilattice. It follows that if $x,y\in B$ are
distinct and nonzero, then $\Cg^\B(x,0)$ and $\Cg^\B(y,0)$ are distinct and
cover $\0$ in $\Con(\B)$, and hence $\B$ is not subdirectly irreducible.
Therefore $B = \{0,x\}$, so $\B$ is isomorphic to the subalgebra $\{0,C\}$.

For the second case, suppose that $H\in B$. Then $I(H) = M_1^0\in B$ as
well. If $F(x)$ is a fundamental translation of $\B$, then $F(M_1^0)=M_1^0$
or $F(M_1^0)=0$ (see the description of the fundamental operations above).
Two consequences of this are that $\Cg^\B(M_1^0,0)$ is the monolith of $\B$
and that if $\Cg^{\B}(a,0) = \Cg^{\B}(M_1^0,0)$ then $a = M_1^0$.

We will now show that $B = \{0,H,M_1^0\}$. Suppose that $x\in B\setminus
\{0,H,M_1^0\}$. If $x = C$, then $H\cdot C = D$, so $D\in B$ as well. An
argument similar to the one in the previous paragraph will show that
$\Cg^\B(D,0)$ covers $\0$. Likewise, if $x = \partial C$, then
$\Cg^\B(\partial D,0)$ covers $\0$. Both of these possibilities are
contradictions. If $x\not\in \{C,\partial C\}$ then $\Cg^\B(x,0)$ also
covers $\0$, again contradicting $\B$ being subdirectly irreducible.
Therefore it must be that $B\setminus \{0,H,M_1^0\} = \emptyset$. It follows
that the only subdirectly irreducible subalgebras of $\A'(\T)$ are
isomorphic to either $\{0,C\}$ or $\{0,H,M_1^0\}$. 

We now examine the situation when $\B$ is a proper quotient of a subalgebra
of $\A'(\T)$. Suppose that $\B = \B_1/\theta\in \HS(\A'(\T))$ is subdirectly
irreducible. In the quotient $\B$, the equations \eqref{eqn:small SI's e_i=0
polys} hold by the same argument appearing at the start of the proof. Since
$\A'(\T)$ is a flat semilattice, the only possibly nontrivial class of
$\theta$ is the one containing $0$. As before, we have two cases to
consider, depending on whether $B$ contains a nonzero $H/\theta$. If
$H\not\in B_1$ or $(H,0)\in \theta$, then $\B$ is a subdirectly irreducible
flat semilattice (i.e. all the operations are $0$ except for $\wedge$), so
$\B$ must be isomorphic to the $2$-element subalgebra $\{0,C\}$ of
$\A'(\T)$. 

Suppose now that $H\in B_1$ and $(H,0)\not\in\theta$. There are three cases
to consider:
\begin{itemize}
  \item $(M_1^0,0)\not\in\theta$,
  \item $(M_1^0,0)\in\theta$ and [ $C\not\in B_1$ or $(C,0)\in \theta$ ], or
  \item $(M_1^0,0)\in\theta$ and [$C\in B_1$ and $(C,0)\not\in\theta$ ].
\end{itemize}
If $(M_1^0,0)\not\in\theta$ then $\Cg^\B(M_1^0,0)$ is the monolith of $\B$,
so by the last paragraph $\B$ is isomorphic to the $3$-element subalgebra
$\{0,H,M_1^0\}$ of $\A'(\T)$. If instead $(M_1^0,0)\in \theta$ and [
$C\not\in B_1$ or $(C,0)\in \theta$ ], then $\Cg^\B(H,0)$ must cover $\0$,
so $\B$ is isomorphic to the $2$-element subalgebra $\{0,C\}$. Suppose now
that $(M_1^0,0)\in \theta$ and [ $C\in B_1$ and $(C,0)\not\in \theta$ ]. If
$(D,0)\in\theta$, then both $\Cg^\B(H,0)$ and $\Cg^\B(C,0)$ cover $\0$,
contradicting the subdirect irreducibility. If $(D,0)\not\in\theta$, then
$\Cg^\B(D,0)$ covers $\0$. An argument similar to when $\B$ is a subalgebra
shows that $x\not\in B_1$ or $(x,0)\in \theta$ for all $x\in B_1\setminus
\{0,H,C,D\}$. When this happens $\B$ is isomorphic to the algebra $\W$
described in $\eqref{defn:W}$ above.
\end{proof}

Large subdirectly irreducible algebras in $\V(\A'(\T))$ come in two types:
sequential and machine. Both of these types of algebras model the identities
$S_i(\overline{n},x,y,z) \approx 0$, $J(x,y,z) \approx x$, and $J'(x,y,z)
\approx K(x,y,z) \approx x\wedge y \wedge z$. Sequential algebras are
distinguished as additionally modeling the identities $I(x) \approx F(x,y,z)
\approx U_F^\epsilon(w,x,y,z) \approx 0$ for all $F\in \mathcal{L}\cup
\mathcal{R}$ and all $\epsilon\in \{0,1\}$. Machine algebras are
distinguished as modeling the identities $x\cdot y \approx T(w,x,y,z)
\approx 0$.

We begin the description of the sequential algebras by describing an algebra
$\S_{\ZZ}$ in which every sequential algebra is embeddable (but which may
not belong to $\V(\A'(\T))$). The algebra $\S_{\ZZ}$ has underlying set
$S_{\ZZ} = \{0, a_i,b_i \mid i\in \ZZ\}$ and fundamental operations of
$S_{\ZZ}$ are the same as $\A'(\T)$, but are all identically $0$ except for
$\wedge$, $(\cdot)$, $T$, $J$, $J'$, and $K$. The operation $\wedge$ is
defined so that $\left< S_{\ZZ}; \wedge\right>$ is a flat meet semilattice
with bottom element $0$. The operation $(\cdot)$ is defined so that
$a_n\cdot b_{n+1} = b_n$, and $0$ otherwise. The operations $T$, $J$, $J'$,
and $K$ are defined
\begin{align*}
  & J(x,y,z) = x\wedge y, & & J'(x,y,z) = K(x,y,z) = x\wedge y\wedge z, \\
  & T(w,x,y,z) = (w\cdot x) \wedge (y\cdot z).
\end{align*}
Define $\S_{\omega}$ to be the subalgebra of $\S_{\ZZ}$ with universe
$\{0,a_i,b_i\mid i\geq 0\}$, and define $\S_n$ to be the subalgebra of
$\S_{\ZZ}$ with universe $\{0,a_0,b_0,\ldots,a_n,b_n\}$. The algebras
$\S_{\omega}$ and $\S_n$ are subdirectly irreducible, with monoliths
$\Cg(b_0,0)$. With the additions described earlier in section
\ref{sec:modifying mckenzies argument}, McKenzie's argument in
\cite{McKenzieResidualBoundNotComp} proves that $\S_{\ZZ}\in \V(\A'(\T))$ if
and only if $\T$ does not halt, and that $\T$ halts if and only if there is
some maximum $N\in \NN$ such that $\S_{N}\in\V(\A'(\T))$.  $\S_{\ZZ}$,
$\S_{\omega}$, and $\S_n$ for $n\in \NN$ are the sequential algebras, but
only $S_n$ and $\S_{\omega}$ are subdirectly irreducible.

\smallskip

Next, we restate the description of machine algebras given by McKenzie
\cite{McKenzieResidualBoundNotComp}. We begin the description of machine
algebras by defining an algebra (possibly not in $\V(\A'(\T))$) that will
have a quotient isomorphic to our hypothetical machine algebra. Let
$N\subseteq \ZZ$ be a nonempty interval and let $\mathcal{Q} = \left<
\tau,j,\gamma\right>$ be any configuration of the Turing machine $\T$ (here
$\tau$ is the tape function, $j\in N$ is the head position, and $\gamma$ is
the state of the machine). We say that $\mathcal{Q}$ is an initial
configuration if $\tau$ is the blank tape (the tape consisting of all $0$'s,
written as $\tau_0$ below) and $\gamma = \mu_1$ (the starting state). We say
that $\mathcal{Q}$ is a halting configuration if $\gamma = \mu_0$ (the
halting state). Let $\Omega_N$ denote the set of all configurations $\left<
\tau,j,\gamma\right>$ with $j\in N$ and $\tau(\ZZ\setminus N) = \{0\}$.
Write $\mathcal{P}\leq_N \mathcal{Q}$ if there is a finite sequence
$\mathcal{Q} = \mathcal{Q}_0, \ldots, \mathcal{Q}_m = \mathcal{P}$ with
$\mathcal{Q}_i\in \Omega_N$ and such that $\mathcal{Q}_{i+1} =
\T(\mathcal{Q}_i)$.

Let $\Sigma_N = \{ a_n \mid n\in N \}$, and assume that $\Sigma_N$,
$\Omega_N$, and $\{0\}$ are pairwise disjoint. Let $\mathbb{P}_N$ be the
algebra where
\begin{itemize}
  \item the universe is $P_N = \{0\} \cup \Sigma_N\cup \Omega_N$.
  \item the operations $(\cdot), S_0, S_1,S_2,T$ are identically $0$.
  \item $\wedge$ makes $\left< P_N, \wedge \right>$ a flat semilattice.
  \item $J(x,y,z) = x\wedge y$ and $J'(x,y,z) = K(x,y,z) = x\wedge y \wedge
    z$.
  \item $I(a_n) = \left< \tau_0, n, \mu_1\right>\in \Omega_N$ and $I(x) = 0$
    otherwise (here $\tau_0$ is the tape consisting of all $0$'s).
  \item if $F = L_{ir\epsilon}\in \mathcal{L}$ where $\mu_i rsL\mu_j$ is an
    instruction of $\T$ and $\mathcal{Q} = \left< \tau,n+1,\mu_i\right>$ is
    a configuration in $\Omega_N$, then $F(a_n,a_{n+1},\mathcal{Q}) =
    \T(\mathcal{Q})$ provided that $n\in N$, $\T(\mathcal{Q})\in \Omega_N$,
    $\tau(n+1) = r$, and $\tau(n) = \epsilon$. In all other cases $F(x,y,z)
    = 0$.  The case when $F = R_{ir\epsilon}\in \mathcal{R}$ is defined
    analogously.
  \item if $F\in \mathcal{L}\cup \mathcal{R}$ and $n,n+1\in N$, we
    have 
    \[
      U_F^0(a_n,a_{n+1},a_{n+1},x) 
      = F(a_n,a_{n+1},x) 
      = U_F^1(a_n,a_n,a_{n+1},x),
    \]
    and $U_F^j(w,x,y,z) = 0$ otherwise.
\end{itemize}
Next, we describe the congruence of $\mathbb{P}_N$ which we will quotient
by. Assume the set $\Phi\subseteq \Omega_N$ and the element $\mathcal{P}\in
\Phi$ satisfy the following conditions.
\begin{itemize}
  \item For all $\mathcal{Q}\in \Phi$ we have $\mathcal{P}\leq_N
    \mathcal{Q}$.
  \item If $\mathcal{Q}\in \Phi$ and $\mathcal{P}\leq_N \T(\mathcal{Q})$
    then $\T(\mathcal{Q})\in \Phi$.
  \item If $\mathcal{Q}\in \Omega_N$ is an initial configuration and
    $\mathcal{P}\leq_N \mathcal{Q}$ then $\mathcal{Q}\in \Phi$.
  \item If $\mathcal{Q}, \mathcal{Q}'\in \Omega_N$, $\mathcal{Q}'$ is a
    halting configuration, and $\mathcal{Q}'\leq_N \mathcal{Q}$ then
    $\mathcal{Q}\not\in\Phi$.
  \item $|N|>1$ and for every $n\in N$, there is some $\left<
    \tau,n,\gamma\right>\in \Phi$.
\end{itemize}
Define $\Gamma$ to be $(\Omega_N\setminus \Phi)\cup \{0\}$ and let
$\Theta_{(\Phi)}$ be the congruence of $P_N$ whose only nontrivial class is
$\Gamma$. McKenzie gives the following theorem at the end of
\cite{McKenzieResidualBoundNotComp}, which with the addition of the
arguments above still holds for the modified $\A'(\T)$.

\begin{thm}[McKenzie \cite{McKenzieResidualBoundNotComp}]
\label{thm:McKenzie classify large SI's}
$\Theta_{(\Phi)}$ is a congruence relation of $\mathbb{P}_N$ and the algebra
$\mathbb{P}_N/\Theta_{(\Phi)}$ is a subdirectly irreducible algebra that
belongs to $\V(\A'(\T))$. Every large SI in $\V(\A'(\T))$ is either
embeddable in $\S_\omega$ or is isomorphic to $\mathbb{P}_N/\Theta_{(\Phi)}$
for some $N$ and $\Phi$ as above.
\end{thm}

The above description of the SI algebras in $\V(\A'(\T))$ extends to
$\V(\A'(\T))$ the result that $\kappa(\A'(\T)) < \omega$ if and only if $\T$
halts to $\A'(\T)$.

\begin{thm}
$\kappa(\A'(\T)) < \omega$ if and only if $\T$ halts.
\end{thm}

\section{If $\T$ halts} \label{sec:T halts}
The argument to show that $\V(\A'(\T))$ has definable principal
subcongruences if $\T$ halts is quite long and intricate, so we will begin
by giving an description of the different cases.

\begin{defn} \label{defn:translation primitive poly}
Let $F(x_1,\ldots,x_n)$ be a fundamental operation of $\A'(\T)$, $\B\in
\V(\A'(\T))$, and $b_1,\ldots b_n\in B$. The polynomials 
\[
  F^{(i)}_{b_1,\ldots,b_n}(x) = F(b_1,\ldots,\stackrel{i}{\hat{x}}, \ldots, b_n)
  \qquad \text{for} \qquad i\in \{1,\ldots,n\}
\]
are called \emph{fundamental translations of $F$}.

If $h(x)$ is a polynomial of $\B$ that is generated under composition by
fundamental translations, we will say that $h(x)$ is a \emph{primitive
polynomial}. The set of all primitive polynomials of $\B$ will be denoted
\[
  \mathcal{P}(\B) = \{ h(x)\in \Pol_1(\B) \mid h(x) \text{ is generated by
  fundamental translations} \}.
\]
When the algebra $\B$ is clear from the context which $h(x)$ is mentioned
in, $\mathcal{P}(\B)$ will be shortened to just $\mathcal{P}$.
\end{defn}

If $\B$ is an algebra, then by Maltsev's Lemma, $(c,d)\in \Cg^\B(a,b)$ if
and only if there is a sequence of elements, $c = k_1, k_2, \ldots, k_n =
d$, terms $f_1, \ldots, f_{n-1}$, and constants $\overline{e}\in B^m$ such
that $\{f_i(\overline{e},a),f_i(\overline{e},b)\} = \{k_i,k_{i+1}\}$ for all
$i$. Equivalently, we can take the polynomials $f_i(\overline{e},x)$ to be
generated by fundamental translations.  A \emph{congruence scheme}, as in
\cite{FriedGratzerQuackenbushUniformCong}, is a first-order formula,
$\phi(w,x,y,z)$, that asserts the existence of such elements $k_1, \ldots,
k_n$ and constants $\overline{e}$ for some fixed sequence of terms. A
disjunction of congruence schemes is a \emph{congruence formula}, and every
$(c,d)\in \Cg^\B(a,b)$ satisfies some congruence scheme.  Thus, showing that
a principal congruence is definable can be reduced to finding a finite
number of schemes that fully describe the congruence, and showing that a
variety has definable principal congruences can be reduced to showing that
there is a finite number of congruence schemes that fully describe every
principal congruence in every algebra in the variety.

Begin with an arbitrary $\B\in \V(\A'(\T))$ with subdirect representation
$\B\leq \prod_{l\in L} \C_l$ where each $\C_l$ is subdirectly irreducible.
Recall from the previous section that $e_i(\overline{n},x) =
S_i(\overline{n},x,x,x)$, where $\overline{n} = n_1$ if $i\in \{0,1\}$ and
$\overline{n} = (n_1,n_2)$ if $i = 2$.  The isomorphism types of the $\C_l$
come in $4$ different flavors. If $\C_l$ is subdirectly irreducible, then
exactly one of the following holds:
\begin{enumerate} \renewcommand{\theenumi}{\alph{enumi}}
  \item \label{enum:e_i SIs} $\C_l\models \exists \overline{n} [
    e_i(\overline{n},x)\approx x ]$ for some $i\in \{0,1,2\}$. Any such
    $\C_l$ is necessarily small (i.e. contained in $\mathbf{HS}(\A'(\T))$
    (see Lemma $5.2$ in \cite{McKenzieResidualBoundNotComp}). For fixed
    $i\in \{0,1,2\}$ and $\overline{m}\in B^2\cup B$, every model of
    $e_i(\overline{m},x)\approx x$ is congruence distributive (see the proof
    of Lemma \ref{lem:dpsc all factors}), and the class of these models (for
    a single fixed $i$) has definable principal subcongruences (see Lemma
    \ref{lem:dpsc all factors}).
  \item \label{enum:sporadic SIs} $\C_l$ is small and $\C_l\models
    e_i(\overline{y},x)\approx 0$ for all $i\in \{0,1,2\}$. In this case
    there are just $3$ isomorphism types (see Lemma \ref{lem:small SI's
    e_i=0}).
  \item \label{enum:seq SIs} $\C_l$ is large (i.e. not contained in
    $\mathbf{HS}(\A'(\T))$) and $\C_l\models e_i(\overline{y},x)\approx 0$
    for all $i\in \{0,1,2\}$ and $\C\models I(x)\approx F(x,y,z)\approx 0$
    for all $F\in \mathcal{L}\cup \mathcal{R}$.  In this case, $\C$ is
    \emph{sequential}. SI's of this type were fully described in Section
    \ref{sec:SIs}.
  \item \label{enum:machine SIs} $\C_l$ is large and $\C\models
    e_i(\overline{y},x)\approx 0$ for all $i\in \{0,1,2\}$ and $\C\models
    x\cdot y \approx T(w,x,y,z)\approx 0$, In this case, $\C$ is
    \emph{machine}. SI's of this type were fully described in Section
    \ref{sec:SIs}.
\end{enumerate}

In order to show that $\V(\A'(\T))$ has definable principal subcongruences,
we will produce congruence formulas $\Gamma$ and $\psi$ such that for any
$\B\in \V(\A'(\T))$ and any $a',b'\in B$ there is $(c,d)\in \Cg^{\B}(a',b')$
witnessed by $\Gamma(c,d,a',b')$ and such that the relation ``$(x,y)\in
\Cg^{\B}(c,d)$'' is defined by $\psi(x,y,c,d)$.  Let $\B\leq \prod_{l\in L}
\C_l$ be a subdirect representation of $\B$ by subdirectly irreducible
algebras. The way in which $(c,d)$ is produced depends on the isomorphism
types of the $\C_l$ with $l\in L$ such that $a'(l)\neq b'(l)$. Our first
step is to assume without loss of generality that $a'\not\leq b'$ and to
take $a = a'$ and $b = a'\wedge b'$ so that $b < a$. Let $K = \{l\in L \mid
a(l)\neq b(l)\}$. The case distinctions follow.
\begin{enumerate}
  \item There is $k\in K$ such that $\C_k\models \exists\overline{n} [
    e_i(\overline{n},x)\approx x ]$. These are the SI's described in item
    (\ref{enum:e_i SIs}) above.
  \item The previous case does not hold, and there is $k\in K$ such that
    $\C_k$ is sequential. In this case either a translation of the operation
    $(\cdot)$ will distinguish $a$ and $b$, or $(a(l),b(l))$ lies in the
    monolith of $\C_l$ for all $l\in L$. These are the SI's described in
    item (\ref{enum:seq SIs}) above.
  \item The previous cases do not hold, and there is $k\in K$ such that
    $\C_k$ is machine, and either a translation of one of the operations
    from $\mathcal{L}\cup\mathcal{R}\cup \{I\}$ will separate $a$ and $b$,
    or $(a(l),b(l))$ lies in the monolith of $\C_l$ for all $l\in L$.  These
    are the SI's described in item (\ref{enum:machine SIs}) above.
  \item The previous cases do not hold, so it must be that the only $k\in K$
    are such that $\C_k$ is one of the three small SI's that satisfy
    $e_i(\overline{n},x)\approx 0$ for all $i\in \{0,1,2\}$. These are the
    SI's described in item (\ref{enum:sporadic SIs}) above.
\end{enumerate}
We begin the proof for Case $1$ with a slightly specialized version of a
theorem from Baker and Wang \cite{BakerWangDPSC}.

\begin{lem}\label{lem:jonsson polys dpsc}
Let $\V$ be a locally finite variety and let 
\[
  P(\overline{c}) = \{ p_j(\overline{c},x_1,x_2,x_3) \mid 1\leq j\leq K \}
\]
be terms in $\V$ with (a fixed number of) constant symbols $\overline{c}$.
Suppose that $J(\overline{c})$ is the set consisting of the J\'{o}nsson
identities for the polynomials $P(\overline{c})$ in the variables $\{x_1,
x_2, x_3\}$. Then the class
\[
  \M = \Mod_{\V}(\exists \overline{c} \; J(\overline{c}))
    = \{ \B\in \V \mid \B\models \exists \overline{c} \; J(\overline{c}) \}
\]
has definable principal subcongruences if $\kappa(\V) = N < \omega$.
\end{lem}
\begin{proof}
The notable modification of the proof given in \cite{BakerWangDPSC} is at
\eqref{eqn:baker wang mod} below.

Let $\B\in \M$, let $a, b\in B$ be distinct, and fix $\overline{c}\in B^n$
witnessing $\B\models J(\overline{c})$. Let $\B \leq \prod_{l\in L} \C_l$ be
a subdirect representation of $\B$ by subdirectly irreducible algebras such
that whenever $k,l \in L$ and $\C_k \equiv \C_l$ then $\C_k = \C_l$. Since
$\kappa(\V) < \omega$, each $\C_l$ is finite and there are only finitely
many distinct ones. We will construct a finite subalgebra $\C\leq \B$, and
then find a pair $(c,d)\in \Cg^\C(a,b)$ such that $c\neq d$ and
$\Cg^\B(c,d)$ is uniformly definable (i.e. definable in a way that depends
only on $\V$, and not on $\B$, $c$, or $d$).

Choose $k\in L$ such that $a(k)\neq b(k)$ and $|C_k|$ is maximal with this
property. Choose preimage representatives $s_1,\ldots s_M\in B$ of $\C_k$
and let
\begin{equation} \label{eqn:baker wang mod}
  \C = \left<\{a,b,\overline{c} \} \cup \{s_1,\ldots, s_M\} \right>.
\end{equation}
Since $\kappa(\V) = N < \omega$ and $\V$ is locally finite, any such $\C$
has size bounded by a number depending only on $N$ and the number of
constants $\overline{c}$. Since $\C$ has bounded size, congruences are
defined by a finite number of congruence schemes. By construction,
$\pi_k(\C) = \C_k$ and since any subalgebra of $\B$ containing
$\overline{c}$ is congruence distributive (any such subalgebra has
J\'{o}nsson polynomials), $\C$ is congruence distributive.

$\C_k$ is subdirectly irreducible, so $\ker(\pi_k|_\C)$ is completely meet
irreducible in $\text{Con}(\C)$ (the congruence lattice of $\C$). Since $\C$
is congruence distributive, the interval $[0,\ker(\pi_k|_\C)]$ is a prime
ideal and therefore the complement is a filter with a least element, call it
$\alpha$, which is join-prime.  Therefore $\alpha$ is a principal
congruence, say $\alpha = \Cg^\C(c,d)$, and $\alpha$ is the least congruence
not below $\ker(\pi_k|_\C)$. Since $\Cg^\C(a,b)\not\leq\ker(\pi_k|_\C)$, by
minimality of $\alpha$ we have $\alpha = \Cg^\C(c,d)\leq \Cg^\C(a,b)$. By
the previous paragraph, $|C|$ is bounded by a number depending only on $N$
and the number of constants $\overline{c}$. It follows that there is a
congruence formula determined entirely by this bound that witnesses
$(c,d)\in \Cg^\C(a,b)$.

Let $l\in L$ and suppose that $c(l)\neq d(l)$. Then $\Cg^\C(c,d)\not\leq
\ker(\pi_l|_\C)$ and $a(l)\neq b(l)$. By the minimality of $\alpha =
\Cg^\C(c,d)$, it must be that $\ker(\pi_l|_\C)\leq \ker(\pi_k|_\C)$. Hence
there is a surjective mapping
\[
  \pi_l(\C) \cong 
  \C/\ker(\pi_l|_\C) \twoheadrightarrow \C/\ker(\pi_k|_\C) 
  \cong \pi_k(\C) = \C_k.
\]
Now, $\C_k$ was chosen to be maximal such that $a(k)\neq b(k)$, so the
mapping must also be injective since $\C_l$ is finite. Thus $\pi_l(\C) =
\C_k$.

Let $r,s\in B$ be distinct with $(r,s)\in \Cg^\B(c,d)$. We shall construct
a finite $\D$ such that $(r,s)\in \Cg^\D(c,d)$. Let $\D = \left<C \cup
\{r,s\} \right>$. As with $\C$, any such $\D$ has size bounded by a number
depending only on $N$ and the number of constants $\overline{c}$, and so
congruences in $\D$ are defined by a congruence formula determined entirely
by this bound.  Since $\overline{c}\in D^n$, we also have that $\D$ is
congruence distributive. Let $l\in L$. If $c(l)\neq d(l)$ then by the above
paragraph $\pi_l(\D) = \pi_l(\C) = \C_k$, so 
\[
  (r(l),s(l)) \in \Cg^{\pi_l(\C)}(c(l),d(l)) = \Cg^{\pi_l(\D)}(c(l),d(l)).
\]
On the other hand, if $c(l) = d(l)$ then $r(l) = s(l)$, so it follows that
$(r(l),s(l))\in \Cg^{\pi_l(\D)}(c(l),d(l)) = \0_{\pi_l(\D)}$. In either
case, $(r(l),s(l))\in \Cg^{\pi_l(\D)}(c(l),d(l))$ for all $l\in L$. To
complete the proof we need only prove the following claim.

\begin{claim}
Let $\D$ be finite and congruence distributive and let
$\D\leq \prod_{i\in I} \C_i$. Then $(r,s)\in \Cg^\D(c,d)$ if and
only if $(r(i),s(i))\in \Cg^{\pi_i(\D)}(c(i),d(i))$ for all $i\in I$.
\end{claim}
\begin{claimproof}
One direction is clear, since the $i$-th projection
map is a homomorphism. For the other direction, we have
\[
  (r,s)\in \Cg^\D(c,d) \vee \ker(\pi_i)
  \qquad \text{for each } i.
\]
The set $\Gamma = \{\ker(\pi_i)\mid i\in I\}$ of congruences of $\D$ is
finite since $D$ is finite. Let $\Gamma = \{ \ker(\pi_j)\mid j\in J\}$,
where $J$ is a finite subset of $I$. Then by the congruence distributivity
of $\D$,
\begin{align*}
  (r,s) \in \bigwedge_{j\in J} \left( \Cg^\D(c,d) \vee \ker(\pi_j) \right)
  & = \Cg^\D(c,d) \vee \bigwedge_{j\in J} \ker(\pi_j) \\
  & = \Cg^\D(c,d) \vee \bigwedge_{i\in I} \ker(\pi_i) 
  = \Cg^\D(c,d) \vee \0_{\D} \\
  & = \Cg^\D(c,d),
\end{align*}
as claimed.
\end{claimproof}
\end{proof}

Let $\V = \V(\A'(\T))$ and define subclasses of $\V$,
\[
  \M_i = \Mod_\V\left( \exists \overline{m} \; \left[
    e_i(\overline{m},x)\approx x\right] \right)
  \qquad \text{for} \qquad i\in \{0,1,2\}.
\]
We will make use of the fact that if $\C$ is subdirectly irreducible, then
either $\C\models \exists \overline{n} [ e_i(\overline{n},x)\approx x ]$ for
some $i\in \{0,1,2\}$ or $\C\models e_i(\overline{y},x)\approx 0$ for all
$i\in \{0,1,2\}$. This fact follows from the description of SI's in Section
\ref{sec:SIs} and from the definition of the $S_i$ and the $e_i$. In the
case where $\C\models \exists \overline{n} [ e_i(\overline{n},x)\approx x
]$, McKenzie \cite{McKenzieResidualBoundNotComp} proves that since $\C$ is
subdirectly irreducible it is necessarily small. The condition 
\[
  \text{``} x\in e_i(\overline{u},B) \text{ for some } i\in\{0,1,2\} \text{
  and some } \overline{u}\in B^2\cup B\text{''} 
\]
and it's negation will be referred to quite often in the upcoming argument,
so we now define an easier way to reference it.

\begin{defn}\label{defn:happy}
We will say that $S\subseteq B$ is \emph{unhappy} if 
\[
  \forall i\in \{0,1,2\}\; \forall \overline{u}\in B^2\cup B \left[
  S\not\subseteq e_i(\overline{u},B) \right],
\]
and $S$ is \emph{happy} otherwise. The element $s\in B$ is \emph{unhappy}
(resp. \emph{happy}) if $\{s\}$ is unhappy (resp. happy). The function
$h:B^n\to B$ is \emph{unhappy} (resp. \emph{happy}) if $\Range(h)$ is
unhappy (resp. happy). Note that these definitions depend on the algebra
$\B$, but the particular algebra something is happy or unhappy with respect
to will always be clear from the context. 
\end{defn}

Here are some useful (and straightforward) observations about happiness with
respect to an algebra $\B$.
\begin{itemize}
  \item If a set $S$ only contains unhappy elements then this is a stronger
    property than $S$ being unhappy.
  \item The operations $S_i$ for $i\in \{0,1,2\}$ are happy.
  \item If the function $h:B^n\to B$ is $0$-absorbing in the $i$th variable
    position, then it \emph{preserves happiness} in the sense that if
    $a_1,\ldots a_n\in B$ and $a_i$ is happy, then $h(a_1,\ldots,a_n)$ is
    happy as well.
  \item If $c,d\in B$, $d\leq c$, and $\{d,c\}$ is unhappy, then $c$ must be
    unhappy.
\end{itemize}

\begin{lem}\label{lem:dpsc all factors}
If $\T$ halts then each $\M_i$ has definable principal subcongruences.
\end{lem}
\begin{proof}
Let $i\in \{0,1,2\}$. We will show that $\M_i$ satisfies the hypotheses of
Lemma \ref{lem:jonsson polys dpsc} and thus has definable principal
subcongruences. Let $\B\in \M_i$. Choose $\overline{m}\in B^2\cup B$
witnessing $\B\models e_i(\overline{m},x)\approx x$. Now, $\B\models
e_i(\overline{m},x)\approx x$ if and only if $\B\models
S_i(\overline{m},x,y,z)\approx (x\wedge y) \vee (x\wedge z)$. Therefore
there exists $\overline{m}\in B^2\cup B$ such that the following 
\begin{align*}
  & p_0(x,y,z) = x, & & p_1(x,y,z) = S_i(\overline{m},x,y,z) 
    = (x\wedge y) \vee (x\wedge z), \\
  & p_2(x,y,z) = x\wedge z, & &  p_3(x,y,z) = S_i(\overline{m},z,y,x) 
    = (y\wedge z) \vee (x\wedge z), \\
  & p_4(x,y,z) = z,
\end{align*}
are polynomials of $\B$ and satisfy the J\'{o}nsson identities. If
$J_i(\overline{m})$ is the set of J\'{o}nsson identities for these
polynomials, then $\M_i\subseteq \Mod_\V(\exists \overline{m} \;
J_i(\overline{m}))$. Since $\T$ halts, $\kappa(\V(\M_i))\leq
\kappa(\A'(\T))< \omega$. By Lemma \ref{lem:jonsson polys dpsc}, it follows
that $\M_i$ has definable principal subcongruences.
\end{proof}

Let $\Gamma_0^i(w,x,y,z)$ and $\psi_0^i(w,x,y,z)$ be the congruence formulas
witnessing definable principal subcongruences for $\M_i$. Define
\begin{equation} \label{defn:psi 0}
  \psi_0(w,x,y,z) = \bigvee_{i=0}^2 \psi_0^i(w,x,y,z)
  \quad \text{and} \quad
  \Gamma_0(w,x,y,z) = \bigvee_{i=0}^2 \Gamma_0^i(w,x,y,z).
\end{equation}

\begin{thm}\label{thm:dpsc Gamma0 psi0}
The class $\M_0 \cup \M_1 \cup \M_2$ has definable principal subcongruences
witnessed by the congruence formulas $\Gamma_0$ and $\psi_0$.
\end{thm}
\begin{proof}
Since $\Gamma_0^i$ and $\psi_0^i$ are congruence formulas, so are $\Gamma_0$
and $\psi_0$. Let $\Pi_{\psi_0}(x,y)$ be the formula expressing that the
pair $(x,y)$ generates a congruence that is defined by $\psi_0(-,-,x,y)$ in
$\V(\A'(\T))$ (i.e. the formula asserting that $\psi_0(-,-,x,y)$ is an
equivalence relation, is invariant under fundamental translations, and
that $\psi_0(x,y,x,y)$ holds). Since each $\M_i$ has definable principal
subcongruences and $\Gamma_0$ and $\psi_0$ are the disjunctions of the
formulas witnessing DPSC, $\Gamma_0$ and $\psi_0$ witness definable
principal subcongruences for the class $\M_1\cup \M_2\cup \M_3$. 
\end{proof}

In symbols, Theorem \ref{thm:dpsc Gamma0 psi0} says
\[
  \M_1\cup \M_2\cup \M_3 \models \forall a,b \left[ a\neq b \rightarrow
  \exists c,d \left[ c\neq d \wedge \Gamma_0(c,d,a,b) \wedge
  \Pi_{\psi_0}(c,d) \right] \right].
\] 
In terms of happiness, Theorem \ref{thm:dpsc Gamma0 psi0} means that if
$\B$ is a happy algebra, then $\B$ has DPSC witnessed by $\Gamma_0$ and
$\psi_0$.

The next $5$ lemmas provide the groundwork for analyzing the polynomials
that make up a hypothetical Maltsev chain. Specifically, they describe the
extent to which the non-$0$-absorbing operations commute with the other
operations.

\begin{lem} \label{lem:S_i commutes 0-abs}
Each of the following hold for every algebra $\B\in \V(\A'(\T))$.
\begin{enumerate}
  \item If $f(x)$ is $0$-absorbing, then $g(x) = f(S_j(\overline{n},p,q,x))$
    is happy for all $j\in \{0,1,2\}$, $\overline{n}\in B^2\cup B$, and
    $p,q\in B$.
  \item If $f(x)$ is happy, then there is $j\in \{0,1,2\}$ and
    $\overline{n}\in B^2\cup B$ such that
    \[
      f(x) = S_j(\overline{n},f(x),f(x),f(x)).
    \]
  \item If $f(x)$ is a polynomial, $c,d\in B$, $d\leq c$, and
    $\{f(c),f(d)\}$ is happy, then there is $j\in \{0,1,2\}$ and
    $\overline{n}\in B^2\cup B$ such that the polynomial
    \[
      g(x) = S_j(\overline{n},f(c),f(d),f(x))
    \]
    satisfies $(g(c),g(d)) = (f(c),f(d))$.
\end{enumerate}
\end{lem}
\begin{proof}
We begin with $(1)$. Let $\B\leq \prod_{l\in L} \C_l$ be a subdirect
representation of $\B$ by subdirectly irreducible algebras and define
\[
  I = \{ l\in L \mid \pi_l( S_j(\overline{n},p,q,B) ) \neq \{0\} \} 
  \qquad \text{and} \qquad J = L\setminus I.
\]
Write a typical $y\in B$ as $y = (y_I, y_J)$, where $y_I = \pi_I(y)$ and
$y_J = \pi_J(y)$. Therefore $S_j(\overline{n},y,y,y) = e_j(\overline{n},y) =
(y_I,0)$, so
\begin{align*}
  g(x) 
  & = f(S_j(\overline{n},p,q,x))
  = f \mat{ (p_I\wedge q_I)\vee (p_I\wedge x_{I}) \\ 0_J } \\
  & = \mat{ f((p_I\wedge q_I)\vee (p_I\wedge x_{I})) \\ f(0_J) }
  = \mat{ f((p_I\wedge q_I)\vee (p_I\wedge x_{I})) \\ 0_J } 
  \in e_j(B)
\end{align*}

For $(2)$, say that $f(x)$ is happy because $\Range(f(x)) \subseteq
e_j(\overline{n},B)$ for some $j\in \{0,1,2\}$ and $\overline{n}\in B^2\cup
B$. Observing that $e_j(\overline{n},\B)\models e_j(\overline{n},x) \approx
x$, the conclusion follows.

For $(3)$, say (as in $(2)$) that $\{f(c),f(d)\}$ is happy because
$\{f(c),f(d)\} \subseteq e_j(\overline{n},B)$ for some $j\in \{0,1,2\}$ and
$\overline{n}\in B^2\cup B$. The operations of $\A'(\T)$ are monotonic, so
$f(d)\leq f(c)$. Therefore
\[
  g(c) 
  = S_j(\overline{n},f(c),f(d),f(c)) 
  = (f(c)\wedge f(d))\vee (f(c)\wedge f(c))
  = f(c),
\]
and likewise
\[
  g(d) 
  = S_j(\overline{n},f(c),f(d),f(d)) 
  = (f(c)\wedge f(d))\vee (f(c)\wedge f(d))
  = f(d). 
  \qedhere
\]
\end{proof}

\begin{lem} \label{lem:S_i J or K poly}
Let $\{c,d\}\subseteq B$ be happy with $d\leq c$ and let $g(x)\in
\mathcal{P}$. Then there are constants $p,q\in B$ such that
\begin{enumerate}
  \item if $\{g(c),g(d)\}$ is happy, then there is $j\in \{0,1,2\}$,
    $\overline{n}\in B^2\cup B$, and a happy polynomial $h(x)\in
    \mathcal{P}$ such that
    \[
      g'(x) = S_j(\overline{n},p,q,h(x))
    \]
    has $(g(c),g(d)) = (g'(c),g'(d))$.
  \item if $\{g(c),g(d)\}$ is unhappy, then there are some fundamental
    translations $F_1, \ldots, F_M$ and a happy polynomial $h(x)\in
    \mathcal{P}$ such that for some choice of operation $G\in \{J,J',K\}$,
    the polynomial
    \[
      g'(x) = F_M\circ\cdots\circ F_1\circ G(p,q,h(x))
    \]
    has $(g(c),g(d)) = (g'(c),g'(d))$ and the set
    \[
      \left\{ F_k\circ\cdots\circ F_1\circ G(p,q,h(c)) 
        \mid 1\leq k\leq M \right\} \cup \left\{ G(p,q,h(c)) \right\}
    \]
    contains only unhappy elements.
\end{enumerate}
\end{lem}
\begin{proof}
Item $(1)$ is a restatement of Lemma \ref{lem:S_i commutes 0-abs}.

Suppose that $\{g(c),g(d)\}$ is unhappy. The polynomial $g(x)$ is primitive,
so there are fundamental translations $F_1, \ldots, F_N$ such that $g(x) =
F_N\circ \cdots\circ F_1(x)$. Define a sequence of polynomials $g_l(x)$ and
elements $c_l, d_l$ by $g_l(x) = F_l \circ \cdots \circ F_1(x)$ and
$(c_l,d_l) = (g_l(c), g_l(d))$ with $(c_0,d_0) = (c,d)$. Choose $L$ maximal
such that $\{c_L,d_L\}$ is unhappy but $\{c_{L-1},d_{L-1}\}$ is happy. Since
$c_L = F_L(c_{L-1})$ and $d_L = F_L(c_{L-1})$, the translation $F_L$ must
map some happy elements to unhappy elements (i.e. it does not preserve
happiness). The only way this can happen is if $F_L$ is non-$0$-absorbing
(by Lemma \ref{lem:S_i commutes 0-abs}, $0$-absorbing functions preserve
happiness).

The only fundamental translations that are not $0$-absorbing are the $S_j$
in the last $2$ variables, and $J$, $J'$, and $K$ in the last variable. If
$F_L$ is such a translation of an $S_j$ operation, then it is happy, which
contradicts the unhappiness of $\{c_L, d_L\}$. Therefore
\[
  F_L(x) \in \left\{ J(p,q,x), J'(p,q,x), K(p,q,x) \right\}
\]
for some $p,q\in B$. Let the happiness of the set $\{c_{L-1}, d_{L-1}\}$ be
witnessed by $\{c_{L-1}, d_{L-1}\} \subseteq e_j(\overline{n},B)$ for some
$j\in \{0,1,2\}$ and $\overline{n}\in B^2\cup B$, and define
\[
  g'(x) = F_M\circ\cdots\circ F_L\circ h(x)
  \qquad \text{where} \qquad
  h(x) = S_j(\overline{n}, c_{L-1}, d_{L-1}, F_{L-1}\circ\cdots\circ
    F_1(x)).
\]
The polynomial $h(x)$ is clearly happy and primitive, and by the maximality
of $L$ the set
\[
  \left\{ F_k\circ\cdots\circ F_{L+1}\circ F_L(c) \mid L+1\leq k\leq N
    \right\}
  = \left\{ c_k \mid L+1\leq k\leq N \right\}
\]
contains only unhappy elements.

The only assertion left to verify is $(g'(c), g'(d)) = (g(c), g(d))$. Since
all operations are monotone and $d\leq c$, we have that $d_{L-1}\leq
c_{L-1}$.  Since $\{c_{L-1}, d_{L-1}\}\subseteq e_j(\overline{n},B)$ and
$d_{L-1}\leq c_{L-1}$,
\begin{align*}
  F_M\circ\cdots\circ F_L\circ h \mat{c \\ d}
  & = F_M\circ\cdots\circ F_L\circ S_j\left( \overline{n}, c_{L-1}, d_{L-1}, 
    F_{L-1}\circ\cdots\circ F_1 \mat{c \\ d} \right) \\
  & = F_M\circ\cdots\circ F_L\circ S_j\left( \overline{n}, c_{L-1}, d_{L-1}, 
    \mat{c_{L-1} \\ d_{L-1}} \right) \\
  & = F_M\circ\cdots\circ F_L\circ \mat{c_{L-1} \\ d_{L-1}} \\
  & = F_M\circ\cdots\circ F_L\circ F_{L-1} \circ\cdots\circ F_1 \mat{c \\ d}
  = g\mat{c \\ d}.
\end{align*}
Therefore $(g'(c), g'(d)) = (g(c), g(d))$. Reindexing $F_M, \ldots
F_{L+1}$ completes the proof of $(2)$.
\end{proof}

The next $3$ lemmas quantify the extent to which the unhappy operations $J$,
$J'$, and $K$ in the polynomial $g'(x)$ from the conclusion of Lemma
\ref{lem:S_i J or K poly} ``commute'' with other unhappy fundamental
operations.

\begin{lem} \label{lem:J commutes}
Let $F_1, \ldots, F_M$ be fundamental translations, $h(x)$ a happy primitive
polynomial, and $p,q,c,d\in B$ with $d\leq c$ such that the set
\[
  \left\{ F_k\circ\cdots\circ F_1(J(p,q,h(c))) \mid 1\leq k\leq M \right\}
  \cup \left\{ J(p,q,h(c)) \right\}
\]
contains only unhappy elements. If $g(x) = F_M\circ\cdots\circ
F_1\circ J(p,q,h(x))$ then there are constants
$p', q'\in B$ and a happy $h'(x)\in \mathcal{P}$ (actually having
$\Range(h')\subseteq e_2(p',q',B)$) such that the polynomial
\[
  g'(x) = J(p',q',h'(x))
\]
satisfies $(g(c),g(d)) = (g'(c),g'(d))$.
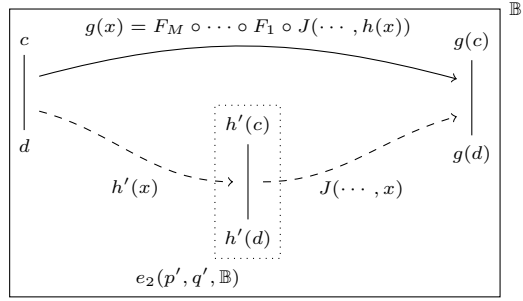
\begin{figure}[H] \centering \begin{tikzpicture}[font=\scriptsize,smooth]
  \node (c) {$c$};
  \node (d) [below=of c] {$d$};
  \node (gc) [right=15em of c] {$g(c)$};
  \node (gd) [below=of gc] {$g(d)$};
  \node (hpc) at ($(c) !0.5! (gd)$) [yshift=-1em] {$h'(c)$};
  \node (hpd) [below=of hpc] {$h'(d)$};

  \draw (c) -- (d) 
    (gc) -- (gd)
    (hpc) -- (hpd);
  \draw[->] ($(c) !1/3! (d) + (0.2,0)$) to[out=15,in=165]
    node[above]{$g(x) = F_M\circ\cdots\circ F_1\circ J(\cdots,h(x))$}
    ($(gc) !1/3! (gd) - (0.2,0)$);
  \draw[->,dashed] ($(c) !2/3! (d) + (0.2,0)$) to[out=-15,in=180]
    node[below,yshift=-0.5em]{$h'(x)$} ($(hpc) !0.5! (hpd) - (0.2,0)$);
  \draw[->,dashed] ($(hpc) !0.5! (hpd) + (0.2,0)$) to[out=0,in=195]
    node[below,yshift=-0.5em]{$J(\cdots,x)$} ($(gc) !2/3! (gd) - (0.2,0)$);
  \draw[dotted] (hpd.south west) node[below,xshift=-1em]{$e_2(p',q',\B)$} 
    rectangle (hpc.north east);
  \draw (current bounding box.north east) node[right]{$\B$} rectangle
  (current bounding box.south west);
\end{tikzpicture}
\caption{Lemma \ref{lem:J commutes} illustration.}
\end{figure}
\end{lem}
\begin{proof}
For convenience, let $r = g(c)$ and $s = g(d)$. We begin by noting that
\[
  J(x,y,z) 
  = (x\wedge y)\vee (x\wedge \partial y\wedge z)
  = (x\wedge y)\vee (x\wedge \partial y\wedge e_2(x,y,z)),
\]
from the definition of $S_2$ (recall $e_2(x,y,z) = S_2(x,y,z,z,z)$) and $J$.
Thus, it will be sufficient to prove that the polynomial $g'(x)$ in the
statement of the lemma satisfies $g'(c) = g(c)$ and $g'(d) = g(d)$ without
any restrictions on the happiness of $h'(x)$.

We will prove the lemma in the restricted setting of $M = 1$ (i.e. when
$g(x) = F\circ J(p,q,h(x))$). Repeated applications of the restricted proof
will then prove the lemma for general $M$. Say that $g(x) = F\circ
J(p,q,h(x))$, where $F$ is a fundamental translation and $g(c) = r$ is
unhappy (and thus $\{r,s\}$ is unhappy). Note that $F$ must be unhappy since
$r = F\circ J(p,q,h(c))$ is unhappy. In particular, this means that $F$
is not a translation of an $S_i$ operation.

Composing the polynomial $J(p,q,h(x))$ with translations of operations from
$\{ (\cdot), I \} \cup \mathcal{L}\cup \mathcal{R}$ produce either constant
polynomials or the composition is commutative (i.e. $F(J(p,q,h(x))) =
J(F(p),F(q),F\circ h(x))$). Thus the claim holds for these operations.

\Case{$\wedge$}
We have that $u\wedge J(p,q,h(x)) = J(p,q,h(x))\wedge u = J(p\wedge
u,q,h(x))$.

\Case{$J$}
The first translation is easy since $J(x,y,z) \wedge w = J(x\wedge w, y, z)$
and $J(x,y,z) \leq J(x,y,x)$. We have
\begin{align*}
  J(J(p,q,h(x)),u,v) 
    & = J(p,q,h(x)) \wedge J(J(p,q,p),u,v) \\
    & = J(p\wedge J(J(p,q,p),u,v),q,h(x)).
\end{align*}
For $g(x) = J(u,J(p,q,h(x)),v)$, let
\[
  g'(x) = J(r,K(r,p,q),S_2(r,K(r,p,q),r,s,g(x))),
\]
where $g(d) = s \leq r = g(c)$. Let $\B \leq \prod_{l\in L} \C_l$ be a
subdirect representation of $\B$ by subdirectly irreducible algebras. We
will show that $g'(c) = r$ and $g'(d) = s$ componentwise. We have that
\begin{multline*}
  g(x) 
  = (u\wedge p\wedge q) \vee (u\wedge p \wedge \partial q \wedge h(x)) \\
    \vee (u\wedge \partial p \wedge \partial q \wedge v) 
    \vee (u\wedge \partial p \wedge \partial\partial q 
    \wedge \partial h(x)).
\end{multline*}
The argument at this point breaks down into many subcases, depending on
whether $r(l)$ is equal to $p(l)$, $q(l)$, $\partial p(l)$, or $\partial
q(l)$ (if $r(l)\neq 0$, then by the flatness of $\C_l$ it must take on one
of these values). The easiest way to keep track of everything is with a
table. Since $r(l) = 0$ implies $g'(x)(l) = 0$ and $s(l) = 0$, we will
assume that $r(l)\neq 0$. For ease of reading, in the table below we will
omit the coordinate when giving values of functions (i.e. ``$(l)$'' will be
omitted from $r(l)$). Additionally, those coordinates which permit $r(l)
\neq s(l)$ have been indicated.
\[ \begin{array}{l|l|l|l|c}
  r                       & K(r,p,q)       & S_2(r,K(r,p,q),r,s,g(x))         & g'(x)                 & r\neq s  \\ \hline
  p = q                   & p = r          & 0                                & r\wedge r             & \text{N} \\
  p = \partial q          & q = \partial r & (r\wedge s) \vee (r \wedge g(x)) & s \vee (r\wedge g(x)) & \text{Y} \\
  \partial p = \partial q & p = \partial r & (r\wedge s) \vee (r \wedge g(x)) & s \vee (r\wedge g(x)) & \text{N} \\
  \partial p = q          & q = \partial r & (r\wedge s) \vee (r \wedge g(x)) & s \vee (r\wedge g(x)) & \text{Y} 
\end{array} \]
Since $r = g(c)$ and $s = g(d)$, we see that $g'(c)(l) = r(l)$ and $g'(d)(l)
= s(l)$, except for possibly when $r(l) = p(l) = q(l)$. In this subcase,
however, from the description of $g(x)$ in terms of $\wedge$ and $\vee$
above we see that $g(x)(l)$ is constant, so it must be that 
\[
  r(l) = g(c)(l) = g(d)(l) = s(l).
\]
Therefore $g'(c) = r$ and $g'(d) = s$, as claimed.

In the subcase where $g(x) = J(u,v,J(p,q,h(x)))$, let
\[
  g'(x) = J(u,v,S_2(u,v,r,s,g(x))).
\]
The note in the first paragraph of the proof shows that $g'(c) = g(c) = r$
and $g'(d) = g(d) = s$.

\Case{$J'$}
The first translation is similar to the case for $\wedge$, and the second
translation reduces to the $J$ case. We have
\begin{align*}
  J'(J(p,q,h(x)),u,v) 
    & = J(p,q,h(x)) \wedge J'(J(p,q,p),u,v) \\
    & = J(J'(J(p,q,p),u,v),q,h(x)), \text{ and} \\
  J'(u, J(p,q,h(x)), v) 
    & = J(J'(u,J(p,q,p),v),J(p,q,h(x)),J'(u,J(p,q,p),v))^\dagger
\end{align*}
[$\dagger$: see Case $J$ above]. For $g(x) = J'(u,v,J(p,q,h(x)))$, let
\[
  g'(x) = J(r,K(r,v,q),S_2(r,K(r,v,q),r,s,g(x))).
\]
An argument similar to the one requiring the table above will show that that
$g'(c) = g(c) = r$ and $g'(d) = g(d) = s$.

\Case{$S_i$}
Since $\{r, s\}$ is unhappy, we can exclude these translations.

\Case{$K$}
For $g(x) = K(J(p,q,h(x)),u,v)$, let
\[
  g'(x) = J(r,K(r,p,q),S_2(r,K(r,p,q),r,s,g(x))).
\]
The approach is to take a subdirect representation of $\B$ and show that
$g'(c) = r$ and $g'(d) = s$ componentwise, as in Case $J$ above. Since
\begin{multline*}
  g(x)
  = (\partial p\wedge \partial q\wedge u) 
    \vee (\partial p\wedge \partial\partial q \wedge \partial h(x) \wedge u)
    \vee (\partial p\wedge \partial q \wedge \partial u\wedge v) \\
  \vee (\partial p \wedge \partial\partial q \wedge h(x) \wedge \partial u 
    \wedge v) 
  \vee (p\wedge q\wedge u\wedge v) 
  \vee (p\wedge \partial q \wedge h(x) \wedge u \wedge v),
\end{multline*}
the $l$-th projection of the polynomial $S_2(r,K(r,p,q),r,s,g(x))$ maps
$c(l)$ to $r(l)$ and $d(l)$ to $s(l)$ unless $r(l) = (p\wedge q\wedge
u\wedge v)(l)$. From the definition of $J$, it therefore follows that $g'(c)
= r$ and $g'(d) = s$.

For the two remaining subcases where either $g(x) = K(u,J(p,q,h(x)),v)$ or
$g(x) = K(u,v,J(p,q,h(x)))$, let
\[
  g'(x) = J(r,K(r,u,q),S_2(r,K(r,u,q),r,s,g(x))).
\]
An argument similar to the previous subcase shows that $g'(c) = r$ and
$g'(d) = s$.

\Case{$T$}
Since $T(w,x,y,z) = T(y,z,w,x)$, we need only consider translations through
the first two coordinates. The equation 
\[
  T(J(p,q,h(x)),u,v,w) = T(p\wedge q, u, v, w)
\]
holds in our variety, so we move on to the subcase $g(x) =
T(u,J(p,q,h(x)),v,w)$. If $g(x) = T(u,J(p,q,h(x)),v,w)$ then
\[
  r = \left[ r\wedge (u\cdot J(p,q,h(x))) \right] \vee \left[ r\wedge
    \partial ( u\cdot J(p,q,h(x)) ) \right]
\]
($x \vee y$ is not a polynomial in our variety, but since $\A'(\T)$ is a
height $1$ semilattice, if $x,y\leq z$ then the quantity $x\vee y$ is
uniquely defined). From the above equation,
\[
  g'(x) = J(r,J(u\cdot p,u\cdot q,u\cdot h(x)),v\cdot w)
\]
has $g'(c) = r$ and $g'(d) = s$ and Case $J$ applies again.

\Case{$F\in \{ U_M^0, U_M^1 \mid M\in \mathcal{L}\cup \mathcal{R} \}$}
The only difficulty in this case arises when $g(x) =
U_M^i(u,v,w,J(p,q,h(x))$. In this subcase, let $t = K(r,M(v,w,p), M(v,w,q))$
and $g'(x) = J(r,t,S_2(r,t,r,s,g(x)))$.
\smallskip

This completes the proof of the restricted setting of $M=1$. By repeatedly
applying this argument, the lemma is proved for general $M$.
\end{proof}

The above lemma shows that for $2$ fixed inputs, the $J$ operation can be
taken to commute in a very specific way with the other fundamental
operations. The situation for $J'$ is similar, but much more complicated,
requiring a sequence of inputs and a mix of the $J$ and $J'$ operations.

\begin{lem} \label{lem:J' commutes}
Let $F_1, \ldots, F_M$ be fundamental translations, $h(x)$ a happy primitive
polynomial, and $p,q,c,d\in B$ with $d\leq c$ such that the set
\[
  \left\{ F_k\circ\cdots\circ F_1(J'(p,q,h(c))) \mid 1\leq k\leq M \right\}
  \cup \{ J'(p,q,h(c)) \}
\]
contains only unhappy elements. If $g(x) = F_M\circ\cdots\circ F_1\circ
J'(p,q,h(x))$ and $(r,s) = (g(c), g(d))$, then there is a decreasing Maltsev
chain $g(c) = r_1, r_2,\ldots, r_n = g(d)$ connecting $g(c)$ to $g(d)$ with
associated polynomials $g_1(x), \ldots, g_{n-1}(x)$ of the form
\[
  g_k(x) = G_k(p_k,q_k,h_k(x)),
  \quad \text{ where } \quad
  G_k\in \{J,J'\},\ p_k,q_k\in B,\ h_k(x)\in\mathcal{P} \text{ happy}.
\]
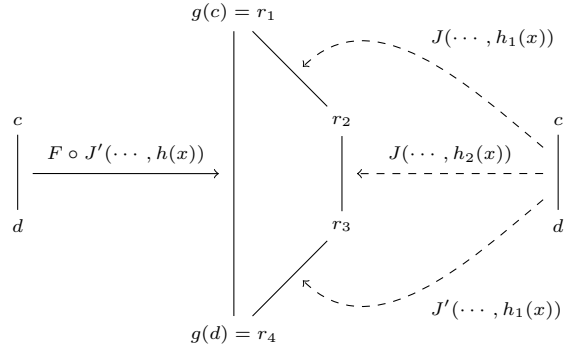
\begin{figure}[H] \centering \begin{tikzpicture}[font=\scriptsize,smooth]
  \node (r1) {$g(c) = r_1$};
  \node (r2) [below=of r1,xshift=4em] {$r_2$};
  \node (r3) [below=of r2] {$r_3$};
  \node (r4) [below=of r3, xshift=-4em] {$g(d) = r_4$};

  \coordinate (t1) at ($(r1) !0.5! (r2)$);
  \coordinate (t1) at (t1|-r2);
  \node (c1) at (t1) [xshift=-10em] {$c$};
  \node (d1) [below=of c1] {$d$};
  \node (c2) at (t1) [xshift=10em] {$c$};
  \node (d2) [below=of c2] {$d$};

  \draw (c1) -- (d1)
    (c2) -- (d2)
    (r1) -- (r2) -- (r3) -- (r4) -- (r1);
  \draw[->] ($(c1) !0.5! (d1) + (0.2,0)$) 
    -- node[above]{$F\circ J'(\cdots,h(x))$} ($(r1) !0.5! (r4) - (0.2,0)$);
  \draw[->,dashed] ($(c2) !1/4! (d2) - (0.2,0)$) to[out=140,in=45]
    node[above right]{$J(\cdots,h_1(x))$} ($(r1) !0.5! (r2) + (0.2,0)$);
  \draw[->,dashed] ($(c2) !0.5! (d2) - (0.2,0)$) --
    node[above]{$J(\cdots,h_2(x))$} ($(r2) !0.5! (r3) + (0.2,0)$);
  \draw[->,dashed] ($(c2) !3/4! (d2) - (0.2,0)$) to[out=220,in=-45]
    node[below right]{$J'(\cdots,h_1(x))$} ($(r3) !0.5! (r4) + (0.2,0)$);
\end{tikzpicture}
\caption{Lemma \ref{lem:J' commutes} illustration.}
\end{figure}
\end{lem}
\begin{proof}
For convenience, let $r = g(c)$ and $s = g(d)$. As in Lemma \ref{lem:J
commutes}, we will prove the claim in the restricted setting of $M=1$ (i.e.
when $g(x) = F\circ J'(p,q,h(x))$). Repeated applications on the proof in
the restricted setting and of Lemma \ref{lem:J commutes} will prove the
lemma for general $M$. Say that $g(x) = F\circ J'(p,q,h(x)))$, where $F$ is
a fundamental translation and $g(c) = r$ is unhappy (and thus $\{r,s\}$ is
unhappy). Note that $F$ must be unhappy since $r = F\circ J'(p,q,h(c)))$ is
unhappy. In particular, this means that $F$ is not a translation of an $S_i$
operation.

Composing the polynomial $J'(p,q,h(x))$ with translations of operations from
$\{ (\cdot), I \} \cup \mathcal{L}\cup \mathcal{R}$ produce either constant
polynomials or the composition is commutative (i.e.  $F\circ J'(p,q,h(x))) =
J'(F(p),F(q),F\circ h(x))$). Since these operations are $0$-absorbing, they
are happiness preserving, and the claim holds for them.

\Case{$\wedge$}
We have that $u\wedge J'(p,q,h(x)) = J'(p,q,h(x))\wedge u = J'(p\wedge
u,q,h(x))$.

\Case{$J$}
The first translation is easy since $J(x,y,z) \wedge w = J(x\wedge w, y, z)$
and $J'(x,y,z) \leq J'(x,y,x)$. We have
\begin{align*}
  J(J'(p,q,h(x)),u,v) 
    & = J'(p,q,h(x)) \wedge J(J'(p,q,p),u,v) \\
    & = J'(p\wedge J(J'(p,q,p),u,v),q,h(x)).
\end{align*}
For $g(x) = J(u,J'(p,q,h(x)),v)$ we must introduce a new ``link'' in our
Maltsev chain. Let
\begin{align*}
  g_1(x) & = J(r,K(r,p,q),S_2(r,K(r,p,q),r,s,g(x))) & & \text{and} \\
  g_2(x) & = J'(t_1,K(t_1,p,q),h(x)), & & \text{where } t_1 = g_1(d)
\end{align*}
(recall that $r = g(c)$ and $s = g(d)$). We have that
\begin{multline*}
  g(x)
  = (u\wedge p\wedge q\wedge h(x) \wedge v) 
    \vee (u\wedge p\wedge \partial q\wedge v) \\
    \vee (u\wedge \partial p\wedge \partial q\wedge \partial h(x))
    \vee (u\wedge \partial p\wedge\wedge q).
\end{multline*}
The argument at this point breaks down into many subcases, depending on
whether $r(l)$ is equal to $p(l)$, $q(l)$, $\partial p(l)$, or $\partial
q(l)$ (if $r(l)\neq 0$, then by the flatness of $\C_l$ it must take on one
of these values). The easiest way to keep track of everything is with a
table. Since $r(l) = 0$ implies $g'(x)(l) = 0$ and $s(l) = 0$, we will
assume that $r(l)\neq 0$. For ease of reading, in the table below we will
omit the coordinate when giving values of functions (i.e. ``$(l)$'' will be
omitted from $r(l)$). Additionally, those coordinates which permit $r(l)
\neq s(l)$ have been indicated.
\[ \begin{array}{l|l|l|l|c}
  r                       & K(r,p,q)       & S_2(r,K(r,p,q),r,s,g(x))         & g_1(x)                & r\neq s  \\ \hline
  p = q                   & p = r          & 0                                & r\wedge r             & \text{Y} \\
  p = \partial q          & q = \partial r & (r\wedge s) \vee (r \wedge g(x)) & s \vee (r\wedge g(x)) & \text{N} \\
  \partial p = \partial q & p = \partial r & (r\wedge s) \vee (r \wedge g(x)) & s \vee (r\wedge g(x)) & \text{Y} \\
  \partial p = q          & q = \partial r & (r\wedge s) \vee (r \wedge g(x)) & s \vee (r\wedge g(x)) & \text{N} 
\end{array} \]
Since $r = g(c)$ and $s = g(d)$, we have that $g_1(c) = r$, and $t_1(l) =
g_1(d)(l) = s(l)$ except for possibly when $r(l) = p(l) = q(l)$. It follows
(by flatness) that $s \leq t_1 \leq r$. We will now show (with another
similar table) that $g_2(c) = t_1$ and $g_2(d) = s$. The first column of the
table below corresponds to the $2$nd-to-last column of the table above
evaluated at $x = d$.
\[ \begin{array}{l|l|l|c}
  t_1 = g_1(d)           & K(t_1,p,q)                          & g_2(x)         & t_1\neq s \\ \hline
  r = p = q              & p = t_1                             & t_1\wedge h(x) & \text{Y}  \\
  r = s = p = \partial q & q = \partial t_1                    & t_1            & \text{N}  \\
  s                                       & (p\wedge \partial s) = \partial t_1 & t_1            & \text{N}  \\
  r = s = \partial p = q & q = \partial t_1                    & t_1            & \text{N}
\end{array} \]
From the table we can see that $g_2(c)(l) = t_1(l)$ in all subcases except for
possibly when $r(l) = p(l) = q(l)$. In this event, from the definition of
$g(x)$ we have that $r(l) = h(c)(l)$, so $g_2(c)(l) = t_1(l)$ (the previous
table indicates that $t_1(l) = r(l)$ when $r(l) = p(l) = q(l)$). Therefore
$g_2(c) = t_1$. Since $t_1(l)$ differs from $s(l)$ only when $r(l) = p(l) =
q(l)$, and since in this subcase $h(d)(l) = s(l)$ (from the definition of
$g(x)$ at the start of the case), it follows that $g_2(d) = s$.

In the case where $g(x) = J(u,v,J'(p,q,h(x)))$, let
\[
  g_1(x) = J(u,v,S_2(u,v,r,s,h(x))).
\]
An argument similar to the case for $J(\cdots,J(\cdots,h(x)))$ in Lemma
\ref{lem:J commutes} will show that $g_1(c) = g(c) = r$ and $g_1(d) = g(d) =
s$.

\Case{$J'$}
We have
\begin{align*}
  J'(J'(p,q,h(x)),u,v) 
  & = J'(p,q,h(x)) \wedge J'(J(p,q,p),u,v) \\
  & = J'(J'(J(p,q,p),u,v),q,h(x)).
\end{align*}
For $g(x) = J'(u,J'(p,q,h(x)),v)$ we must again introduce a new ``link'' in
our Maltsev chain. Let
\begin{align*}
  g_1(x) & = J(r,K(r,p,q),S_2(r,K(r,p,q),r,s,g(x))) & &  \text{and} \\
  g_2(x) & = J'(t_1,K(t_1,p,q),h(x)), & & \text{where } t_1 = g_1(d).
\end{align*}
An argument similar to the corresponding subcase of Case $J$ will show that
$g_1(c) = r$, $g_1(d) = g_2(c) = t_1$, and $g_2(d) = s$. For $g(x) =
J'(u,v,J'(p,q,h(x)))$, let
\[
  g_1(x) = J'(r,K(r,v,q),h(x)).
\]
An argument similar to the one at the start of Case $J$ above will show that
$g_1(c) = r$ and $g_2(d) = s$.

\Case{$S_i$}
Since $\{r,s\}$ is unhappy, we can exclude these translations.

\Case{$K$}
For $g(x) = K(J'(p,q,h(x)),u,v)$, let
\begin{align*}
  g_1(x) & = J(r,K(r,p,q),S_2(r,K(r,p,q),r,s,g(x))) & \text{and} && \\
  g_2(x) & = J'(t_1,K(t_1,p,q),h(x)), & \text{where } t_1 = g_1(d).
\end{align*}
An argument similar to the one in Case $J$ requiring the tables shows that
$g_1(c) = r$, $g_1(d) = g_2(c) = t_1$, and $g_2(d) = s$. For the two
remaining subcases where we have either $g(x) = K(u,J'(p,q,h(x)),v)$ or
$g(x) = K(u,v,J'(p,q,h(x)))$, let
\begin{align*}
  g_1(x) & = J(r,K(r,u,p),S_2(r,K(r,u,q),r,s,g(x))) & \text{and} && \\
  g_2(x) & = J'(t_1,K(t_1,u,q),h(x)), & \text{where } t_1 = g_1(d).
\end{align*}
An argument similar to the one using the tables above will show that $g_1(c)
= r$, $g_1(d) = g_2(c) = t_1$, and $g_2(d) = s$.

\Case{$T$}
Since $T(w,x,y,z) = T(y,z,w,x)$, we need only consider translations through
the first two coordinates. If $g(x) = T(J'(p,q,h(x)),u,v,w)$, then 
\[
  r = \left[ r\wedge (J'(p,q,h(x))\cdot u) \right] \vee \left[ r\wedge
  \partial ( J'(p,q,h(x))\cdot u ) \right].
\]
Therefore
\[
  g'(x) = J'(r,J'(p\cdot u,q\cdot u,h(x)\cdot u),v\cdot w)
\]
has $g(c) = r$ and $g(d) = s$ and Case $J'$ applies. Similarly if we have
$g(x) = T(u,J'(p,q,h(x)),v,w)$, then 
\[
  r = \left[ r\wedge (u\cdot J'(p,q,h(x))) \right] \vee \left[ r\wedge
  \partial ( u\cdot J'(p,q,h(x)) ) \right].
\]
Therefore
\[
  g'(x) = J'(r,J'(u\cdot p,u\cdot q,u\cdot h(x)),v\cdot w)
\]
has $g(c) = r$ and $g(d) = s$ and Case $J'$ applies again.

\Case{$F\in \{U_M^0, U_M^1 \mid M\in \mathcal{L}\cup \mathcal{R} \}$}
If $g(x) = U_F^i(u,v,w,J'(p,q,h(x)))$, then let
\[
  g'(x) = J'(r,U_F^i(u,v,w,q),U_F^i(u,v,w,h(x)));
\]
if $g(x) = U_F^i(u,v,J'(p,q,h(x)),w)$, then let
\[
  g'(x) = J'(r,U_F^i(u,v,p,w), U_F^i(u,v,h(x),w));
\]
if $g(x) = U_F^i(u,J'(p,q,h(x)),v,w)$, then let
\[
  g'(x) = J'(r,U_F^i(u,p,v,w),U_F^i(u,h(x),v,w));
\]
if $g(x) = U_F^i(J'(p,q,h(x)),u,v,w)$, then let
\[
  g'(x) = J'(r,U_F^i(p,u,v,w),U_F^i(h(x),u,v,w)).
\]
\smallskip

This completes the proof of the restricted setting of $M=1$. By repeatedly
applying this argument and Lemma \ref{lem:J commutes}, the lemma is proved
for general $M$.
\end{proof}

In the next lemma, we see that the $K$ operation behaves essentially the
same as the $J'$ operation.

\begin{lem} \label{lem:K commutes}
Let $F_1, \ldots, F_M$ be fundamental translations, $h(x)$ a happy primitive
polynomial, and $p,q,c,d\in B$ with $d\leq c$ such that the set
\[
  \left\{ F_k\circ\cdots\circ F_1(K(p,q,h(c))) \mid 1\leq k\leq M, \right\}
  \cup \{ K(p,q,h(c)) \}
\]
contains only unhappy elements. If $g(x) = F_M\circ\cdots\circ
F_1\circ K(p,q,h(x))$, then there is a decreasing
Maltsev chain $g(c) = r_1, r_2,\ldots, r_n = g(d)$ connecting $g(c)$ to
$g(d)$ with associated polynomials $g_1(x), \ldots, g_{n-1}(x)$ of the form
\[
  g_k(x) = G_k(p_k,q_k,h_k(x)),
  \quad \text{ where } \quad
  G_k\in \{J,J'\},\ p_k,q_k\in B,\ h_k(x)\in\mathcal{P} \text{ happy}.
\]
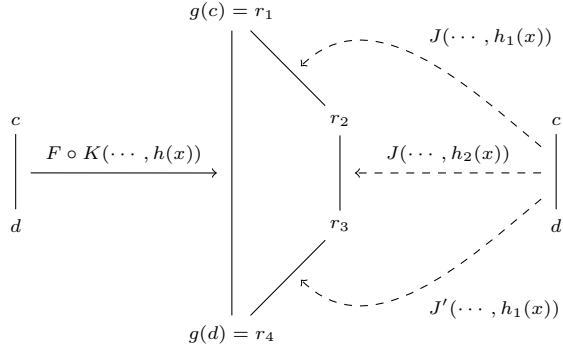
\begin{figure}[H] \centering \begin{tikzpicture}[font=\scriptsize,smooth]
  \node (r1) {$g(c) = r_1$};
  \node (r2) [below=of r1,xshift=4em] {$r_2$};
  \node (r3) [below=of r2] {$r_3$};
  \node (r4) [below=of r3, xshift=-4em] {$g(d) = r_4$};

  \coordinate (t1) at ($(r1) !0.5! (r2)$);
  \coordinate (t1) at (t1|-r2);
  \node (c1) at (t1) [xshift=-10em] {$c$};
  \node (d1) [below=of c1] {$d$};
  \node (c2) at (t1) [xshift=10em] {$c$};
  \node (d2) [below=of c2] {$d$};

  \draw (c1) -- (d1)
    (c2) -- (d2)
    (r1) -- (r2) -- (r3) -- (r4) -- (r1);
  \draw[->] ($(c1) !0.5! (d1) + (0.2,0)$) 
    -- node[above]{$F\circ K(\cdots,h(x))$} ($(r1) !0.5! (r4) - (0.2,0)$);
  \draw[->,dashed] ($(c2) !1/4! (d2) - (0.2,0)$) to[out=140,in=45]
    node[above right]{$J(\cdots,h_1(x))$} ($(r1) !0.5! (r2) + (0.2,0)$);
  \draw[->,dashed] ($(c2) !0.5! (d2) - (0.2,0)$) --
    node[above]{$J(\cdots,h_2(x))$} ($(r2) !0.5! (r3) + (0.2,0)$);
  \draw[->,dashed] ($(c2) !3/4! (d2) - (0.2,0)$) to[out=220,in=-45]
    node[below right]{$J'(\cdots,h_1(x))$} ($(r3) !0.5! (r4) + (0.2,0)$);
\end{tikzpicture}
\caption{Lemma \ref{lem:K commutes} illustration.}
\end{figure}
\end{lem}
\begin{proof}
Let $g'(x) = K(p,q,h(x))$, where $p,q$ and $h(x)$ are as in the hypotheses
of the lemma, and let $r = g'(c)$ and $s = g'(d)$. Define
\begin{align*}
  f_1(x) & = J(r,K(r,p,q),S_2(r,K(r,p,q),r,s,h(x))) & & \text{and} \\
  f_2(x) & = J'(t_1,K(t_1,p,q),h(x)), & & \text{where } t_1 = f_1(d).
\end{align*}
We will show that $r = f_1(c)$, $t_1 = f_2(c)$, and $s = f_2(d)$. Since $s$
is unhappy and $s\leq t_1\leq r$, the elements $s$, $t_1$, and $r$ are all
unhappy. Using this fact and Lemmas \ref{lem:J commutes} and \ref{lem:J'
commutes}, the conclusion will follow.

Let $l\in L$. The proof breaks into cases depending on whether $r(l) = 0$.
Overall, it is useful to note that $J$ and $J'$ are $0$-absorbing in the
first and second variables, that $s\leq t_1\leq r$, and if $r(l)\neq 0$,
then $r(l) = h(c)(l)$ or $r(l) = \partial h(c)(l)$.

If $r(l) = 0$ then $s(l) = 0$ and $t_1(l) = 0$. Thus $r(l) = 0 = f_1(x)(l) =
f_1(c)(l)$ and $t_1(l) = s(l) = 0 = f_2(x)(l) = f_2(c)(l) = f_2(d)(l)$. If
$r(l)\neq 0$, then by the definition of $K$, either $p(l) = \partial q(l)$,
$p(l) = q(l) = \partial r(l)$, or $p(l) = q(l) = r(l)$. In each of these
cases, the equations $f_1(c)(l) = r(l)$, $t_1(l) = f_2(c)(l)$, and $s(l) =
f_2(d)(l)$ are easily verified from the definitions.
\end{proof}

\begin{lem} \label{lem:e_i Maltsev sequence}
Let $c,d\in B$ be such that $d\leq c$ and $\{c,d\}$ is happy. Suppose that
$(r,s)\in \Cg^{\B}(c,d)$ with $s \leq r$ is witnessed by the decreasing
Maltsev sequence $r = u_1, \ldots, u_n = s$ with associated primitive
polynomials $\lambda_1(x), \ldots, \lambda_{n-1}(x)$. Then there is another
decreasing Maltsev sequence, $r = t_1, \ldots, t_m = s$, with associated
primitive polynomials $g_1(x), \ldots, g_{m-1}(x)$ such that for each $k\in
\{1,\ldots,m-1\}$, one of
\begin{enumerate}
  \item $g_k(x)$ is happy,
  \item $g_k(x) = J(t_k,q_k,h_k(x))$ and $h_k(x)\in\mathcal{P}$ is happy, or
  \item $g_k(x) = J'(t_k,q_k,h_k(x))$ and $h_k(x)\in\mathcal{P}$ is happy
\end{enumerate}
holds for some constants $q_k\in B$.
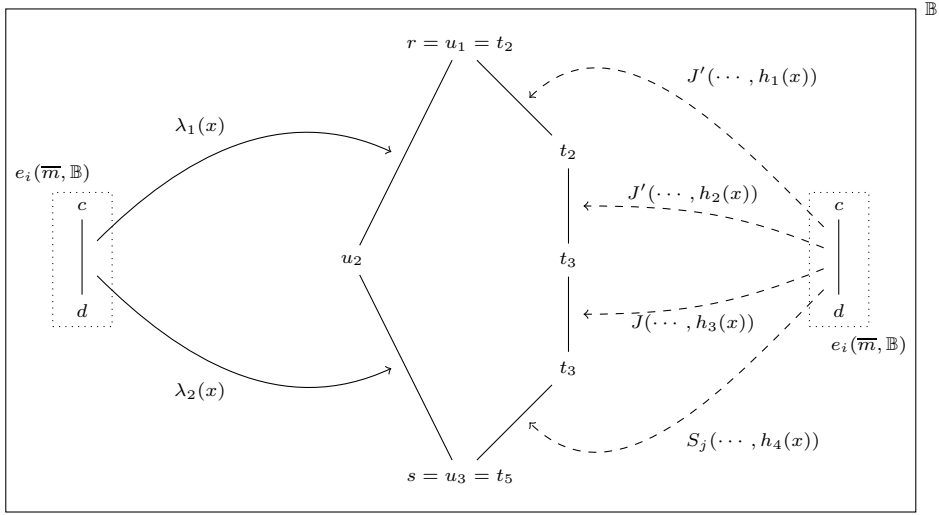
\begin{figure}[H] \centering \begin{tikzpicture}[font=\scriptsize,smooth]
  \node (r) {$r = u_1 = t_2$};
  \node (t2) [below=of r,xshift=4em] {$t_2$};
  \node (t3) [below=of t2] {$t_3$};
  \node (t4) [below=of t3] {$t_3$};
  \node (s) [below=of t4,xshift=-4em] {$s = u_3 = t_5$};
  \node (u2) at (t3) [xshift=-8em] {$u_2$};
  
  \coordinate (t1) at ($0.5*(t2)+0.5*(t3)$);
  \coordinate (t1) at (r|-t1);
  \node (c1) at (t1) [xshift=-14em] {$c$};
  \node (d1) [below=of c1] {$d$};
  \node (c2) at (t1) [xshift=14em] {$c$};
  \node (d2) [below=of c2] {$d$};

  \draw (r) -- (t2) -- (t3) -- (t4) -- (s) -- (u2) -- (r)
    (c1) -- (d1)
    (c2) -- (d2);
  \draw[->] ($(c1) !1/3! (d1) + (0.2,0)$) to[out=45,in=155] 
    node[above left]{$\lambda_1(x)$} ($(r) !0.5! (u2) - (0.2,0)$);
  \draw[->] ($(c1) !2/3! (d1) + (0.2,0)$) to[out=-45,in=205] 
    node[below left]{$\lambda_2(x)$} ($(s) !0.5! (u2) - (0.2,0)$);
  
  \draw[->,dashed] ($(c2) !1/5! (d2) - (0.2,0)$) to[out=135,in=45]
    node[above right]{$J'(\cdots,h_1(x))$} ($(r) !0.5! (t2) + (0.2,0)$);
  \draw[->,dashed] ($(c2) !2/5! (d2) - (0.2,0)$) to[out=158,in=0]
    node[above, xshift=-0.5em]{$J'(\cdots,h_2(x))$} 
    ($(t2) !0.5! (t3) + (0.2,0)$);
  \draw[->,dashed] ($(c2) !3/5! (d2) - (0.2,0)$) to[out=202,in=0]
    node[below, xshift=-0.5em]{$J(\cdots,h_3(x))$} 
    ($(t3) !0.5! (t4) + (0.2,0)$);
  \draw[->,dashed] ($(c2) !4/5! (d2) - (0.2,0)$) to[out=225,in=-45]
    node[below right]{$S_j(\cdots,h_4(x))$} ($(t4) !0.5! (s) + (0.2,0)$);

  \draw[dotted] ($(c1.north west)-(0.2,0)$) node[above]{$e_i(\overline{m},\B)$}
    rectangle ($(d1.south east)+(0.2,0)$);
  \draw[dotted] ($(c2.north west)-(0.2,0)$) rectangle 
    ($(d2.south east)+(0.2,0)$) node[below]{$e_i(\overline{m},\B)$};
  \draw (current bounding box.north east) node[right]{$\B$}
    rectangle (current bounding box.south west);
\end{tikzpicture}
\caption{Lemma \ref{lem:e_i Maltsev sequence} illustration.}
\end{figure}
\end{lem}
\begin{proof}
Select a consecutive pair, $u_k$ and $u_{k+1}$ from the Maltsev sequence. We
will show that the claim holds for the pair, and by applying the argument to
each consecutive pair, it therefore must hold for the entire sequence. By
Lemma \ref{lem:S_i J or K poly}, we can assume that one of the following is
true:
\begin{enumerate}
  \item $\{u_k, u_{k+1}\}$ is happy, so there is a happy primitive $g_k(x)$
    with $(u_k,u_{k+1}) = (g_k(c), g_k(d))$, or
  \item $\{u_k, u_{k+1}\}$ is unhappy, so there are fundamental translations
    $F_1, \ldots, F_M$ and a happy polynomial $h(x)\in \mathcal{P}$ such
    that for some $G\in \{J,J',K\}$ and some $p,q\in B$ the polynomial
    \[
      g'_k(x) = F_m\circ\cdots\circ F_1\circ G(p,q,h(x))
    \]
    has $(u_k,u_{k+1}) = (g'_k(c),g'_k(d))$ and the set
    \[
      \{ F_k\circ\cdots\circ F_1\circ G(p,q,h(c)) \mid 1\leq k\leq M \}
      \cup \{ G(p,1,h(c)) \}
    \]
    contains only unhappy elements.
\end{enumerate}
In the first possibility, we are done. In the second possibility, apply
Lemma \ref{lem:J commutes} (if $G=J$), \ref{lem:J' commutes} (if $G=J'$), or
\ref{lem:K commutes} (if $G=K$) to get a decreasing Maltsev sequence $u_k =
t_k, t_{k+1}, \ldots, t_{k+m'} = u_{k+1}$ with associated primitive
polynomials $g_k(x), \ldots, g_{k+m'-1}(x)$ such that for all $l\in
\{k,\ldots,k+m'-1\}$, 
\[
  g_l(x) = G_l(p_l,q_l,h_l(x)) 
  \qquad \text{where} \qquad
  G_l\in \{J, J'\},\ p_l,q_l\in B,\ h_l(x)\in \mathcal{P} \text{ happy}.
\]
This is almost the conclusion of the lemma. To finish, we observe that if
$f(x) = G(p,q,h(x))$ is a polynomial with $G\in \{J,J'\}$ and $f(d) \leq
f(c)$, then
\[
  f(c) = G(f(c), q, h(c)) 
  \qquad \text{and} \qquad 
  f(d) = G(f(c), q, h(d)).
\]
Applying this observation to the $g_l(x)$ and using the fact that $t_k,
\ldots, t_{k+m'}$ is a decreasing sequence completes the proof.
\end{proof}

At this point, we have established the tools necessary to transform general
decreasing Maltsev chains into longer chains whose associated polynomials
are of a very specific form.  Now, we move to on to show that these longer
chains can be shortened and come in just $7$ types, and that these $7$
different types of chains are definable. The following definition simplifies
the discussion.

\begin{defn} \label{defn:Chain}
Let $r_1, \ldots, r_n\in B$ be a sequence of elements. We
write 
\[
  r_1 \Chain{F_1} r_2 \Chain{F_2} r_3 \cdots r_{n-1} \Chain{F_{n-1}} r_n
\]
for $F_i\in \{J,J',S_0,S_1,S_2\}$ if both of the following hold
\begin{enumerate}
  \item if $F_i \in \{J,J'\}$, then there exist constants $p_i,q_i\in B$ and
    $\overline{n}_i\in B^2\cup B$ such that
    \[
      r_i = F_i(p_i,q_i,e_{j_i}(\overline{n}_i,r_i)) 
      \qquad \text{and} \qquad
      r_{i+1} = F_i(p_i,q_i,e_{j_i}(\overline{n}_i,r_{i+1}))
    \]
    for some $j_i\in \{0,1,2\}$, and
  \item if $F_i\in \{S_0,S_1,S_2\}$, then there exists $\overline{n}_i\in
    B^2\cup B$ such that
    \[
      r_i = F_i(\overline{n}_i,r_i,r_i,r_i)
      \qquad \text{and} \qquad
      r_{i+1} = F_i(\overline{n}_i,r_{i+1},r_{i+1},r_{i+1}).
    \]
\end{enumerate}
Such a sequence will be referred to as an
\emph{$F_1$-$F_2$-$\cdots$-$F_{n-1}$ chain}. If it is the case that for all
$i$, $(r_i,r_{i+1})\in \Cg^\B(c,d)$ for some $c,d\in B$, then we will say
that $(r_1,r_n)\in \Cg^\B(c,d)$ is \emph{witnessed} by an
$F_1$-$\ldots$-$F_{n-1}$ chain.
\end{defn}

\begin{lem} \label{lem:S_i defable}
Let $c,d\in e_i(\overline{m},B)$ for some $i\in \{0,1,2\}$ and
$\overline{m}\in B^2\cup B$ and assume that the congruence formula
$\psi(-,-,c,d)$ defines $\Cg^{e_i(\overline{m},\B)}(c,d)$ in
$e_i(\overline{m},\B)$. Suppose that $r, s\in e_j(\overline{n},B)$ for some
$j\in \{0,1,2\}$ and $\overline{n}\in B^2\cup B$ with $s\leq r$. Then
$(r,s)\in \Cg^{\B}(c,d)$ if and only if
\[
  \B \models \psi(e_i(\overline{m},r), e_i(\overline{m},s), c, d)
\]
and $r = S_j(\overline{n},r,s,e_i(\overline{m},r))$ and $s =
S_j(\overline{n},r,s,e_i(\overline{m},s))$.
\end{lem}
\begin{proof}
Suppose first that $(r,s)\in \Cg^{\B}(c,d)$ and let $\B\leq \prod_{l\in L}
\C_l$ be a subdirect representation of $\B$ by subdirectly irreducible
algebras in $\V(\A'(\T))$. Define
\[
  I = \{ l\in L \mid e_i(\overline{m},B)(l)\neq \{0\} \} 
  \qquad \text{and} \qquad 
  J = L\setminus I,
\]
and write a typical element $x\in B$ as $x = (x_I, x_J)$, where $x_I\in
\pi_I(B)$ and $x_J\in \pi_J(B)$. Since $c,d\in e_i(\overline{m},B)$, we have
$c = (c_I,0_J)$ and $d = (d_I,0_J)$. Hence, if $(r,s)\in \Cg^{\B}(c,d)$, it
must be that $r = (r_I,z_J)$ and $s = (s_I,z_J)$ (i.e. $\pi_J(r) =
\pi_J(s)$).

From the definition of $S_i$, we have that $e_i(\overline{m},-)$ is a
homomorphism from $\B$ to $e_i(\overline{m},\B)$. Therefore
$(e_i(\overline{m},r), e_i(\overline{m},s))\in
\Cg^{e_i(\overline{m},\B)}(c,d)$, and
\[
  \B \models \psi(e_i(\overline{m},r), e_i(\overline{m},s), c, d),
\]
since $\psi$ is existentially quantified (it is a congruence formula) and
$e_i(\overline{m},\B)\leq \B$. Since $r\in e_j(\overline{n},B)$, if $t\leq
r$ then $t\in e_j(\overline{n},B)$. Therefore 
\[
  \{ e_i(\overline{m},r), e_i(\overline{m},s) \} 
  \subseteq e_i(\overline{m},e_j(\overline{n},B)) 
  \subseteq e_i(\overline{n},B).
\]
It follows that
\begin{align*}
  S_j(\overline{n}, r, s, e_i(\overline{m},r))
    & = S_j\left( \overline{n}, \mat{r_I \\ z_J}, \mat{s_I \\ z_J}, 
      \mat{r_I \\ 0} \right) \\
    & = \left( \mat{r_I \\ z_J} \wedge \mat{s_I \\ z_J} \right) 
      \vee \left( \mat{r_I \\ z_J} \wedge \mat{r_I \\ 0} \right)
    = \mat{s_I \\ z_J} \vee \mat{r_I \\ 0} \\
    & = \mat{r_I \\ z_J}
    = r, \text{ and likewise} \\
  S_j(\overline{n}, r, s, e_i(\overline{m},s))
    & = S_j \left(\overline{n}, \mat{r_I \\ z_J}, \mat{s_I \\ z_J},
      \mat{s_I \\ 0} \right) \\
    & = \left( \mat{r_I \\ z_J} \wedge \mat{s_I \\ z_J} \right) 
      \vee \left( \mat{r_I \\ z_J}\wedge \mat{s_I \\ 0} \right)
    = \mat{s_I \\ z_J} \vee \mat{s_I \\ 0} \\
    & = \mat{s_I \\ z_J}
    = s,
\end{align*}
completing the forward direction.

Suppose now that $\B \models \psi(e_i(\overline{m},r), e_i(\overline{m},s),
c, d)$, and $r = S_j(\overline{n},r,s,e_i(\overline{m},r))$ and $s =
S_j(\overline{n},r,s,e_i(\overline{m},s))$. Since $\psi$ is a congruence
formula and $e_i(\overline{m},-)$ is a homomorphism from $\B$ to
$e_i(\overline{m},\B)$ and $c,d\in e_i(\overline{m},B)$, we have
\[
  e_i(\overline{m},\B)\models \psi(e_i(\overline{m},r), e_i(\overline{m},s),
  c, d).
\]
Thus, $(e_i(\overline{m},r), e_i(\overline{m},s))\in
\Cg^{e_i(\overline{m},\B)}(c,d)\subseteq \Cg^{\B}(c,d)$. By hypothesis, we
also have that $r = S_j(\overline{n},r,s,e_i(\overline{m},r))$ and $s =
S_j(\overline{n},r,s,e_i(\overline{m},s))$, so it follows that $(r,s)\in
\Cg^{\B}(c,d)$.
\end{proof}

In light of the above lemma, define
\begin{equation} \label{defn:psi S} \begin{aligned}
  \psi_S(w,x,y,z) 
  = \bigvee_{i=0}^2 \bigvee_{j=0}^2 \exists \overline{m}, \overline{n} 
    & \left[ y = e_i(\overline{m},y) \wedge z = e_i(\overline{m},z) \right.  \\
  & \quad \left. \wedge \psi_0(e_i(\overline{m},w),e_i(\overline{m},x),y,z) \right. \\
  & \quad \left. \wedge\, w = S_j(\overline{n},w,x,e_i(\overline{m},w)) \right. \\
  & \quad \left. \wedge x = S_j(\overline{n},w,x,e_i(\overline{m},x)) \right]
\end{aligned} \end{equation}
(recall that $\psi_0$ was defined in \eqref{defn:psi 0}). If $c, d, r, s$
satisfy the hypotheses of Lemma \ref{lem:S_i defable}, then $(r,s)\in
\Cg^{\B}(c,d)$ if and only if $\B\models \psi_S(r,s,c,d)$. That is, if
$\{c,d\}$ and $\{r,s\}$ are happy with $d\leq c$ and $s\leq r$, then
$(r,s)\in \Cg^{\B}(c,d)$ if and only if $\B\models \psi_S(r,s,c,d)$.

\begin{lem} \label{lem:J defable}
Suppose that $(r,s)\in \Cg^{\B}(c,d)$ for some $c,d\in  B$ and that there is
a decreasing sequence $r = r_1 \geq r_2\geq \ldots \geq r_n = s$ and
some constants $p_1, q_1, \ldots, p_{n-1}, q_{n-1}\in B$ such that
\[
  r_i = J(p_i, q_i, r_i)
  \qquad \text{and} \qquad
  r_{i+1} = J(p_i,q_i,r_{i+1})
\]
for $1\leq i\leq n-1$. Then there exists a constant $\rho\in B$ such that $r
= J(r,\rho,r')$ and $s = J(r,\rho,s')$, where $r' = e_2(r,\rho,r)$ and
$s' = e_2(r,\rho,s)$.

Since $J(x,y,z) = J(x,y,e_2(x,y,z))$ (from the definition of $J$), this is
equivalent to the assertion that for each decreasing $J$-$J$-$\ldots$-$J$
chain (of any length), there is a (length 1) $J$ chain with the same
endpoints.
\begin{figure}[H] \centering \begin{tikzpicture}[font=\scriptsize,smooth]
  \node (r)  {$r=r_1$};
  \node (r2) [below=of r,xshift=-4em] {$r_2$};
  \node (s)  [below=of r2,xshift=4em] {$s = r_3$};
  \node (er) [right=12em of r2]       {$e_2(r,\rho,r)$};
  \node (es) [below=of er]            {$e_2(r,\rho,s)$};

  \draw (r) -- node[above left]{$J$} (r2) -- node[below left]{$J$} (s) 
    -- (r)
    (er) -- (es);
  \draw[->,dashed] ($(er)!0.5!(es)+(-0.2,0)$) to[out=180,in=0]
    node[above,yshift=0.5em]{$J(r,\rho,x)$} ($(r)!0.5!(s)+(0.2,0)$);
  \draw[dotted] (es.south east) node[below]{$e_2(r,\rho,\B)$} rectangle
    (er.north west);
  \draw (current bounding box.north east) node[right]{$\B$} rectangle
    (current bounding box.south west);
\end{tikzpicture}
\caption{Lemma \ref{lem:J defable} illustration.}
\end{figure}
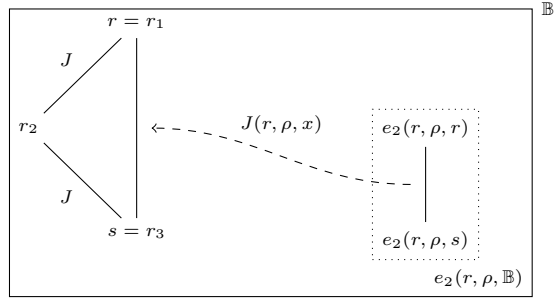
\end{lem}
\begin{proof}
Note that $(r,s)\in \Cg^{\B}(c,d)$ and the presence of a semilattice
operation implies $(r_i,r_{i+1})\in \Cg^{\B}(c,d)$. Next, observe that since
the chain is decreasing and $s\leq r$, if we replace $q_i$ with
$J(q_i,p_i,q_i)$, then we can replace each $p_i$ with $r$. Thus, we may
assume that
\[
  r_i = J(r, q_i, r_i)
  \qquad \text{and} \qquad
  r_{i+1} = J(r,q_i,r_{i+1}).
\]

The proof shall be by induction on $n$ (the length of the chain). If $n=1$,
then 
\[
  r = J(r,q_1,r) 
  \qquad \text{and} \qquad 
  s = J(r,q_1,s).
\]
Therefore
\[
  r = (r\wedge \partial q_1\wedge r) \vee (r\wedge q_1) 
  \qquad \text{and} \qquad
  s = (r\wedge \partial q_1\wedge s) \vee (r\wedge q_1).
\]
Hence without loss of generality, we can replace the last occurrence of $r$
in $r = J(r,q_1,r)$ with $r' = e_2(r,q_1,r)$, and the last occurrence of $s$
in $s = J(r,q_1,s)$ with $s' = e_2(r,q_1,s)$. After making these
replacements, the conclusion of the lemma follows with $\rho = q_1$.

Assume now that the lemma holds for all chains of length less than $N$, and
consider a chain of length $N$: $r = r_1, \ldots, r_N = s$. Applying the
inductive hypothesis to the subchain $r = r_1, \ldots, r_{N-1}$, there
exists $\rho_1\in B$ with $r = J(r,\rho_1,r'')$ and $r_{N-1} =
J(r,\rho_1,r_{N-1}'')$, where $r'' = e_2(r,\rho_1,r)$ and
$r_{N-1}'' = e_2(r,\rho_1,r_{N-1})$. We therefore have 
\begin{equation} \label{eqn:J defable maltsev chain} \begin{aligned}
  r & = J(r,\rho_1,r''), 
    & r_{N-1} = J(r,\rho_1,r_{N-1}'') = J(r,q_{N-1},r_{N-1}), \\
  s & = J(r,q_{N-1},s).
\end{aligned} \end{equation}
Let $\rho = K(r,\rho_1,q_{N-1})$, $r' = e_2(r,\rho,r)$, and $s' =
e_2(r,\rho,s)$. We will now show that $r = J(r,\rho,r')$ and $s =
J(r,\rho,s')$, proving the lemma.

Let $\B = \prod_{l\in L} \C_l$ be a subdirect representation of $\B$ by
subdirectly irreducible algebras. We will analyze the polynomial
$J(r,\rho,x)$ coordinatewise, and as usual it will be easiest to use a
table. Before the table is constructed, however, we will determine which
coordinates permit $r(l)\neq s(l)$. Since $s \leq r_{N-1} \leq r$, either
$r(l)\neq r_{N-1}(l) = s(l) = 0$, or $r(l) = r_{N-1}(l)\neq s(l) = 0$. The
equalities \eqref{eqn:J defable maltsev chain} give us
\begin{align*}
  r & = (r\wedge \rho_1) \vee (r\wedge \partial \rho_1 \wedge r''),
    & r_{N-1} & = (r\wedge \rho_1) \vee (r\wedge \partial \rho_1 \wedge
      r_{N-1}''), \\
  r_{N-1} & = (r \wedge q_{N-1}) \vee (r \wedge \partial q_{N-1}\wedge
      r_{N-1}), 
    & s & = (r\wedge q_{N-1}) \vee (r\wedge \partial q_{N-1} \wedge s).
\end{align*}
Observe that $r(l) = \partial\rho_1(l)$ implies $r(l) =
e_2(r,\partial\rho_1,r)(l) = r''(l)$. Assume first that $r(l) \neq
r_{N-1}(l) = s(l) = 0$. Under this assumption, it must be that $r(l) =
\partial\rho_1(l)$ and $r(l) = \partial q_{N-1}(l)$. Assume now that $r(l) =
r_{N-1}(l)\neq s(l) = 0$. Under this assumption, it must be that $r(l) =
r_{N-1}(l)\in \{\rho_1, \partial\rho_1\}$ and $r = \partial q_{N-1}$.
We now assemble all of this in the table below. As usual, since $r(l) = 0$
implies $r_{N-1}(l) = s(l) = 0$, we assume that $r(l)\neq 0$. In particular,
this means that $r(l)\in \{\rho_1(l), \partial\rho_1(l)\}$.
\[ \begin{array}{l|l|l|c}
  r                         & \rho = K(r,\rho_1,q_{N-1}) & J(r,\rho,e_2(r,\rho,x)) & r\neq s  \\ \hline
  \rho_1 = q_{N-1}          & r                          & r                       & \text{N} \\
  \rho_1 = \partial q_{N-1} & q_{N-1} = \partial r       & e_2(r,\partial r, x)    & \text{Y} \\
  \partial\rho_1            & \rho_1 = \partial r        & e_2(r,\partial r,x)     & \text{Y}
\end{array} \]
If $r(l) = \partial \rho(l)$, then $r'(l) = e_2(r,\rho,r)(l) = r(l)$ and
$s'(l) = e_2(e,\rho,s)(l) = s(l)$. Therefore the table above show that
$J(r,\rho,r') = r$ and $J(r,\rho,s') = s$.
\end{proof}

In light of the above lemma, define
\begin{multline} \label{defn:psi J}
  \psi_J(w,x,y,z) 
    = \exists b\left[ \psi_S(e_2(w,b,w),e_2(w,b,x),y,z) \right. \\
  \left. \wedge w = J(w,b,e_2(w,b,w)) \wedge x = J(w,b,e_2(w,b,x)) \right]
\end{multline}
($\psi_S$ was defined in \eqref{defn:psi S}). From the above lemma, if
$\B\in \V(\A'(\T))$, $\{c,d\}$ is happy, and $s\leq r$, then $(r,s)\in
\Cg^{\B}(c,d)$ is witnessed by a decreasing $J$-$\cdots$-$J$, chain if and
only if $\B\models \psi_J(r,s,c,d)$.

\begin{lem} \label{lem:J' defable}
Suppose that $(r,s)\in \Cg^{\B}(c,d)$ for some $c,d\in  B$ and that there is
a decreasing sequence $r = r_1 \geq r_2\geq \ldots \geq r_n = s$ and
some constants $p_1, q_1, \ldots, p_{n-1}, q_{n-1}\in B$ such that
\[
  r_i = J'(p_i,q_i,e_{j_i}(\overline{n}_i,r_i)) 
  \qquad \text{and} \qquad 
  r_{i+1} = J'(p_i,q_i,e_{j_i}(\overline{n}_i,r_{i+1}))
\]
for some $j_i\in \{0,1,2\}$ and $\overline{n}_i\in B^2\cup B$. Then there
exist constants, $\rho, t\in B$ such that
\begin{align*}
  r & = J'(r,q_1,e_{j_1}(\overline{n}_1,r)), 
    & t & = J'(r,q_1,e_{j_1}(\overline{n}_1,s)) = J(t,\rho,r'), \\
  s & = J(t,\rho,s'),
\end{align*}
where $r' = e_2(r,\rho,r)$, and $s' = e_2(r,\rho,s)$.

That is, for every $J'$-$\ldots$-$J'$ chain (of arbitrary length) there is a
$J'$-$J$ chain with the same endpoints.
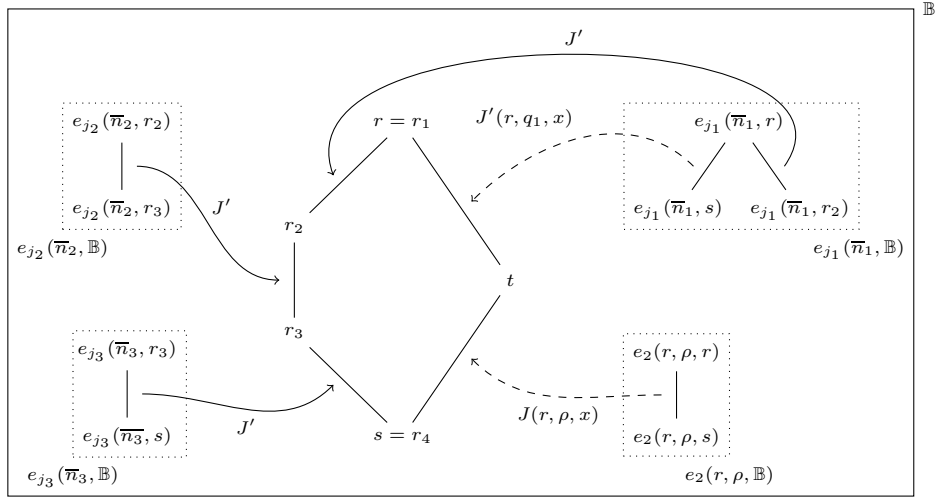
\begin{figure}[H] \centering \begin{tikzpicture}[font=\scriptsize,smooth]
  \node (r)                                 {$r=r_1$};
  \node (r2)   [below=of r,xshift=-4em]     {$r_2$};
  \node (r3)   [below=of r2]                {$r_3$};
  \node (s)    [below=of r3,xshift=4em]     {$s=r_4$};
  \node (t)    at ($(r2)!0.5!(r3)+(8em,0)$) {$t$};
  \begin{scope}[node distance=1.75em]
    \node (e1r)  [right=9em of r]              {$e_{j_1}(\overline{n}_1,r)$};
    \node (e1s)  [below=of e1r,xshift=-2.25em] {$e_{j_1}(\overline{n}_1,s)$};
    \node (e1r2) [below=of e1r,xshift=2.25em]  {$e_{j_1}(\overline{n}_1,r_2)$};
    \node (e2r2) [left=6.75em of r]               {$e_{j_2}(\overline{n}_2,r_2)$};
    \node (e2r3) [below=of e2r2]               {$e_{j_2}(\overline{n}_2,r_3)$};
    \node (e3s)  [left=6.75em of s]               {$e_{j_3}(\overline{n_3},s)$};
    \node (e3r3) [above=of e3s]                {$e_{j_3}(\overline{n}_3,r_3)$};
    \node (es)   [right=6.75em of s]              {$e_2(r,\rho,s)$};
    \node (er)   [above=of es]                 {$e_2(r,\rho,r)$};
  \end{scope}

  \draw (r) -- (r2) -- (r3) -- (s) -- (t) -- (r)
    (e1r) -- (e1r2) (e1r) -- (e1s) (er) -- (es)
    (e2r2) -- (e2r3) (e3r3) -- (e3s);
  \draw[->] ($(e1r)!0.5!(e1r2)+(0.2,0)$) to[out=55,in=110] node[above]{$J'$}
    ($(r)!0.5!(r2)-(0.2,0)$);
  \draw[->,dashed] ($(e1r)!0.5!(e1s)-(0.2,0)$) to[out=135,in=45]
    node[above left]{$J'(r,q_1,x)$} ($(r)!0.5!(t)+(0.2,0)$);
  \draw[->,dashed] ($(er)!0.5!(es)-(0.2,0)$) to[out=180,in=-45]
    node[below]{$J(r,\rho,x)$} ($(t)!0.5!(s)+(0.2,0)$);
  \draw[->] ($(e2r2)!0.5!(e2r3)+(0.2,0)$) to[out=0,in=180]
    node[above,xshift=0.5em]{$J'$} ($(r2)!0.5!(r3)-(0.2,0)$);
  \draw[->] ($(e3r3)!0.5!(e3s)+(0.2,0)$) to[out=0,in=225]
    node[below]{$J'$} ($(r3)!0.5!(s)-(0.2,0)$);

  \coordinate (t1) at (e1r.north west-|e1s.south west);
  \draw[dotted] (t1) rectangle (e1r2.south east)
    node[below]{$e_{j_1}(\overline{n}_1,\B)$};
  \draw[dotted] (er.north west) rectangle (es.south east)
    node[below]{$e_2(r,\rho,\B)$};
  \draw[dotted] (e2r2.north east) rectangle (e2r3.south west)
    node[below]{$e_{j_2}(\overline{n}_2,\B)$};
  \draw[dotted] (e3r3.north east) rectangle (e3s.south west)
    node[below]{$e_{j_3}(\overline{n}_3,\B)$};
  \draw (current bounding box.north east) node[right]{$\B$} rectangle
    (current bounding box.south west);
\end{tikzpicture}
\caption{Lemma \ref{lem:J' defable} illustration.}
\end{figure}
\end{lem}
\begin{proof}
Note that $(r,s)\in \Cg^{\B}(c,d)$ and the presence of a semilattice
operation implies $(r_i,r_{i+1})\in \Cg^{\B}(c,d)$. The proof shall be by
induction on $n$ (the length of the sequence). If $n = 1$, the lemma is
trivially true. Assume now that the lemma holds for all sequences of length
less than $N$, and consider a sequence of length $N$: $r = r_1, \ldots, r_N
= s$. Apply the inductive hypothesis to the subsequence $r_2, \ldots, r_N =
s$ to get
\begin{align*}
  r_2 & = J'(r_2,q_2,e_{j_2}(\overline{n}_2,r_2)), 
    & & t_1 = J'(r_2,q_2,e_{j_2}(\overline{n}_2,s)) = J(t_1,\rho_1,r_2') \\
  s & = J(t_1,\rho_1,s'),
    & & \text{where } r_2' = e_2(r_2,\rho,r_2) \text{ and } 
    s' = e_2(r_2,\rho,s)
\end{align*}
for some constants $\rho_1, t_1\in B$. Since the sequence is decreasing, by
replacing $q_2$ with $J'(q_2,r_2,q_2)$ we are free to replace $r_2$ with
$r$. After doing this replacement we have
\begin{align*}
  r & = J'(r,q_1,e_{j_1}(\overline{n}_1,r)), \\
  r_2 & = J'(r,q_1,e_{j_1}(\overline{n}_1,r_2)) =
    J'(r,q_2,e_{j_2}(\overline{n}_2,r_2)), \\
  t_1 & = J'(r,q_2,e_{j_2}(\overline{n}_2,s)) = J(t_1,\rho_1,r_2'),
    \text{ and} \\
  s & = J(t_1,\rho_1,s').
\end{align*}
We will analyze the subsequence $r, r_2, t_1$.

Let $t = J'(r,q_1,e_{j_1}(\overline{n}_1,t_1))$. We will show that
\begin{align*}
  r & = J'(r,q_1,e_{j_1}(\overline{n}_1,r)), 
    & & t = J'(r,q_1,e_{j_1}(\overline{n}_1,t_1)) = J(t,q_1,r'), \\
  t_1 & = J(t,q_1,t_1'),
    & & \text{for } r' = e_2(r,q_1,r) \text{ and } t_1' = e_2(r,q_1,t_1)
\end{align*}
(that is, $r\Chain{J'} r_2\Chain{J'} t\Chain{J} s$ implies $r\Chain{J'}
t\Chain{J} t_1\Chain{J} s$). The only equalities that have not been shown
already are $t = J(t,q_1,r')$ and $t_1 = J(t,q_1,t_1')$. As usual, let $\B
\leq \prod_{l\in L} \C_l$ be a subdirect representation of $\B$ by
subdirectly irreducible algebras. We will proceed componentwise.

We begin by showing that $t = J(t,q_1,r')$. Since $J(t,q_1,r')\leq t$, by
the flatness of $\C_l$, it will be sufficient to show that $t(l)\neq 0$
implies $J(t,q_1,r')(l)\neq 0$. Suppose that $t(l)\neq 0$. Since $t =
J'(r,q_1,e_{j_1}(\overline{n}_1,t_1))$, either $t(l) = q_1(l)$ or $t(l) =
\partial q_1(l)$. If $t(l) = q_1(l)$, then $J(t,q_1,r')(l) = (t\wedge
q_1)(l) = t(l)$. If $t(l) = \partial q_1(l)$, then $r(l) = t(l)$, since
$\C_l$ is flat and $t\leq r$. Therefore $r'(l) = e_2(r,q_1,r)(l) = r(l) =
t(l)$, and so $J(t,q_1,r')(l) = t(l)$. Hence $J(t,q_1,r') = t$.

Next, we show that $t_1 = J(t,q_1,t_1')$. Again, we will assume that
$t(l)\neq 0$, since $t(l) = 0$ implies that $t_1(l) = 0$ and
$J(t,q_1,t_1')(l) = 0$. Since $\C_l$ is flat, if $t(l)\neq 0$, then
$t_1(l) = t(l)$ and $t(l)\in \{q_1(l), \partial q_1(l)\}$. If $t(l) =
q_1(l)$, then $J(t,q_1,t_1')(l) = t(l) = t_1(l)$. If $t(l) = \partial
q_1(l)$, then $t_1'(l) = e_2(r,q_1,t_1) = t_1(l)$, so $J(t,q_1,t_1')(l) =
t_1'(l) = t_1(l)$. Hence $J(t,q_1,t_1') = t_1$.

We now have
\begin{align*}
  r & = J'(r,q_1,e_{j_1}(\overline{n}_1,r)),
    & t & = J'(r,q_1,e_{j_1}(\overline{n}_1,t_1)) = J(t,q_1,r'), \\
  t_1 & = J(t,q_1,t_1') = J(t_1,\rho_1,r_2'),
    & s & = J(t_1,\rho_1,s'),
\end{align*}
where $r' = e_2(r,q_1,r)$, $t_1' = e_2(r,q_1,t_1)$, $r_2' =
e_2(t_1,\rho_1,t_1)$, and $s' = e_2(r_2,\rho_1,s)$.  Apply Lemma
\ref{lem:J defable} to the sequence $t, t_1, s$ (the part of the sequence
in the range of $J$) to get an element $\rho\in B$ such that $t =
J(t,\rho,t'')$ and $s = J(t,\rho,s'')$ for $t'' = e_2(t,\rho,t)$
and $s'' = e_2(t,\rho,s)$. Since $t\leq r$, if $r' = e_2(r,\rho,r)$ and
$s' = e_2(r,\rho,s)$, we have $t = J(t,\rho,r')$ and $s = J(t,\rho,s')$.
Finally, we now have
\begin{align*}
  r & = J'(r,q_1,e_{j_1}(\overline{n}_1,r)),
    & & t = J'(r,q_1,e_{j_1}(\overline{n}_1,s)) = J(t,\rho,r') \\
  s & = J(t,\rho,s'),
    & & \text{for } r' = e_2(r,\rho,r) \text{ and } s' = e_2(r,\rho,s),
\end{align*}
proving the lemma.
\end{proof}

In light of the above lemma, define
\begin{equation} \label{defn:psi J'J}
  \psi_{J'J}(w,x,y,z) = \exists t \left[ \alpha(t,w,x,y,z)
  \wedge \beta(t,w,x,y,z) \right],
\end{equation}
where
\begin{multline*}
  \alpha(t,w,x,y,z) =
  \bigvee_{i=0}^2 \exists \overline{n},a \left[
  \psi_S(e_i(\overline{n},w),e_i(\overline{n},x),y,z) \right. \\
  \left. \wedge w = J'(w,a,e_i(\overline{n},w))
  \wedge t = J'(w,a,e_i(\overline{n},x)) \right]
\end{multline*}
and
\begin{multline*}
  \beta(t,w,x,y,z) =
    \exists b \left[ \psi_S(e_2(w,b,w),e_2(w,b,x),y,z) \right. \\
  \left. \wedge t = J(t,b,e_2(w,b,w)) \wedge x = J(t,b,e_2(w,b,x)) \right].
\end{multline*}
Recall that $\psi_S$ was defined in \eqref{defn:psi S}. From the above
lemma, if $\B\in \V(\A'(\T))$ and $c,d,r,s\in B$ with $c < d$ and $s\leq r$,
then $(r,s)\in \Cg^{\B}(c,d)$ witnessed by a decreasing $J'$-$\cdots$-$J'$
chain implies that $\B\models \psi_{J'J}(r,s,c,d)$. Conversely, $\B\models
\psi_{J'J}(r,s,c,d)$ implies that $(r,s)\in\Cg^\B(c,d)$ (although this is
perhaps not witnessed by a $J'$-$\cdots$-$J'$ chain).

At this point, we have the machinery necessary to change a general
decreasing Maltsev chain into a longer chain whose associated polynomials
all have $J$, $J'$, or $S_j$ as the outermost operations, and then to
collapse repeated occurrences of $J$ and $J'$ to either a single occurrence of
$J$ or the chain $J'$-$J$. In order to fully collapse the chain, we still
need to address what happens when the chain has alternating $J$ and $J'$
operations.

\begin{lem} \label{lem:J-J' => J'-J}
Let $r, t, s\in B$ be such that $s \leq t\leq r$ and $(r,t)$, $(t,s)\in
\Cg^{\B}(c,d)$ for some $c$, $d\in B$. Suppose that for constants $p_1, p_2,
q_1, q_2\in B$,
\begin{align*}
  r & = J(p_1,q_1,r), 
    & & t = J(p_1,q_1,t) = J'(p_2,q_2,t'), \\
  s & = J'(p_2,q_2,s'),
\end{align*}
for $t' = e_i(\overline{n},t)$ and $s' = e_i(\overline{n},s)$ for some $i\in
\{0,1,2\}$ and $\overline{n}\in B^2\cup B$. Then there exist constants
$\rho$, $u\in B$ such that
\begin{align*}
  r & = J'(r,\rho, r') & u & = J'(r,\rho,s') = J(u,\rho,r'') \\
  s & = J(u,\rho,s''),
\end{align*}
where $r' = e_i(\overline{n},r)$, $r'' = e_2(r,\rho,r)$, and $s'' =
e_2(r,\rho,s)$.

That is, for every $J$-$J'$ chain there is a $J'$-$J$ chain with the same
endpoints.
\begin{figure}[H] \centering \begin{tikzpicture}[font=\scriptsize,smooth]
  \node (r2)   {$r$};
  \node (t2)   [below=of r2]      {$t$};
  \node (s2)   [below=of t2]      {$s$};
  \coordinate (temp1) at ($(r2)!0.5!(t2)$);
  \coordinate (temp2) at ($(t2)!0.5!(s2)$);
  \begin{scope}[node distance=1.25em]
    \node (s1)   [left=6em of temp1] {$s$};
    \node (t1)   [above=of s1]       {$t$};
    \node (r1)   [above=of t1]       {$r$};
    \node (eit)  [left=5em of temp2] {$e_i(\overline{n},t)$};
    \node (eis1) [below=of eit]      {$e_i(\overline{n},s)$};
  \end{scope}

  \node (r3)   [right=8.5em of r2] {$r$};
  \node (u)    [below=of r3]      {$u$};
  \node (s3)   [below=of u]       {$s$};
  \coordinate (temp1) at ($(r3)!0.5!(u)$);
  \coordinate (temp2) at ($(u)!0.5!(s3)$);
  \begin{scope}[node distance=1.25em]
    \node (eis2) [right=5.33em of temp1]          {$e_i(\overline{n},s)$};
    \node (eir)  [above=of eis2]                  {$e_i(\overline{n},r)$};
    \node (e2r)  [right=5em of temp2]             {$e_2(r,\rho,r)$};
    \node (e2s)  [below=of e2r]                   {$e_2(r,\rho,s)$};
    \node (r4)   [right=12em of u,yshift=0.625em] {$r$};
    \node (s4)   [below=of r4]                    {$s$};
  \end{scope}

  \node (temp1) at ($(r2)!0.5!(s3)$) {$\Rightarrow$};

  \draw (r2) -- (t2) -- (s2) (r3) -- (u) -- (s3)
    (r1) -- (t1) -- (s1) (eit) -- (eis1)
    (eir) -- (eis2) (e2r) -- (e2s)
    (r4) -- (s4);

  \draw[->] ($(r1)!0.5!(t1)+(0.2,0)$) to[out=0,in=180]
    node[above,xshift=0.25em]{$J$} ($(r2)!0.5!(t2)-(0.2,0)$);
  \draw[->] ($(t1)!0.5!(s1)-(0.2,0)$) to[out=180,in=180]
    node[left]{$e_i$} ($(eit)!0.5!(eis1)-(0.2,0)$);
  \draw[->] ($(eit)!0.5!(eis1)+(0.2,0)$) to[out=0,in=180]
    node[below]{$J'$} ($(t2)!0.5!(s2)-(0.2,0)$);

  \draw[->] ($(eir)!0.5!(eis2)-(0.2,0)$) to[out=180,in=0]
    node[above]{$J'$} ($(r3)!0.5!(u)+(0.2,0)$);
  \draw[->] ($(e2r)!0.5!(e2s)-(0.2,0)$) to[out=180,in=0]
    node[below]{$J$} ($(u)!0.5!(s3)+(0.2,0)$);
  \draw[->] ($(r4)!1/3!(s4)-(0.2,0)$) to[out=180,in=0]
    node[above,xshift=0.25em]{$e_i$} ($(eir)!0.5!(eis2)+(0.2,0)$);
  \draw[->] ($(r4)!2/3!(s4)-(0.2,0)$) to[out=180,in=0]
    node[below,xshift=0.25em]{$e_2$} ($(e2r)!0.5!(e2s)+(0.2,0)$);

  \draw[dotted] (eit.north east) rectangle 
      (eis1.south west) node[below]{$e_i(\overline{n},\B)$}
    (eir.north east) node[above]{$e_i(\overline{n},\B)$} rectangle
      (eis2.south west)
    (e2r.north west) rectangle 
      (e2s.south east) node[below]{$e_2(r,\rho,\B)$};
  \draw (current bounding box.north east) node[right]{$\B$} rectangle
    (current bounding box.south west);
\end{tikzpicture}
\caption{Lemma \ref{lem:J-J' => J'-J} illustration.}
\end{figure}
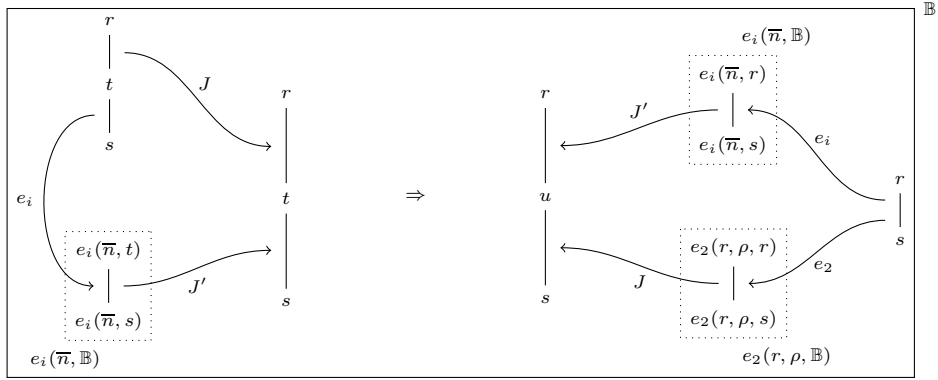
\end{lem}
\begin{proof}
Since $s\leq t\leq r$, in the equations in the hypothesis by replacing $q_1$
with $J(p_1,q_1,p_1)$ and $q_2$ with $J'(p_2,q_2,p_2)$, we can replace $p_1$
and $p_2$ with $r$. Thus,
\begin{equation} \label{eqn:J-J' => J'-J lemma chain} \begin{aligned}
  r & = J(r,q_1,r) = (r\wedge \partial q_1\wedge r) \vee (r\wedge q_1), \\
  t & = J(r,q_1,t) = (r\wedge\partial q_1 \wedge t)\vee (r\wedge q_1) \\
    & = J'(r,q_2,t') = (r\wedge q_2\wedge t') \vee (r\wedge\partial q_2),
      \text{ and} \\
  s & = J'(r,q_2,s') = (r\wedge q_2\wedge s') \vee (r\wedge\partial q_2).
\end{aligned} \end{equation}
Let $\rho = K(r,q_1,q_2)$ and $u = J'(r,\rho,s')$. Let $\B \leq \prod_{l\in
L} \C_l$ be a subdirect representation of $\B$ by subdirectly irreducible
algebras. Note that from the definition of $K$ and the equations
\eqref{eqn:J-J' => J'-J lemma chain}, for all $l\in L$, $r(l)\in
\{\rho,\partial \rho, 0\}$.  We will show that the equalities in the
conclusion of the lemma hold componentwise.

We begin by showing that $r = J'(r,\rho,r')$. As usual, a table is the
easiest way to organize the proof. Since $r(l) = 0$ implies
$J'(r,\rho,r')(l) = 0$, assume that $r(l)\neq 0$.
\[ \begin{array}{l|l|l|c}
  r                                        & \rho = K(r,q_1,q_2) & J'(r,\rho,r') & r\neq t  \\ \hline
  q_1 = q_2                                & q_1 = q_2 = r       & r\wedge r'    & \text{N} \\
  q_1 = \partial q_2                       & q_2 = \partial r    & r             & \text{N} \\
  \partial q_1 = q_2                       & q_1 = \partial r    & r             & \text{Y} \\
  \partial q_1 = \partial q_2              & q_1 = \partial r    & r             & \text{N} \\
  q_1 \not\in\{q_2,\partial q_2\}          & 0                   & 0             & \text{N} \\
  \partial q_1 \not\in\{q_2,\partial q_2\} & q_1 = \partial r    & r             & \text{Y}
\end{array} \]
The only possibly problematic cases are when $r(l) = q_1(l) = q_2(l)$ and
when $r(l) = q_1(l)\not\in\{q_2(l) ,\partial q_2(l)\}$. 

\Case{$r(l) = q_1(l) = q_2(l)$}
If $r(l) = q_1(l)$, then $r(l) = t(l)$, by \eqref{eqn:J-J' => J'-J lemma
chain}, so $r(l) = t(l) = t'(l)$.  Since $t'(l)\leq r'(l)$ (because
$e_i(\overline{n},-)$ is monotonic and $t\leq r$), it follows that $r'(l) =
r(l)$. Thus $J(r,\rho,r')(l) = r(l)$ in this case.

\Case{$r(l) = q_1(l) \not\in\{ q_2(l), \partial q_2(l) \}$}
If $r(l) = q_1(l)$, then $r(l) = t(l)$, but if $r(l)\not\in \{q_2(l),
\partial q_2(l) \}$ then $t(l) = 0$ by \eqref{eqn:J-J' => J'-J lemma chain},
contradicting our assumption that $r(l) \neq 0$. Therefore, in this case
$J'(r,\rho,r')(l) = r(l)$ as well.
\smallskip

Next, we show that $u = J(u,\rho,r'')$. Since $J(u,\rho,r'')\leq u$ and each
$\C_l$ is flat, it will be sufficient to show that when $u(l)\neq 0$,
$J(u,\rho,r'')(l)\neq 0$ as well. When $u(l)\neq 0$, since $u\leq r$, it
must be that $u(l) = r(l)$. From the construction of $\rho$, if $r(l)\neq 0$
either $r(l)\in \{\rho(l),\partial \rho(l)\}$. If $\rho(l) = r(l)$, then
$J(u,\rho,r'')(l) = r(l) = u(l)$. Suppose now that $\rho(l) = \partial
r(l)$. Then $r''(l) = e_2(r,\rho,r)(l) = r(l) = u(l)$, so $J(u,\rho,r'')(l)
= r(l)$. Since $\rho(l) = r(l)$ or $\rho(l) = \partial r(l)$ for all $l\in
L$, we have that $u = J(r,\rho,r'')$.

Finally, we show that $s = J(u,\rho,s'')$. There are three possibilities:
$r(l) = u(l) = s(l)$, $r(l) = u(l)$ but $s(l)=0$, or $r(l)\neq 0$ but $u(l)
= s(l) = 0$.

\Case{$r(l) = u(l) = s(l)$}
If $\rho(l) = r(l)$, then $J(u,\rho,s'') = (u\wedge \rho)(l) = r(l) = s(l)$.
If $\rho(l) = \partial r(l)$, then $s''(l) = e_2(r,\rho,s)(l) = s(l)$, so
$J(u,\rho,s'') = (u\wedge \rho\wedge s'')(l) = s(l)$.

\Case{$r(l) = u(l)$ but $s(l) = 0$}
If $\rho(l) = r(l)$, then $r(l) = q_1(l) = q_2(l)$, so $s(l) =
J'(r,q_2,s')(l) = J'(r,\rho,s')(l) = u(l)$, contradicting $u(l)\neq 0$. If
$\rho(l) = \partial r(l)$, then $s''(l) = e_2(r,\rho,s) = s(l)$, so
$J(u,\rho,s'')(l) = s''(l) = s(l)$.

\Case{$r(l)\neq 0$ and $u(l) = s(l) = 0$}
If $\rho(l) = r(l)$, then $J(u,\rho,r'')(l) = u(l) = J(u,\rho,s'')(l)$, so
$s(l) = J(u,\rho,s'')$. If $\rho(l) = \partial r(l)$, then $s''(l) =
e_2(r,\rho,s) = s(l)$, so $J(u,\rho,s'')(l) = s''(l) = s(l)$.
\smallskip

In all cases, we have $J(u,\rho,s'')(l) = s(l)$, so it must be that
$J(u,\rho,s'') = s$, completing the proof.
\end{proof}

Lemma \ref{lem:J' defable} allows us to reduce a chain consisting of a
string of $J'$ operations to a $J'$-$J$ chain. A chain of length $1$
consisting of a single $J'$ operation is an example of a $J'$-$J$ chain
since $J(x,x,x)\approx x$ in $\V(\A'(\T))$.

Since $e_i(\overline{n},\B)$ has the property that $a\in
e_i(\overline{n},B)$ and $b\leq a$ implies $b\in e_i(\overline{n},B)$, for
any decreasing Maltsev chain, if any one of the intermediate elements is
happy, then all subsequent ones are happy as well. Thus, every Maltsev chain
must terminate in a (possibly length 0) $S_i$ chain, and $S_i$ chains do not
appear anywhere else in the chain except at the end. Lemma \ref{lem:S_i
defable} allows us to collapse repeated $S_i$ links to a single $S_i$.
Hence, to the already defined $\psi_S$, $\psi_J$, and $\psi_{J'J}$ we add
the following:
\begin{align}
  \label{defn:psi JS}
  \psi_{JS}(w,x,y,z) 
    & = \exists t \left[ \psi_J(w,t,y,z) \wedge \psi_S(t,x,y,z) \right], \\
  \label{defn:psi J'JS}
  \psi_{J'JS}(w,x,y,z)
    & = \exists t \left[ \psi_{J'J}(w,t,y,z) \wedge \psi_S(t,x,y,z) \right]
\end{align}
(see equations \eqref{defn:psi S}, \eqref{defn:psi J}, and \eqref{defn:psi
J'J} for definitions of $\psi_S$, $\psi_J$, and $\psi_{J'J}$, respectively).

\begin{lem} \label{lem:reduce chains}
Let $\{c,d\}$ be happy. If $(r,s)\in \Cg^{\B}(c,d)$ is witnessed by a
decreasing Maltsev sequence whose associated polynomials are primitive, then
$(r,s)\in \Cg^{\B}(c,d)$ is witnessed by one of the following chains:
\begin{enumerate}
  \item $S_j$ for some $j\in \{0,1,2\}$ and $\B\models \psi_S(r,s,c,d)$,
  \item $J$ and $\B\models \psi_J(r,s,c,d)$,
  \item $J'$-$J$ and $\B\models \psi_{J'J}(r,s,c,d)$,
  \item $J$-$S_j$ for some $j\in \{0,1,2\}$ and $\B\models
    \psi_{JS}(r,s,c,d)$, or
  \item $J'$-$J$-$S_j$ for some $j\in \{0,1,2\}$ and $\B\models
    \psi_{J'JS}(r,s,c,d)$.
\end{enumerate}
Moreover, if $\B\models \psi_G(r,s,c,d)$ for $G\in \{J,J',J'J,JS,J'S,J'JS\}$
then $(r,s)\in \Cg^{\B}(c,d)$.
\end{lem}
\begin{proof}
In all of the cases, that $\B$ models the claimed first-order formula
follows from the definition of the formula and the conclusion of the
appropriate lemmas: \ref{lem:S_i defable} for formulas whose subscript ends
in $S$, \ref{lem:J defable} for formulas whose subscript begins in $J$, and 
\ref{lem:J' defable} and \ref{lem:e_i Maltsev sequence} for formulas whose
subscript begin with $J'$. The ``moreover'' part of the lemma follows from
the fact that each $\psi_G$ is a congruence formula.

Let $r = r_1, r_2,\ldots, r_n = s$ be the decreasing Maltsev sequence
witnessing $(r,s)\in \Cg^{\B}(c,d)$ and let $\lambda_1(x), \ldots,
\lambda_{n-1}(x)$ be the primitive polynomials associated to it. From Lemma
\ref{lem:e_i Maltsev sequence}, without loss of generality we may assume
that for each $k\in \{1,\ldots,n-1\}$ one of the following holds 
\begin{enumerate}
  \item $\lambda_k(x) = S_{j_k}(\overline{m}_k,r_k,r_{k+1},h_k(x))$ for some
    $j_k\in \{0,1,2\}$ and $\overline{m}_k\in B^2\cup B$ (i.e. $\lambda_k$
    is happy),
  \item $\lambda_k(x) = J(r_k,q_k,h_k(x))$ for some $q_k\in B$, or
  \item $\lambda_k(x) = J'(r_k,q_k,h_k(x))$ for some $q_k\in B$,
\end{enumerate}
where the polynomials $h_k(x)$ are happy and primitive for all $k$. Since
$e_{i}(\overline{n},B)$ has the property that if $q\in
e_{i}(\overline{n},B)$ and $p\leq q$ then $p\in e_{i}(\overline{n},B)$, if
$r_k \in e_i(\overline{n},B)$, then 
\[
  r = r_1 \Chain{F_1} r_2 \Chain{F_2} r_3 \cdots r_{k-1} \Chain{F_{k-1}} r_k
    \Chain{S_{i}} r_{k+1} \Chain{S_{i}} r_{k+2} \cdots r_{n-1} 
    \Chain{S_{i}} r_n = s,
  \quad
  F_i\in \{J,J'\}.
\]
Lemma \ref{lem:S_i defable} can be applied to the pair $(r_k,s)\in
\Cg^{\B}(c,d)$ to collapse the end of the chain, and produce a new shorter
chain of the form $F_1$-$F_2$-$\cdots$-$S_{i}$.

From Lemmas \ref{lem:J defable} and \ref{lem:J' defable}, subchains
consisting of entirely $J$ or $J'$ can be converted to subchains consisting
of a single $J$ or $J'$-$J$, respectively:
\begin{align*}
  J - J - \cdots - J & \Rightarrow J, \\
  J' - J' - \cdots J' & \Rightarrow J' - J
\end{align*}
Thus, we need only consider chains in which the $J$ and $J'$ are mixed. We
will show that all such chains can be reduced to $J'-J$ chains. We have
\begin{align*}
  & J - J' - J' 
    \Rightarrow J - J' - J
    \Rightarrow J' - J - J
    \Rightarrow J' - J, \\
  & J' - J - J'
    \Rightarrow J' - J' - J
    \Rightarrow J' - J - J
    \Rightarrow J' - J, \\
  & J' - J' - J 
    \Rightarrow J' - J - J
    \Rightarrow J' - J, \\
  & J' - J - J
    \Rightarrow J' - J, \\
  & J - J' - J 
    \Rightarrow J' - J - J
    \Rightarrow J' - J, \text{ and} \\
  & J - J - J'
    \Rightarrow J - J'
    \Rightarrow J' - J
\end{align*}
(using Lemmas \ref{lem:J defable}, \ref{lem:J' defable}, and \ref{lem:J-J'
=> J'-J}). It follows that all mixed chains of $J$ and $J'$ can be reduced
to a $J'$-$J$ chain. The conclusion of the lemma follows.
\end{proof}

Given the previous lemma, let
\begin{multline*}
  \psi_1(w,x,y,z) = \psi_S(w,x,y,z) 
    \vee \psi_J(w,x,y,z) 
    \vee \psi_{J'J}(w,x,y,z) \\
    \vee \psi_{JS}(w,x,y,z) 
    \vee \psi_{J'JS}(w,x,y,z).
\end{multline*}
From above lemma, if $\{c,d\}$ is happy and $s\leq r$, then $(r,s)\in
\Cg^{\B}(c,d)$ if and only if $\B\models \psi_1(r,s,c,d)$.

All of the lemmas above required that $s\leq r$. Since $\B$ is a
semilattice, if $(r,s)\in \Cg^{\B}(c,d)$, then there is an intermediate
element, $t\leq r\wedge s$, such that $(r,t), (t,s)\in \Cg^{\B}(c,d)$ and
there are decreasing Maltsev chains connecting $r$ to $t$ and $s$ to $t$
(this will be proved in detail in Theorem \ref{thm:dpsc e_i factor}).
Therefore, define
\begin{equation} \label{defn:psi 2}
  \psi_2(w,x,y,z) = \exists t \left[ \psi_1(w,t,y,z) \wedge \psi_1(x,t,y,z)
  \right].
\end{equation}
Finally, we will now use the above lemmas to prove that there is a
congruence formula $\Gamma_1$ (defined in the theorem below) such that if
$a,b\in B$ are distinguished by a polynomial of the form
$e_i(\overline{n},x)$ for some $i\in \{0,1,2\}$ and some $\overline{n}$,
then $\Cg^{\B}(a,b)$ has a subcongruence witnessed by $\Gamma_1(-,-,a,b)$
and that this subcongruence is defined by $\psi_2$.

\begin{thm} \label{thm:dpsc e_i factor}
Let $a,b\in B$ and suppose that there is $i\in \{0,1,2\}$ and
$\overline{m}\in B^2\cup B$ such that $e_i(\overline{m},a)\neq
e_i(\overline{m},b)$. Let 
\[
  \Gamma_1(w,x,y,z) = \bigvee_{j=0}^2 \exists \overline{n}\;
    \Gamma_0(w,x,e_j(\overline{n},y),e_j(\overline{n},z))
\]
($\Gamma_0$ was defined in \eqref{defn:psi 0}). The congruence
$\Cg^{\B}(a,b)$ has a principal subcongruence witnessed by
$\Gamma_1(-,-,a,b)$ and defined by $\psi_2$. That is,
\[
  \B \models 
    \exists c,d \left[ \Gamma_1(c,d,a,b) \wedge \Pi_{\psi_2}(c,d) \right].
\]
\end{thm}
\begin{proof}
From the definition of $e_i$ \eqref{defn:e_i} and Lemma \ref{lem:dpsc all
factors}, $e_i(\overline{m},\B)$ is congruence distributive and has
definable principal subcongruences witnessed by $\Gamma_0$ and $\psi_0$
(defined in \eqref{defn:psi 0}). Therefore there are $c,d\in
e_i(\overline{m},B)$ such that 
\[
  (c,d)\in
    \Cg^{e_i(\overline{m},\B)}(e_i(\overline{m},a),e_i(\overline{m},b))
\]
is witnessed by $\Gamma_0$ and $\Cg^{e_i(\overline{m},\B)}(c,d)$ is defined
in by $\psi_0$. By Lemma, we have that \ref{lem:S_i defable},
$e_i(\overline{m},\B)\models \Pi_{\psi_2}(c,d)$. Summarizing,
\[
  e_i(\overline{m},\B)\models
    \Gamma_0(c,d,e_i(\overline{m},a),e_i(\overline{m},b))
  \qquad \text{and} \qquad
  e_i(\overline{m},\B)\models \Pi_{\psi_2}(c,d).
\]
Since $\Gamma_0$ is existentially quantified (it is a congruence formula)
and $e_i(\overline{m},\B)\leq \B$, $\B\models \Gamma_1(c,d,a,b)$. It remains
to be shown that $\B\models \Pi_{\psi_2}(c,d)$ (that is, that
$\Cg^{\B}(c,d)$ is defined in $\B$ by $\psi_2$).

Let $r, s\in B$ and $(r,s)\in \Cg^{\B}(c,d)$. To show that $\B\models
\psi_2(r,s,c,d)$, by Lemma \ref{lem:reduce chains} we need only show that
there are decreasing Maltsev sequences connecting $r$ to some $t$ and $s$ to
$t$ and whose associated polynomials are primitive.

Let $r = r_1, \ldots, r_n = s$ be a Maltsev sequence connecting $r$ to $s$
with associated primitive polynomials $\lambda_1(x), \ldots,
\lambda_{n_1}(x)$. Let
\[
  t_i = \begin{cases}
    r_1\wedge r_2 \wedge \cdots \wedge r_i & \text{if } i\leq n, \\
    t_{i-n} \wedge r_{i-n+1} \cdots \wedge r_n & \text{if } n\leq i\leq 2n,
  \end{cases}
\]
and
\[
  \mu_i(x) = \begin{cases}
    \lambda_i(x) \wedge t_i & \text{if } i<n, \\
    \lambda_{i-n} \wedge t_{i+1} & \text{if } n < i \leq 2n.
  \end{cases}
\]
Then the sequences $r = t_1, t_2, \ldots, t_n$ and $s = t_{2n}, \ldots
t_{n+1} = t_n$ are decreasing Maltsev sequences witnessed by the primitive
polynomials $\mu_i(x)$. Thus,
\[
  \B \models \psi_1(r,t_n,c,d) \wedge \psi_1(s,t_n,c,d),
\]
and hence $\B\models \psi_2(r,s,c,d)$. From the definition of $\psi_2$, it
is a congruence formula, so if $\B\models \psi_2(u,v,c,d)$ then $(u,v)\in
\Cg^{\B}(c,d)$. Therefore $\B\models \Pi_{\psi_2}(c,d)$.
\end{proof}

Having completed the argument for the case when $a,b\in B$ are distinguished
by a polynomial of the form $e_i(\overline{m},x)$ for some $i\in \{0,1,2\}$
and some $\overline{m}\in B^2\cup B$, we move on to the case where $a,b$ are
distinguished by an operation from a sequential SI. The next lemma is
crucial for this case as well as the case for machine SI's.

\begin{lem} \label{lem:e_i=0 no maltsev seq}
Let $c,d\in B$ be such that $d\leq c$ and $e_i(\overline{m},c) =
e_i(\overline{m},d)$ for all $i\in \{0,1,2\}$ and all $\overline{m}\in
B^2\cup B$. Suppose that
\begin{align*}
  & r = f_1(c), && t = f_1(d) = f_2(c), \\
  & s = f_2(d),
\end{align*}
for some polynomials $f_1(x)$ and $f_2(x)$. Then $r = t$ or $t = s$.
\end{lem}
\begin{proof}
Suppose that $t\neq s$. We will show that $r = t$. Let $\B\leq \prod_{l\in
L} \C_l$ be a subdirect representation of $\B$ by subdirectly irreducible
algebras. Since each $\C_l$ is flat and $s<t\leq r$, there is $k\in L$
such that $r(k) = t(k)\neq 0$, and $s(k) = 0$.

\begin{claim}
if $s(k) = 0$ then $d(k) = 0$.
\end{claim}
\begin{claimproof}
Suppose to the contrary that $d(k)\neq 0$ but $0 = s(k) = f_2(d)(k)$. Since
$\C_k$ is flat, $d(k)\neq 0$ implies that $d(k) = c(k)$, so
\[
  t(k)
  = f_2(c)(k)
  = f_2(d)(k) 
  = s(k) 
  = 0.
\]
This contradicts our choosing $k$ such that $t(k) \neq s(k) = 0$, and
proves the claim.
\end{claimproof}

By the above claim, we have that $d(k) = 0$, but since the only $\C_l$ where
$d(l)\neq c(l)$ are $0$-absorbing (see the description of SI's in Section
\ref{sec:SIs}), this implies that either $f_1(d)(k) = 0$, contradicting
$t(k) = f_1(d)(k)\neq s(k) = 0$, or that $f_1(c)(l) = f_1(d)(l)$ for all $l$
such that $\C_l$ is $0$-absorbing (i.e. $f_1(x)$ doesn't depend on $x$ in
the $0$-absorbing $\C_l$). Since $c(l)$ and $d(l)$ can only differ when
$\C_l$ is $0$-absorbing, this means that $f_1(c) = f_1(d)$, implying that $r
= t$.
\end{proof}

Our last remaining task is to address the case where $a,b\in B$ differ at
coordinate that is one of the three small SI's that model
$e_i(\overline{y},x)\approx 0$. These algebras are described in Section
\ref{sec:SIs}.

\begin{lem} \label{lem:e_i=0 all polys=J'T}
Let $\B\leq \prod_{l\in L} \C_l$ be a subdirect representation of $\B$ by
subdirectly irreducible algebras, and suppose that $c,d\in B$ are such that
\begin{enumerate}
  \item $d\leq c$,
  \item $e_i(\overline{n},c) = e_i(\overline{n},d)$ for all $i\in \{0,1,2\}$
    and all $\overline{n}\in B^2\cup B$, and
  \item for each $l\in L$, $\Cg^{\C_l}(c(l),d(l))$ lies in the monolith of
    $\C_l$.
\end{enumerate}
Let 
\begin{multline*}
  \mathcal{C} = \{\id(x)\} \cup \{ F_1(a_1,b_1,F_2(a_2,b_2,\cdots
    F_n(a_n,b_n,x)\cdots )) \\
  \mid n\in \NN, F_i\in \mathcal{L}\cup \mathcal{R}, \text{and } a_i,b_i\in
  B\}.
\end{multline*}
If $g(x)$ is a primitive polynomial of $\B$ such that $g(c)\neq g(d)$, then
there is some $\rho\in B$, some $F(x)\in \mathcal{C}$, and some polynomial
$g'(x) = J'(g(c),\rho,F(x))$ such that $(g(c),g(d)) = (g'(c),g'(d))$.
\end{lem}
\begin{proof}
Let $\B = \prod_{l\in L} \C_l$ be a subdirect representation of $\B$ by
subdirectly irreducible algebras $\C_l$. We begin by proving a claim.

\begin{claim}
If $h(x) = H(f(x))$ for a fundamental translation $H(x)$ and polynomial
$f(x)$, then $h'(x) = J'(h(c),\rho,F(f(x)))$ satisfies $(h(c),h(d)) =
(h'(c),h'(d))$ for some $\rho\in B$ and some $F(x)\in \mathcal{C}$.
\end{claim}
\begin{claimproof}
For convenience, let $r = h(c)$ and $s = h(d)$. The proof shall be by cases,
depending on which particular fundamental operation $H$ is a translation of.
As usual, we shall proceed componentwise. The only possible $l\in L$ with
$c(l)\neq d(l)$ are such that $\C_l\models e_i(\overline{n},x)\approx 0$, by
the second hypothesis, and the third hypothesis implies that 
\[
  \B\models \big[ x\cdot c \approx x\cdot d \big] 
    \wedge \big[ c\cdot x \approx d\cdot x \big]
\]
Therefore by Lemma \ref{lem:small SI's e_i=0} and from the hypotheses, the
only fundamental translations that possibly do not collapse $(c,d)$ are
translations of the operations $\wedge$, $J$, $J'$, $K$, $E$, $U_E^0$, and
$U_E^1$, where $E\in \mathcal{L}\cup\mathcal{R}$. In all cases except for
operations from $\mathcal{L}\cup\mathcal{R}$ we will take the $F(x)\in
\mathcal{C}$ in the statement of the claim to be $\id(x)$.

Before beginning with the cases, note that if $h(x) = H(f(x)) \leq f(x)$,
then since $\C_l$ is flat either $r(l) = f(c)(l)$ or $r(l) = 0$, and
likewise for $s(l)$. The polynomial $h'(x) = J'(r,r,f(x)) = r\wedge f(x)$
therefore has $h'(c) = r(l)$ and $h'(d) = s(l)$. Therefore in cases where $h(x)
\leq f(x)$, taking $\rho = r$ and $F = \id$ is sufficient.

\Case{$\wedge$}
If $h(x) = u\wedge f(x)$, then $h(x)\leq f(x)$, so by the above remarks,
take $\rho = r$.

\Case{$J$}
If $h(x) = J(f(x),u,v)$, then $h(x)\leq f(x)$, so by the remarks at the
start of the cases let $\rho = r$. If $h(x) = J(u,f(x),v)$, then let $\rho =
K(r,f(c),r)$. Since many of the later cases are similar to this, we will
carefully prove that $h'(x) = J'(r,\rho,f(x))$ satisfies $(h'(c), h'(d)) =
(r,s)$. We have
\[
  h(x) 
  = J(u,f(x),v)
  = (u\wedge f(x)) \vee (u\wedge \partial f(x) \wedge v).
\]
This yields the following table (assume that $r(l)\neq 0$, since $h'(x)(l) =
0$ otherwise).
\[ \begin{array}{l|l|l|l}
  r             & \rho = K(r,f(c),r) & h'(c)            & h'(d)           \\ \hline
  f(c)          & r                  & r\wedge f(c) = r & r\wedge f(d) =s \\
  \partial f(c) & f(c) = \partial r  & r                & r
\end{array} \]
The only possibly problematic case is when $r(l) = \partial f(c)(l)$, but in
this case we have that $s(l) = \partial f(d)(l)$, so $r(l) =
e_2(r,f(c),r)(l)$ and $s(l) = e_2(r,f(c),s)(l)$, contradicting hypothesis
(2) in the statement of the lemma. It follows that $h'(c) = r$ and $h'(d) =
s$.

If $h(x) = J(u,v,f(x))$, then $h(c)(l)$ and $h(d)(l)$ agree whenever $u(l) =
v(l)$ and can only possibly differ when $u(l) = \partial v(l)$. Hence, if
$h(c)\neq h(d)$, then $e_2(u,v,h(c))\neq e_2(u,v,h(d))$, contradicting
hypothesis (2) again.

\Case{$J'$}
If $h(x) = J'(f(x),u,v)$, then $h(x)\leq f(x)$, so by the remarks at the
start of the cases let $\rho = r$. If $h(x) = J'(u,f(x),v)$ or $h(x) =
J'(u,v,f(x))$, let $\rho = K(r,f(c),r)$. An argument similar to the one in
Case $J$ will work.

\Case{$K$}
If $h(x) = K(f(x),u,v)$, then let $\rho = K(r,f(c),r)$. If we have $h(x) =
K(u,f(x),v)$, then let $\rho = K(r,u,r)$. If $h(x) = K(u,v,f(x))$, then let
$\rho = K(r,u,r)$. Arguments similar to the one in Case $J$ will work.

\Case{$E\in \mathcal{L}\cup \mathcal{R}$}
If $g(x) = E(f(x), u, v)$ or $g(x) = E(u, f(x), v)$ then hypothesis (3) implies that
when $\C_l$ is of machine type and $f(c)(l) = g(c)(l)$ then $c(l)$
is a configuration element and therefore $g(c)(l) = 0$. Thus in this case,
$g(c) = g(d)$, a contradiction. Thus we need only examine $g(x) = E(u, v,
f(x))$. In this case, $g'(x) = J'(r, r, f(x))$ clearly works.

\Case{$U_E^i$ for $E\in \mathcal{L}\cup \mathcal{R}$ and $i\in \{0,1\}$}
This case is quite similar to the previous one. Use the fact that $c(l)$ and
$d(l)$ only differ on sequential and machine $\C_l$, that
$U_E^i(w,x,y,z)\approx 0$ on sequential SI's, and that
\[
  U_E^i(u,v,w,x) = 0 
  \qquad \text{except for} \qquad
  U_E^0(v,u,v,x) = E(u,v,x) = U_E^1(u,u,v,x)
\]
in the machine SI's.
\end{claimproof}

The polynomial $g(x)$ is primitive and therefore generated by fundamental
translations. It follows that there is some fundamental translation $G(x)$
such that $g(x) = G(f(x))$. Apply the above claim to $g(x) = G(f(x))$ to
get that there is $\rho\in B$ and $F(x)\in \mathcal{C}$ such that the
polynomial $g'(x) = J'(g(c),\rho,F(f(x)))$ satisfies $(g(c),g(d)) =
(g'(c),g'(d))$.

Since $g(x) = G(f(x))$ is a primitive polynomial, $f(x)$ is also primitive.
Therefore there is a fundamental translation $H(x)$ such that $f(x) =
H(h(x))$. Apply the above claim to the $f(x)$ in $g'(x) =
J'(g(c),\rho,F(f(x)))$ from the above paragraph to get that there is
$\rho'\in B$ and $F'(x)\in \mathcal{C}$ such that the polynomial
\[
  g''(x) = J'(g(c),\rho,F(J'(f(c),\rho',F'(h(x)))))
\]
satisfies $(g''(c),g''(d)) = (g'(c),g'(d)) = (g(c),g(d))$. The second claim
will show how to reduce this polynomial to the form required by the
conclusion of the lemma.

\begin{claim}
If $h(x) = J'(u,v,F(J'(p,q,E(f(x)))))$ for constants $u,v,p,q\in B$ and
$F(x), E(x)\in \mathcal{C}$ then there is some $\rho\in B$ and $E_1\in
\mathcal{C}$ such that 
\[
  h'(x) = J'(h(c),\rho,E_1(f(x)))
\]
has $h'(c) = h(c)$ and $h'(d) = h(d)$.
\end{claim}
\begin{claimproof}
Our first task will be to find another polynomial that agrees with
$F(J'(p,q,E(f(x))))$ on $\{c,d\}$ but has the form $J'(r_1,\rho_1,G(f(x)))$
for some $r_1,\rho_1\in B$.

If $h_1(x) = G_1(J'(p,q,f_1(x)))$ where $G_1(x) = G_1'(a_1,b_1,x)$ for
$G_1'\in \mathcal{L}\cup\mathcal{R}$, then there is $\rho_1\in B$ such that
$h_1'(x) = J'(h_1(c),\rho_1,G_1(f(x)))$ has $h_1'(c) = h_1(c)$ and $h_1'(d)
= h_1(d)$.  To see this, let $\rho_1 = G_1(q)$. We have
\[
  h_1(x) 
  = G_1(J'(p,q,f_1(x))) 
  = G_1( (p\wedge\partial q)\vee (p\wedge q\wedge f_1(x)) ),
\]
and $G_1(\partial x) = \partial G_1(x)$ (this last equation is true because
$G_1'\in \mathcal{L}\cup \mathcal{R}$). Therefore $h_1'(x) =
J'(h(c),G_1(q),G_1(f_1(x)))$ agrees with $h_1(x)$ on $\{c,d\}$.

By repeatedly applying the result of the above paragraph to
$F(J'(p,q,E(f(x))))$, (and using the fact that $F$ is a composition of
translations of the form $F_i(a_i,b_i,x)$ for $F_i\in
\mathcal{R}\cup\mathcal{L}$), we obtain a polynomial of the form
$J'(r_1,\rho_1,G(f(x)))$ that agrees with $F(J'(p,q,E(f(x))))$ on
$\{c,d\}$.

At this point, we can take the polynomial $h(x) =
J'(u,v,F(J'(p,q,E(f(x)))))$ in the statement of the claim and produce a
polynomial 
\[
  h_1(x) = J'(u,v,J'(r_1,\rho_1,G(f(x))))
\]
for some $r_1,\rho_1\in B$ and $G\in mathcal{C}$ such that $(h_1(c),h_1(d))
= (h(c),h(d))$. 

Next, we will show that there is some $\rho\in B$ such that $h'(x) =
J'(h(c),\rho,G(f(x)))$ satisfies $(h'(c),h'(d)) = (h_1(c),h_1(d)) =
(h(c),h(d))$. Let $\rho = K(h(c),v,\rho_1)$. We have
\begin{align*}
  h_1(x)
  & = J'(u,v,J'(r_1,\rho_1,G(f(x)))) \\
  & = (u\wedge v\wedge r_1 \wedge \rho_1 \wedge G(f(x))) \vee (u\wedge
    v\wedge r_1\wedge \partial \rho_1) \vee (u\wedge \partial v).
\end{align*}
This gives us the following table of cases (as usual, assume that $r(l)\neq
0$ since $h_1(x)(l) = 0$ otherwise).
\[ \begin{array}{l|l|l|l|c}
  h_1(c)=h(c)         & \rho = K(h(c),v,q)      & h'(c)              & h'(d)              & h(c)\neq h(d)  \\ \hline
  v = \rho_1          & h(c)                    & h(c)\wedge G(f(c)) & h(c)\wedge G(f(d)) & \text{Y}       \\
  v = \partial \rho_1 & \rho_1 = \partial h(c)  & h(c)               & h(c)               & \text{N}       \\
  \partial v          & v = \partial h(c)       & h(c)               & h(c)               & \text{N}
\end{array} \]
In the case where $v(l) = \rho_1(l)$ we also have that $h(c)(l) =
G(f(c))(l)$ and $h(d)(l) = G(f(d))(l)$, so the table above indicates that
$h'(c) = h(c)$ and $h'(d) = h(d)$.
\end{claimproof}

Applying this claim to the previously computed 
\[
  g''(x) = J'(g(c),\rho,F(J'(f(c),\rho',F'(h(x)))))
\]
produces a polynomial $g_1(x) = J'(g(c),\rho_1,F_1(h(x)))$ such that 
\[
  (g_1(c),g_1(d)) = (g''(c),g''(d)) = (g(c),g(d)).
\]
Repeating this argument with $g_1(x)$ proves the lemma.
\end{proof}

If $a,b\in B$ differ at a coordinate that is sequential, then the next lemma
proves that there is some polynomial that maps $(a,b)$ coordinatewise into
the monoliths of the $\C_l$ (the subdirect factors of $\B$) and does not
collapse $(a,b)$.

\begin{lem} \label{lem:seq produce c d}
There is a finite set of terms $P$ depending only on $\V(\A'(\T))$ such that
if $a,b\in B$ are distinct, $e_i(\overline{n},a) = e_i(\overline{n},b)$ for
all $i\in \{0,1,2\}$ and all $\overline{n}\in B^2\cup B$, and there is $p\in
B$ with $p\cdot a\neq p\cdot b$ or $a\cdot p\neq b\cdot p$, then there is
$t(\overline{y},x)\in P$ and $\overline{m}\in B^n$ with the property that if
$c = t(\overline{m},a)$ and $d = t(\overline{m},b)$ then
\begin{itemize}
  \item $c\neq d$,
  \item $x\cdot c = x\cdot d$,
  \item $c\cdot x = d\cdot x$,
  \item $I(c) = I(d)$,
  \item $F(x,y,c) = F(x,y,d)$,
  \item $F(x,c,y) = F(x,d,y)$, and
  \item $F(c,x,y) = F(d,x,y)$
\end{itemize}
for $F\in \mathcal{L}\cup \mathcal{R}$ and for all $x,y\in B$.
\end{lem}
\begin{proof}
Let $\B\leq \prod_{l\in L} \C_l$ be a subdirect representation of $\B$ by
subdirectly irreducible algebras. The subdirectly irreducible algebras in
$\V(\A'(\T))$ can be divided into two groups: either $\C_l\models
e_i(\overline{n},x)\approx x$ for some $i\in \{0,1,2\}$ and some
$\overline{n}\in C_l^2\cup C_l$ or $\C_l\models e_i(\overline{n},x)\approx
0$ for all $i\in \{0,1,2\}$. Since $e_i(\overline{n},a) =
e_i(\overline{n},b)$, the projections $a(l)$ and $b(l)$ must agree on all
factors that satisfy $\C_l\models e_i(\overline{n},x)\approx x$ for some $i$
and some $\overline{n}$, and can only possibly disagree on factors
satisfying $\C_l\models e_i(\overline{n},x)\approx 0$ for all $i$.

\begin{claim}
There is a finite number $N\in \NN$ such that for all $l\in
L$, if $\C_l\models e_1(n,x)\approx 0$ then
\[
  \C_l\models x_1\cdot x_2\cdots x_{N-1}\cdot x_N \approx 0.
\]
\end{claim}
\begin{claimproof}
First recall that since $\T$ halts, there are only finitely many subdirectly
irreducible algebras, all finite. Therefore if $\C_l$ does not model the
identity in the claim, there must exist nonzero elements $r,s\in C_l$ such
that $r\cdots r\cdot s = s$. Considering $\C_l$ as a quotient of a product
of subalgebras of $\A'(\T)$, this means that there is some coordinate of the
preimages (under the quotient map) of $r$ and $s$ such that $(r(i),s(i))\in
\{(1,C), (2,D)\}$. Therefore $e_1(r,s)\neq 0$, contradicting our assumption
that $\C_l\models e_1(n,x)\approx 0$. Let $S$ be a finite set containing a
representative of each isomorphism type of the subdirectly irreducible
algebras of $\V(\A'(T))$, and for $\C\in S$, let $n_{\C}\in \NN$ be minimal
such that $\C\models x_1\cdots x_{n_{\C}} \approx 0$. Taking $N = \max \{
n_{\C}\mid \C\in S \}$ completes the proof of the claim.
\end{claimproof}

Since the product $(\cdot)$ associates to the left, every polynomial of the
form $f(x) = y_1\cdots y_m \cdot x \cdot y_{m+1}\cdots y_M$ can be rewritten
as $f(x) = y_1\cdots y_m\cdot x\cdot z$, where $z = y_{m+1}\cdots y_M$. Let
\begin{multline*}
  P = \{ f(y_1,\ldots,y_M, x) = y_1\cdots y_M\cdot x, \\
    g(y_1,\ldots,y_M,x) = y_1\cdots y_{M-1}\cdot x \cdot y_M \mid 0\leq M <
    N \}.
\end{multline*}
Thus, there is a term $t(\overline{y},x)\in P$ and constants
$\overline{m}\in B^n$ such that $t(\overline{m},a) \neq t(\overline{m},b)$
and $x\cdot t(\overline{m},a) = x\cdot t(\overline{m},b)$ and
$t(\overline{m},a)\cdot x = t(\overline{m},b)\cdot x$ for all $x\in B$.
Furthermore, since $p\cdot a\neq p\cdot b$ or $a\cdot p\neq b\cdot p$, the
term $t$ is not the identity. Therefore $t(\overline{m},x)(l)\approx 0$ when
$\C_l$ is machine (recall that machine $\C_l$ model $x\cdot y \approx 0$).
Thus for all $x,y,z\in B$ and all $F\in
\mathcal{L}\cup\mathcal{R}$,
\[
  F(t(\overline{m},z),x,y) 
  = F(x,t(\overline{m},z),y) 
  = F(x,y,t(\overline{m},z))
  = I(t(\overline{m},x))
  = 0. \qedhere
\]
\end{proof}

Let the set $P$ be as in Lemma \ref{lem:seq produce c d} and define
\begin{equation} \label{defn:Gamma cdot}
  \Gamma_{(\cdot)}(w,x,y,z) = \bigvee_{t\in P\cup \{\id(x)\}} \exists
  \overline{n} \left[ w = t(\overline{n},y) \wedge x = t(\overline{n},z)
  \right].
\end{equation}
Given Lemmas \ref{lem:e_i=0 all polys=J'T} and \ref{lem:e_i=0 no maltsev
seq}, define 
\begin{equation} \label{defn:psi cdot}
  \psi_{(\cdot)}(w,x,y,z) = \exists t\left[ w = J'(w,t,y) 
  \wedge x = J'(w,t,z) \right].
\end{equation}
If $c,d\in B$ with $d<c$ satisfy the conclusion of Lemma \ref{lem:seq
produce c d}, then $c,d$ also satisfy the hypotheses of Lemma \ref{lem:e_i=0
all polys=J'T}. In this situation, $(r,s)\in \Cg^{\B}(c,d)$ if and only if
$\B\models \psi_{(\cdot)}(r,s,c,d)$. Since we will be employing a strategy
similar to the proof of Theorem \ref{thm:dpsc e_i factor}, where a general
Maltsev sequence is divided into $2$ strictly decreasing sequences, let
\begin{equation} \label{defn:psi 3}
  \psi_3(w,x,y,z) = \exists t\left[ \psi_{(\cdot)}(w,t,y,z) \wedge
  \psi_{(\cdot)}(x,t,y,z) \right].
\end{equation}

\begin{thm} \label{thm:seq dpsc}
Let $a,b\in B$ be distinct and such that $b\leq a$ and $e_i(\overline{n},a)
= e_i(\overline{n},b)$ for all $i\in \{0,1,2\}$ and all $\overline{n}\in
B^2\cup B$. If one of the following
\begin{enumerate}
  \item there is $p\in B$ such that $p\cdot a\neq p\cdot b$ or $a\cdot p
    \neq b\cdot p$, or
  \item for all $u,v\in B$ and $F\in \mathcal{L}\cup \mathcal{R}$ each of
    the translations $x\cdot u$, $u\cdot x$, $I(x)$, $F(u,v,x)$, $F(u,x,v)$,
    and $F(x,u,v)$ are constant for $x\in \{a,b\}$
\end{enumerate}
holds, then the congruence $\Cg^{\B}(a,b)$ has a principal subcongruence
witnessed by the formula $\Gamma_{(\cdot)}(-,-,a,b)$ and defined by the
formula $\psi_3$:
\[
  \B\models \exists c,d \left[ c\neq d \wedge \Gamma_{(\cdot)}(c,d,a,b) 
  \wedge \Pi_{\psi_3}(c,d) \right].
\]
\end{thm}
\begin{proof}
Let $\B = \prod_{l\in L} \C_l$ be a subdirect representation of $\B$ by
subdirectly irreducible algebras $\C_l$. If (1) holds, then from Lemma
\ref{lem:seq produce c d} the pair $(a,b)$ differs at a coordinate that is
sequential, and $(a,b)(l)$ lies outside of the monolith of some sequential
$\C_l$. Find $t\in P$ and constants $\overline{m}$ such that if $c =
t(\overline{m},a)$ and $d = t(\overline{m},b)$ then
\begin{itemize}
  \item $c\neq d$,
  \item $x\cdot c = x\cdot d$,
  \item $I(c) = I(d)$
  \item $c\cdot x = d\cdot x$,
  \item $F(x,y,c) = F(x,y,d)$,
  \item $F(x,c,y) = F(x,d,y)$, and
  \item $F(c,x,y) = F(d,x,y)$.
\end{itemize}
If (2) holds, then the pair $(a,b)$ differ at a coordinate that is 
sequential, but $(a,b)(l)$ lies in the monolith of each sequential
$\C_l$. Let $c = a$ and $d = b$. In both (1) and (2), $\B\models
\Gamma_{(\cdot)}(c,d,a,b)$ and $c$ and $d$ satisfy the hypotheses of Lemma
\ref{lem:e_i=0 all polys=J'T}.

Since $b\leq a$ and the operations of $\B$ are monotonic, $d\leq c$. Suppose
now that $(r,s)\in \Cg^{\B}(c,d)$.  Then using the same argument as in the
proof of Theorem \ref{thm:dpsc e_i factor}, there are decreasing Maltsev
chains $r = r_1, \ldots, r_m = t$ and $s = s_1, \ldots, s_n = t$ with
associated primitive polynomials. Using first Lemma \ref{lem:e_i=0 all
polys=J'T} and the description of $c$ and $d$ in the preceding paragraph,
and then applying Lemma \ref{lem:e_i=0 no maltsev seq}, we have that there
are constants $\rho$ and $\rho'$ such that
\begin{align*}
  r & = J'(r,\rho,c), & t & = J'(r,\rho,d) = J'(s,\rho',d), \\
  s & = J'(s, \rho',c).
\end{align*}
Hence $\B\models \psi_{(\cdot)}(r,t,c,d) \wedge \psi_{\cdot}(s,t,c,d) =
\psi_3(r,s,c,d)$,
completing the proof.
\end{proof}

Next, we move on to analyzing the case where $a, b\in B$ differ at a machine
coordinate. We will employ a strategy similar to the sequential case, and
produce from $(a,b)$ a pair $(c,d)$ such that $(c,d)(l)$ lies in the
monolith of $\C_l$ for each $l\in L$.

\begin{lem} \label{lem:machine produce c d}
There are finite sets of terms $S$ and $T$ depending only on $V(\A'(\T))$
such that if $a,b\in B$ are distinct such that $b\leq a$,
$e_i(\overline{n},a) = e_i(\overline{n},b)$ for all $i\in \{0,1,2\}$ and all
$\overline{n}\in B^2\cup B$, one of
\begin{enumerate}
  \item there is $F\in \mathcal{L}\cup \mathcal{R}$ and $u,v\in B$
    such that $F(u,v,a)\neq F(u,v,b)$, or
  \item there is $F\in \mathcal{L}\cup \mathcal{R}$ and $u,v\in B$
    such that $F(u,v,I(a))\neq F(u,v,I(b))$,
  \item for all $F\in \mathcal{L}\cup \mathcal{R}$ and all $u,v\in B$,
    \begin{enumerate} 
      \item $I(a) = I(b)$,
      \item $u\cdot a = u\cdot b$ and $a\cdot u = b\cdot u$, and
      \item $F(u,v,a) = F(u,v,b)$,
    \end{enumerate}
\end{enumerate}
holds, then there is $t(\overline{y},x)\in S$ and constants $\overline{m}\in
B$ such that if $c = t(\overline{m},a)$ and $d = t(\overline{m},b)$ then for
any $n\in \NN$ and $F_1,\ldots,F_n\in \mathcal{L}\cup \mathcal{R}$ and any
$a_1,\ldots, a_{2n}\in B$ there is $G(\overline{y},x)\in T$ and
$\overline{b}\in B^m$ such that
\begin{align*}
  F_1(a_1,a_2,F_2(a_3,a_4,\ldots F_n(a_{2n-1},a_{2n},c)\ldots ))
  & = G(\overline{b},c) & \text{ and} \\
  F_1(a_1,a_2,F_2(a_3,a_4,\ldots F_n(a_{2n-1},a_{2n},d)\ldots ))
  & = G(\overline{b},d).
\end{align*}
Furthermore, if $\B\leq \prod_{l\in L} \C_l$ is a subdirect representation
of $\B$ by subdirectly irreducible algebras then $(c(l),d(l))$ lies in the
monolith of $\C_l$ for each $l\in L$.
\end{lem}
\begin{proof}
We will begin by examining algebras whose only subdirect factors are
machine. If $\m{D}$ is a machine SI, then using the notation from the
discussion of large SI's in Section \ref{sec:SIs}, the monolith of $\m{D}$
is $\Cg^{\m{D}}(\mathcal{P},0)$, and there are two possibilities for its
structure: either
\begin{itemize}
  \item $\mathcal{T}(\mathcal{P}) = 0$, in which case the only nontrivial
    class of the monolith is $\{\mathcal{P},0\}$, or
  \item there is $N\in \NN$ such that $\mathcal{T}^N(\mathcal{P}) =
    \mathcal{T}(\mathcal{T}(\cdots \mathcal{T}(\mathcal{P})\cdots )) =
    \mathcal{P}$ (that is, the Turing machine enters a non-terminating loop),
    in which case the only nontrivial class of the monolith is $\{
    \mathcal{P}, \T(\mathcal{P}), \ldots, \T^{N-1}(\mathcal{P}),0\}$.
\end{itemize}

Let $(\C_k)_{k\in K}$ be a family of machine SI's and suppose that $\C\leq
\prod_{k\in K} \C_k$ and that $c,d\in C$ are such that $c\geq d$ and
$(c(k),d(k))$ lies in the monolith of $\C_k$ for each $k\in K$. We now make
two straightforward observations that follow from the description of the
monoliths of the machine SI's in the above paragraph and in Section
\ref{sec:SIs} and from $\V(\A'(\T))$ having finite residual bound: 
\begin{itemize}
  \item if $F\in \mathcal{L}\cup \mathcal{R}$, $k\in K$, and $a_1,a_2\in
    C_k$ then $F(a_1,a_2,c(k))\in \{0, \T(c(k))\}$ (likewise for $d$ in
    place of $c$), and

  \item the set $T' = \{ \T^i(d), \T^i(c) \mid 0\leq i < \infty \}$ is
    finite (we apply $\T$ coordinatewise).
\end{itemize}
We now define the set $T$. For each isomorphism type of a machine SI, $\D$,
let $N_{\D}$ be minimal such that $\T^{N_{\D}}(\mathcal{P}) \in \{0,
\mathcal{P} \}$, and let $N$ be the least common multiple of the $N_{\D}$.
Since $\V(\A'(\T))$ has finite residual bound, $N$ is finite. Define
\begin{multline*}
  T = \\
  \{ G_1(y_1,y_2,G_2(y_3,y_4,\ldots G_m(y_{2m-1},y_{2m},x)\ldots )) 
    \mid m\leq N, G_i\in \mathcal{L}\cup\mathcal{R} \} \cup \{\id(x)\}.
\end{multline*}
A consequence of these observations and the definition of $T$ is that for
any $n\in \NN$, $F_1,\ldots,F_n\in \mathcal{L}\cup \mathcal{R}$, and
$a_1,\ldots, a_{2n}\in C$, there is $G(\overline{y},x)\in T$ and
$\overline{b}\in C^m$ such that
\begin{align*}
  F_1(a_1,a_2,F_2(a_3,a_4,\ldots F_n(a_{2n-1},a_{2n},c)\ldots ))
  & = G(\overline{b},c) & \text{ and} \\
  F_1(a_1,a_2,F_2(a_3,a_4,\ldots F_n(a_{2n-1},a_{2n},d)\ldots ))
  & = G(\overline{b},d).
\end{align*} 
Next, we move on to the set $S$. $\V(\A'(\T))$ is residually finite, so
there is a finite set of terms $S$ depending only on $\V(\A'(\T))$ such that
for all $\C\leq \prod_{k\in K} \C_k$ (recall $(\C_k)_{k\in K}$ is a family
of machine SI's) and all $p,q\in C$ with $q < p$ there is a term $t\in S$
and constants $\overline{m}\in C^m$ such that $t(\overline{m},p)\neq
t(\overline{m},q)$ and $(t(\overline{m},p)(k),t(\overline{m},q)(k))$ lies in
the monolith of $\C_k$ for each $k\in K$. Note that the set of terms $S$ can
be taken to consist of the identity and a finite subset of terms generated
by composing operations from $\mathcal{L}\cup\mathcal{R}\cup \{ I(x) \}$. At
this point we have produced $S$ and $T$ that will work for algebras $\C$
whose subdirect factors are all machine.

We now examine algebras whose subdirect factors contain non-machine SI's.
Let $\B\leq \prod_{l\in L} \C_l$ be a subdirect representation of $\B$ by
subdirectly irreducible algebras. By the hypotheses, $a(l)\neq b(l)$ on some
machine $\C_l$. From the paragraph above, it follows that there is a term
$t\in S$ and constants $\overline{m}\in B^m$ such that
$(t(\overline{m},a),t(\overline{m},b))(l)$ lies in the monolith of $\C_l$
for all machine $\C_l$. Let $(c,d) = (t(\overline{m},a),t(\overline{m},b))$.
The term $t$ is (if hypotheses (1) or (2) hold) a composition of operations
from $\mathcal{L}\cup\mathcal{R}\cup \{ I(x) \}$ or (if hypothesis (3)
holds) the identity. Aside from the identity, terms from $S$ are constant in
SI's modeling $e_i(\overline{x},y)\approx 0$ except for machine SI's and the
$3$-element small SI $\{0,H,M_1^0\}$. Therefore $c(l) = d(l)$ on non-machine
$\C_l$ and $(c,d)(l)$ lies in the monolith of $\C_l$ for machine $\C_l$.
Applying the observations from above about the set $T$ now proves the lemma.
\end{proof}

Let the sets $S$ and $T$ be as in Lemma \ref{lem:machine produce c d} and
define
\begin{equation} \label{defn:Gamma T}
  \Gamma_{\T}(w,x,y,z) = \bigvee_{t\in S} \exists \overline{n} \left[ w =
  t(\overline{n},y) \wedge x = t(\overline{n},z)
  \right].
\end{equation}
Given Lemmas \ref{lem:e_i=0 all polys=J'T} and \ref{lem:e_i=0 no maltsev
seq}, define
\begin{equation} \label{defn:psi T}
  \psi_{\T}(w,x,y,z) = \exists t\left[ \bigvee_{G\in T} \exists \overline{b}
  \left[ w = J'(w,t,G(\overline{b},y)) \wedge x = J'(w,t,G(\overline{b},z))
  \right] \right].
\end{equation}
If $c,d\in B$ satisfy the conclusion of Lemma \ref{lem:machine produce c d},
then $c,d$ also satisfy the hypotheses of Lemma \ref{lem:e_i=0 all
polys=J'T}. Then $(r,s)\in \Cg^{\B}(c,d)$ if and only if $\B\models
\psi_{\T}(r,s,c,d)$. Since we will be employing a strategy similar to the
proof of Theorems \ref{thm:dpsc e_i factor} and \ref{thm:seq dpsc}, where a
Maltsev chain is broken into $2$ decreasing segments, let
\begin{equation} \label{defn:psi 4}
  \psi_4(w,x,y,z) = \exists t\left[ \psi_{\T}(w,t,y,z) \wedge
  \psi_{\T}(x,t,y,z) \right].
\end{equation}

\begin{thm} \label{thm:machine dpsc}
Let $a,b\in B$ be distinct and such that $b\leq a$ and $e_i(\overline{n},a)
= e_i(\overline{n},b)$ for all $i\in \{0,1,2\}$ and all $\overline{n}\in
B^2\cup B$. If
\begin{enumerate}
  \item there is $F\in \mathcal{L}\cup \mathcal{R}$ and $u,v\in B$
    such that $F(u,v,a)\neq F(u,v,b)$, or
  \item there is $F\in \mathcal{L}\cup \mathcal{R}$ and $u,v\in B$
    such that $F(u,v,I(a))\neq F(u,v,I(b))$,
  \item for all $F\in \mathcal{L}\cup \mathcal{R}$ and all $u,v\in B$,
    \begin{enumerate} 
      \item $I(a) = I(b)$,
      \item $u\cdot a = u\cdot b$ and $a\cdot u = b\cdot u$, and
      \item $F(u,v,a) = F(u,v,b)$,
    \end{enumerate}
\end{enumerate}
holds, then the congruence $\Cg^\B(a,b)$ has a principal subcongruence
witnessed by $\Gamma_{\T}(-,-,a,b)$ and defined by $\psi_4$. In symbols,
\[
  \B\models \exists c,d \left[ c\neq d \wedge \Gamma_{\T}(c,d,a,b) \wedge
  \Pi_{\psi_4}(c,d) \right].
\]
\end{thm}
\begin{proof}
By hypothesis, Lemma \ref{lem:machine produce c d} holds. Let $T$ and $S$ be
the finite sets of terms and $c,d\in B$ be the elements guaranteed by the
conclusion of Lemma \ref{lem:machine produce c d}. Then there is $t\in S$
and $\overline{m}\in B^n$ such that $c = t(\overline{m},a)$ and $d =
t(\overline{m},b)$. Furthermore, if
\begin{multline*}
  F(x)\in \{\text{id(x)}\} \cup \{ F_1(a_1,b_1,F_2(a_2,b_2,\cdots
  F_n(a_n,b_n,x)\cdots )) \mid \\
  n\in \NN, F_i\in \mathcal{L}\cup \mathcal{R}, \text{ and } a_i, b_i\in
    B\},
\end{multline*}
then there is $G\in T$ and $\overline{b}\in B^n$ such that $F(c) =
G(\overline{b},c)$ and $F(d) = G(\overline{b},d)$ (the existence of such
elements is the conclusion of Lemma \ref{lem:machine produce c d}). Thus
$\B\models \Gamma_{\T}(c,d,a,b)$ and $c$ and $d$ satisfy the hypotheses of
Lemma \ref{lem:e_i=0 all polys=J'T}.

Since $b\leq a$ and the operations of $\B$ are monotone, $d\leq c$. Suppose
now that $(r,s)\in \Cg^{\B}(c,d)$. Using the same argument as in the proof
of Theorem \ref{thm:dpsc e_i factor}, there are decreasing Maltsev chains $r
= r_1, \ldots, r_m = t$ and $s = s_1, \ldots, s_n = t$ with associated
primitive polynomials. Using first Lemma \ref{lem:e_i=0 all polys=J'T}, and
then Lemma \ref{lem:e_i=0 no maltsev seq}, we have that there are constants
$\rho,\rho'\in B$ such that
\begin{align*}
  r & = J'(r,\rho,G(\overline{b},c)), & t 
    & = J'(r,\rho,G(\overline{b},c)) = J'(s,\rho',G'(\overline{b}',d)), \\
  s & = J'(s, \rho',G'(\overline{b}',c))
\end{align*}
for some $G,G'\in T$ and constants $\overline{b},\overline{b}'\in B^n$.
Hence $\B\models \psi_{\T}(r,t,c,d) \wedge \psi_{\T}(s,t,c,d)$, completing
the proof.
\end{proof}

The last case where $a, b\in B$ differ at a coordinate that is small but
that does not satisfy $\exists \overline{n} [ e_i(\overline{n},x)\approx x
]$ remains. From Lemma \ref{lem:small SI's e_i=0}, we know that there are
only $3$ isomorphism types for such SI's. If the coordinate is isomorphic to
$\{0,C\}$, then the lemmas used in the sequential case apply. We are
therefore concerned with the remaining two isomorphism types. To this end,
let
\begin{multline} \label{defn:Gamma I}
  \Gamma_I(w,x,y,z) = \\
  \exists u,v \left[ 
    \left( u = I(y) \wedge v = I(z) \wedge \Gamma_{(\cdot)}(w,x,u,v) \right)
    \vee \Gamma_{(\cdot)}(w,x,y,z) \right].
\end{multline}

\begin{lem} \label{lem:small dpsc}
Suppose that $a,b\in B$ are distinct, but that $I(x)$ is the only
fundamental operation that distinguishes them. Then the congruence
$\Cg^{\B}(a,b)$ has a principal subcongruence witnessed by
$\Gamma_I(-,-,a,b)$ and defined by $\psi_3$ (see \eqref{defn:psi 3} and
\eqref{defn:Gamma I}):
\[
  \B \models \exists c,d \left[ c\neq d \wedge \Gamma_I(c,d,a,b) \wedge
  \Pi_{\psi_3}(c,d) \right].
\]
\end{lem}
\begin{proof}
Let $\B\leq \prod_{l\in L} \C_l$ be a subdirect representation of $\B$ by
subdirectly irreducible algebras. If $a,b\in B$ are distinct and only
distinguished by $I(x)$, then $a(l)\neq b(l)$ if and only if $\C_l\cong \W$
or $\C_l\cong \{0,H,M_1^0\}$ (see \eqref{defn:W} and Lemma \ref{lem:small
SI's e_i=0}).

If $a' = I(a)$ and $b' = I(b)$, then $a'$ and $b'$ satisfy the hypotheses of
Theorem \ref{thm:seq dpsc} and thus $\Cg^{\B}(a',b')$ has a principal
subcongruence witnessed by $\Gamma_{(\cdot)}(-,-,a',b')$ and defined by
$\psi_3$. Therefore $\Cg^{\B}(a,b)$ has a principal subcongruence witnessed
by $\Gamma_I(-,-,a,b)$ and defined by $\psi_3$, as claimed.
\end{proof}

\begin{thm} \label{thm:V(A(T)) dpsc}
If $\T$ halts then $\V(\A'(\T))$ has definable principal subcongruences.
\end{thm}
\begin{proof}
Let
\[
  \Gamma(w,x,y,z) = \Gamma_1(w,x,y,z) \vee \Gamma_{(\cdot)}(w,x,y,z) \vee
  \Gamma_{\T}(w,x,y,z) \vee \Gamma_I(w,x,y,z)
\]
(see Theorem \ref{thm:dpsc e_i factor} and equations \eqref{defn:Gamma
cdot}, \eqref{defn:Gamma T}, and \eqref{defn:Gamma I} for definitions of
these), and
\[
  \psi(w,x,y,z) = \psi_2(w,x,y,z) \vee \psi_3(w,x,y,z) \vee \psi_4(w,x,y,z)
\]
(see equations \eqref{defn:psi 2}, \eqref{defn:psi 3}, and \eqref{defn:psi
4} for definitions of these). We claim that $\V(\A'(\T))$ has definable
principal congruences witnessed by $\Gamma$ and $\psi$. In symbols,
\[
  \V(\A'(\T)) \models \forall a,b \left[ a\neq b \rightarrow \exists c,d
  \left[ c\neq d \wedge \Gamma(c,d,a,b) \wedge \Pi_{\psi}(c,d) \right]
  \right].
\]

Let $\B\in\V(\A'(\T))$ with $a,b\in B$ distinct and let $\B\leq \prod_{l\in
L} \C_l$ be a subdirect representation by subdirectly irreducible algebras.
Since $a$ and $b$ are distinct, there is some $l\in L$ such that $a(l)\neq
b(l)$. Let
\[
  K = \{ l\in L \mid a(l)\neq b(l) \}.
\]
The case distinction breaks down as follows:
\begin{enumerate}
  \item There is some $k\in K$ such that $\C_k\models
    e_i(\overline{n},x)\approx x$ for some $i\in \{0,1,2\}$ and some
    $\overline{n}\in C_k^2\cup C_k$. In this case, Theorem \ref{thm:dpsc e_i
    factor} applies.
  \item The previous case does not apply, but there is some $k\in K$ such
    that $\C_k$ is sequential. If this is the case, there is some $u\in B$
    such that $u\cdot a\neq u\cdot b$ or $a\cdot u\neq b\cdot u$, or the
    machine operations $\mathcal{L}\cup \mathcal{R}$ cannot distinguish
    between $a$ and $b$. In this case, Theorem \ref{thm:seq dpsc} applies.
  \item The previous cases do not apply, but there is some $k\in K$ such
    that $\C_k$ is machine. If this is the case, there is some machine
    operation in $\mathcal{L}\cup \mathcal{R}$ that can distinguish between
    $a$ and $b$. In this case, Theorem \ref{thm:machine dpsc} applies.
  \item The previous cases do not apply, so there must be some $k\in K$ such
    that $\C_k$ is small and models $e_i(\overline{y},x)\approx 0$ for all
    $i\in \{0,1,2\}$ (see Lemma \ref{lem:small SI's e_i=0}). If $C_k =
    \{0,C\}$ or $\C\cong \W$, then Theorem \ref{thm:seq dpsc} applies. If
    $C_k = \{0,H,M_1^0\}$ then either Theorem \ref{thm:seq dpsc} applies (if
    $a(k) = M_1^0$) or Lemma \ref{lem:small dpsc} applies (if $a(k) = H$).
\end{enumerate}

Since the SI's of $\V(\A'(\T))$ either satisfy $e_i(\overline{n},x)\approx
x$ for some $i\in \{0,1,2\}$ and $\overline{n}\in B^2\cup B$, are
sequential, are machine, or are isomorphic to one of the $3$ small algebras
given in Lemma \ref{lem:small SI's e_i=0}, this completes the proof.
\end{proof}

One of the interesting applications of definable principal subcongruences is
in defining the subdirectly irreducible members of some class of algebras.
If $\mathcal{C}$ is a class of algebras with definable principal
subcongruences witnessed by congruence formulas $\Gamma$ and $\psi$, then
\[
  \mathcal{C} \models \forall a,b \left[ a\neq b \rightarrow \exists c,d
    \left[ c\neq d \wedge \Gamma(c,d,a,b) \wedge \Pi_\psi(c,d) \right]
    \right],
\]
and the sentence
\[
  \sigma = \exists r,s \left[ r\neq s \wedge \forall a,b 
    \left[ a\neq b \rightarrow \exists c,d \left[ \Gamma(c,d,a,b) 
    \wedge \psi(r,s,c,d) \right] \right] \right]
\]
defines the subdirectly irreducible algebras in $\mathcal{C}$. Baker and
Wang \cite{BakerWangDPSC} use this to prove the following theorem.

\begin{thm}[Baker, Wang \cite{BakerWangDPSC}]
\label{thm:Baker Wang DPSC fin based}
A variety $\V$ with definable principal subcongruences is finitely based
if and only if the class of subdirectly irreducible members of $\V$ is
finitely axiomatizable.
\end{thm}

\noindent In particular, if $\kappa(\V)<\omega$ then the class of
subdirectly irreducible members of $\V$ is finitely axiomatizable since
there are only finitely many of them, all finite. This observation and the
above theorem yields a corollary to Theorem \ref{thm:V(A(T)) dpsc}.

\begin{cor} \label{cor:V(A(T)) fin based}
If $\T$ halts, then $\V(\A'(\T))$ is finitely based.
\end{cor}

\section{If $\T$ does not halt} \label{sec:T does not halt}
In the case where $\T$ halts, every sequential subdirectly irreducible
algebra is finite and there are only finitely many of them. In the case
where $\T$ does not halt, McKenzie \cite{McKenzieResidualBoundNotComp} and
the additions from Section \ref{sec:modifying mckenzies argument} show that
the algebra $\S_{\ZZ}$ (defined in Section \ref{sec:SIs}) is a member of
$\V(\A'(\T))$. McKenzie \cite{McKenzieTarskisFiniteBasis} uses $\S_{\ZZ}$ to
show that if $\T$ does not halt, then $\A(\T)$ is inherently nonfinitely
based. Although $\S_{\ZZ}$ is not subdirectly irreducible, it contains an
infinite subalgebra $\S_{\omega}$ which is, and every finite sequentiable SI
can be embedded in it. We will use the presence of $\S_{\ZZ}$ in
$\V(\A'(T))$ to show that if $\T$ does not halt, then $\V(\A'(\T))$ doesn't
have DPSC.

An algebra $\C$ is said to be \emph{finitely subdirectly irreducible (FSI)}
if for all $a,b,c,d\in C$ such that $a\neq b$ and $c\neq d$,
$\Cg^{\C}(a,b)\cap \Cg^{\C}(c,d) \neq \0$ (i.e. $\0$ is meet irreducible).
Every SI is FSI, but not every FSI is SI.

\begin{thm} \label{thm:0 mi not defn}
The class of finitely subdirectly irreducible algebras in $\V(\A'(\T))$ is
not axiomatizable if $\T$ does not halt.
\end{thm}
\begin{proof}
We will use an ultrapower argument. Suppose to the contrary that the class
of finitely subdirectly irreducible algebras in $\V(\A'(\T))$ is
axiomatizable, say by $\Phi$. $\T$ does not halt if and only if $\S_\ZZ\in
\V(\A'(\T))$. Let $\S$ be an ultrapower of $\S_\omega$, so that $\S$
satisfies all first-order properties of $\S_\omega$. In particular, since
$\S_\omega\models \Phi$, we have that $\S\models \Phi$, so $\0$ is meet
irreducible in $\text{Con}(\S)$. We will now give some first-order
properties of $\S_{\omega}$ which we will make use of.

Let
\[
  A = \{ \alpha\in S_{\omega} \mid \exists \beta [ \alpha\cdot \beta\neq 0 ] \}
  \qquad \text{and} \qquad
  B = \{ \beta\in S_{\omega} \mid \exists \alpha [ \alpha\cdot \beta\neq 0 ] \}.
\]
Then in $\S_{\omega}$, for each $\alpha\in A$ there is a unique $\beta\in B$
such that $\alpha\cdot \beta\neq 0$, and for each $\beta\in B$, there is a
unique $\alpha\in A$ such that $\alpha\cdot\beta\neq 0$. This gives us that
$|A| = |B|$. We also have
\[
  A\cap B = \emptyset
  \qquad \text{and} \qquad
  S_{\omega} = A \cup B \cup \{0\}.
\]
For $b\in B$, let
\[
  A^n\cdot b = \{ \alpha_1 \cdots \alpha_m\cdot b\mid 0\leq m\leq n \text{
  and } \alpha_1,\ldots, \alpha_m \in A \}.
\]
Then $|A^n\cdot b| = n+2$. Furthermore, for $b,c\in B$,
\[
  \text{if } (A^n\cdot b) \cap (A^m\cdot c) \neq \{0\} 
  \qquad \text{then} \qquad
  b\in A^m\cdot c \text{ or } c\in A^n\cdot b.
\]
Lastly, if $F(x)$ is a fundamental translation in $\S_{\omega}$, then
$F(A^n\cdot b)\subseteq A^{n+1}\cdot b$. All of these sets and properties
are first-order definable and hold in $\S_{\omega}$, so their analogues hold
in $\S$ as well.

We will now begin to examine $\S$. For $b\in B$, define the \emph{orbit of
$b$} to be
\[
  b^A = \bigcup_{n\in \NN} A^n\cdot b.
\]
Since $|A^n\cdot b| = n+2$, the set $b^A$ is countable. Suppose now that
there are $b,c\in B$ such that $b^A\cap c^A = \{0\}$. Then by the properties
above, $\Cg^\S(b,0)$ relates the orbit of $b$ to $0$ and is the identity
relation elsewhere. A similar statement is true of $\Cg^\S(c,0)$. It follows
that the two congruences meet to $\0$, which contradicts $\0$ being meet
irreducible in $\Con(\S)$. It follows that for all $b,c\in B$, $b^A\cap c^A
\neq \{0\}$.

Pick distinct $b,c\in B$. Then $b^A\cap c^A \neq \{0\}$, so by the
properties above, we have that either $b\in c^A$ or $c\in b^A$. Without loss
of generality, assume that $b\in c^A$. There is a finite number $n$ and
$\alpha_1,\ldots,\alpha_n\in A$ such that $\alpha_1\cdots\alpha_n\cdot c =
b$, so since this is true for all $b,c$, we have that $\bigcup_{b\in B} b^A$
is countable.  Since $B = \bigcup_{b\in B} b^A$, it must be that $B$ is
countable. The property that $|A| = |B|$ and $S = A\cup B\cup \{0\}$
therefore gives us that $\S$ is also countable. Since nonprincipal
ultrapowers of infinite structures are uncountable, this implies that the
ultrapower is principal and $\S\cong \S_{\omega}$.
\end{proof}

\begin{cor} \label{cor:no DPSC}
$\V(\A'(\T))$ does not have definable principal subcongruences if $\T$ does
not halt.
\end{cor}
\begin{proof}
Suppose that $\V(\A'(\T))$ has definable principal subcongruences witnessed
by $\Gamma$ and $\psi$, and let
\begin{multline*}
  \zeta = \forall a,b,a',b' \left[(a\neq b) 
    \wedge (a'\neq b') \rightarrow \exists c,d,c',d' 
    \left[ \Gamma(c,d,a,b) \wedge \Gamma(c',d',a',b') \right. \right. \\
  \left. \left. \wedge \exists r,s \left[ r\neq s \wedge \psi(r,s,c,d) 
    \wedge \psi(r,s,c',d') \right] \right] \right].
\end{multline*}
For $\B\in \V(\A'(\T))$, we have that $\B\models \zeta$ if and only if $\B$
is finitely subdirectly irreducible (that is, FSI's are axiomatized by
$\zeta$). This contradicts Theorem \ref{thm:0 mi not defn}, so $\V(\A'(\T))$
cannot have definable principal subcongruences as we assumed.
\end{proof}

\section{Conclusion} \label{sec:conclusion}
Theorem \ref{thm:V(A(T)) dpsc} and Corollaries \ref{cor:V(A(T)) fin based}
and \ref{cor:no DPSC} yield the next theorem.

\begin{thm}
The following are equivalent.
\begin{enumerate}
  \item $\T$ halts.
  \item $\V(\A'(\T))$ has definable principal subcongruences.
  \item $\V(\A'(\T))$ is finitely based.
\end{enumerate}
\end{thm}

\noindent This completes the proof that DPSC is an undecidable property, and
provides another negative answer to Tarski's well-known finite basis
problem.

\bibliographystyle{amsplain}
\bibliography{references}
\begin{center}
  \rule{0.6180\textwidth}{0.1ex}
\end{center}
\end{document}